\documentclass[12pt]{amsart}
\usepackage[all]{xy}
\usepackage{amssymb}
\usepackage{tikz}
\usepackage{graphics}
\usepackage{graphicx}
\graphicspath{ {images/} }
\usepackage{epstopdf}
\calclayout                   
\newtheorem{dummy}{anything}[section]
\newtheorem{theorem}[dummy]{Theorem}

\newtheorem{lemma}[dummy]{Lemma}

\newtheorem{corollary}[dummy]{Corollary}

\theoremstyle{definition}
\newtheorem{definition}[dummy]{Definition}

\newtheorem{example}[dummy]{Example}

\newtheorem{remark}[dummy]{Remark}

\renewcommand{\labelenumi}{(\roman{enumi})}
\renewcommand
{\theequation}{\thesection.\arabic{dummy}}
\newcommand
{\eqncount}{\setcounter{equation}{\value{dummy}}%
\addtocounter{dummy}{1}}
\newcommand{\cH}{\mathcal H}
\newcommand{\cM}{\mathcal M}
\newcommand{\cP}{\mathcal P}
\newcommand{\cS}{\mathcal S}
\newcommand{\cD}{\mathcal D}
\newcommand{\cE}{\mathcal E}
\newcommand{\cT}{\mathcal T}
\newcommand{\bF}{\mathbf F}
\newcommand{\bH}{\mathbf H}
\newcommand{\bZ}{\mathbb Z}
\newcommand{\bQ}{\mathbf Q}
\newcommand{\bC}{\mathbb C}
\newcommand{\bR}{\mathbb R}
\newcommand{\fu}{\mathfrak u}
\newcommand{\fv}{\mathfrak v}
\newcommand{\fw}{\mathfrak w}
\newcommand{\fo}{\mathfrak o}
\newcommand{\cy}[1]{\bZ/{#1}}
\newcommand{\Zhat}{\widehat \bZ}
\newcommand{\Qhat}{\widehat \bQ}
\newcommand{\wX}{\widetilde X}
\newcommand{\wK}{\widetilde K}
\newcommand{\bd}{\partial}
\newcommand{\vv}{\, | \,}
\newcommand{\trf}{tr{\hskip -1.8truept}f}
\newcommand{\ZG}{\bZ G}
\newcommand{\QG}{\bQ G}
\newcommand{\La}{\Lambda}
\newcommand{\Zpi}{\bZ \pi}
\newcommand{\Ga}{\Gamma}
\newcommand{\ra}{\rightarrow}
\renewcommand{\restriction}{\mathord{\upharpoonright}} 
\def\:{\mkern 1.2mu \colon}
\newcommand{\mmatrix}[4]{\left (\vcenter
{\xymatrix@C-2pc@R-2pc{#1&#2\\#3&#4} } \right )}
\newcommand{\mysquare}[4]{\vcenter
{\xymatrix
{#1\ar[r]\ar[d]&#2\ar[d]\\#3\ar[r]&#4} }}
\newcommand{\wbar}[1]{\overset{\, \hrulefill\, }{#1}}
\newcommand{\disjointunion}{\hbox{$\perp\hskip -4pt\perp$}}
\DeclareMathOperator{\Hom}{Hom} \DeclareMathOperator{\wh}{Wh}
\DeclareMathOperator{\Mod}{\,mod} \DeclareMathOperator{\rank}{rank}
\DeclareMathOperator{\Sharp}{\sharp} \DeclareMathOperator{\im}{im}
\DeclareMathOperator{\Isom}{Isom}
\DeclareMathOperator{\colim}{colim}
\DeclareMathOperator{\Homeo}{Homeo}
\DeclareMathOperator{\Diff}{Diffeo}
\DeclareMathOperator{\hepta}{Aut}
\DeclareMathOperator{\projdim}{projdim}
\DeclareMathOperator{\id}{id}\DeclareMathOperator{\coker}{coker}
\DeclareMathOperator{\Her}{Her} \DeclareMathOperator{\cd}{cd}
\DeclareMathOperator{\vcd}{vcd}  \DeclareMathOperator{\ev}{ev}
\DeclareMathOperator{\pr}{pr}   \DeclareMathOperator{\p}{p}
\DeclareMathOperator{\Lk}{Lk} 

\newcommand{\homeo}[1]{\Homeo(#1)}
\newcommand{\homeopt}[1]{\Homeo_{\bullet} (#1)}
\newcommand{\diff}[1]{\Diff(#1)}
\newcommand{\diffpt}[1]{\Diff_{\bullet} (#1)}
\newcommand{\he}[1]{\hepta(#1)}
\newcommand{\hept}[1]{\hepta_{\bullet}(#1)}
\newcommand{\heq}[1]{\cE (#1)}
\newcommand{\heqpt}[1]{\cE_{\bullet} (#1)}
\newcommand{\htildeM}{\widetilde \cH (M)}
\newcommand{\htildeB}{\widetilde \cH (B)}
\newcommand{\wB}{\widetilde B}
\newcommand{\hM}{\cH (M)}
\newcommand{\hB}{\cH (B)}
\newcommand{\quadtypeMb}{[\pi, \pi_2, k_M, s_M]}
\newcommand{\quadtypeMw}{[\pi, \pi_2, s_M, w_2]}
\newcommand{\quadtypeM}{[\pi, \pi_2, s_M]}
\newcommand{\quadtypecM}{[\pi, \pi_2, c_*[M]]}
\newcommand{\quadtypecMw}{[\pi, \pi_2, c_*[M], w_2]}
\newcommand{\Ospin}{\Omega^{Spin}}
\newcommand{\rOspin}{\widehat\Omega^{Spin}}
\newcommand{\Ostop}{\Omega^{SO}}
\newcommand{\Kw}{KH_2(M;\cy 2)}
\newcommand{\Bw}{B\langle w_2\rangle}
\newcommand{\Mw}{M\negthinspace\langle w_2\rangle}
\newcommand{\whept}[1]{\hepta_{\bullet}(#1,w_2)}
\newcommand{\whtildeM}{\widetilde \cH (M,w_2)}
\newcommand{\whtildeB}{\widetilde \cH (B,w_2)}

\begin{document}
\title[Perfect Discrete Morse Functions on Connected Sums]
{Perfect Discrete Morse Functions on Connected Sums}

\subjclass[2010]{Primary:57R70, 37E35; Secondary:57R05}
\keywords{Perfect discrete Morse function, discrete vector field, connected sum.}

\author{Ne\v{z}a Mramor Kosta, Mehmetc{\.I}k Pamuk and Han{\.I}fe Varl{\i}}

\address{Faculty of Computer and Information Science 
\newline\indent
and Institute of Mathematics, Physics and Mechanics, 
\newline\indent
University of Ljubljana
\newline\indent
Ljubljana, Jadranska 19, Slovenia} \email{neza.mramor@fri.uni-lj.si}

\address{Department of Mathematics
\newline\indent
Middle East Technical University
\newline\indent
Ankara 06531, Turkey} \email{mpamuk{@}metu.edu.tr}

\address{Department of Mathematics
\newline\indent
Middle East Technical University
\newline\indent
Ankara 06531, Turkey}  \email{hisal@metu.edu.tr}

\date{\today}

\begin{abstract}\noindent
We study perfect discrete Morse functions on closed oriented $n$-dimensional manifolds.  We show how to compose such functions on 
connected sums of closed oriented manifolds and how to decompose them on connected sums of closed oriented surfaces.
\end{abstract}
\maketitle

\section{INTRODUCTION}

Since it was introduced by Marston Morse in the $1920$s, Morse theory has been a powerful tool in the study of smooth manifolds.  
It allows to describe 
the topology of a manifold in terms of the cellular decomposition generated by the critical points of a smooth map defined on it.  
By analyzing the function's critical points, it is possible to construct a cell structure for the manifold.  

In the $1990$s  Robin Forman developed a discrete version of Morse theory that turned out to be an efficient method
for the study of the topology of discrete objects, such as regular cell complexes.  As in the smooth setting,
changes in the topology are deeply related to the presence of critical cells of a discrete Morse function.   The analysis of the evolution of the homology of
the cell complexes can be a very useful tool, for example in computer vision to deal with shape recognition
problems by means of topological shape descriptors, and in topological data analysis, where new information can be extracted from the data. 
It can also be used for efficient computation of homology of the cell complex \cite{forman1},\cite{Mischaikow}. Many of the familiar results 
from the smooth theory apply in the discrete setting.  

A discrete Morse function on a cell complex is an assignment of a real number to each cell in such a way that the natural order given by the dimension of cells 
is respected, except at most in one (co)face for each cell. A discrete Morse function on a regular cell complex induces a partial 
pairing on the cells called a discrete vector field. It consists of pairs of cells of two consecutive dimensions 
on which the function reverses the order given by the dimension.  Those cells that do not belong to 
any pair are precisely the critical cells of the map. A gradient path  is 
a connected sequence of the pairs on which the function is non-increasing.  The idea of discrete gradient path 
plays a central role in Forman's theory, due to the fact that the properties related to the critical 
elements are more easily visualized.  Moreover, the discrete gradient paths give us an 
easy and efficient method for getting examples of such functions verifying certain initial conditions (for example, 
having a given number of critical cells) avoiding rather complicated numerical assignments.  On the other hand, since 
different discrete Morse functions may induce the same discrete gradient field, working with a vector field instead 
of a Morse function may lead to some disambiguities about values of a function on cells.  We will point out this fact 
again later in Section $3$.

The discrete Morse functions that have as few critical cells as possible have been widely studied in the literature 
(see \cite{hersh}, \cite{lewiner}).  The number of critical $n$-cells of a discrete Morse function is greater than 
or equal to the $n$-th Betti number of the cell complex by Morse inequalities \cite[Corollary 3.7]{forman1}.   
For a perfect discrete Morse function these two quantities are the same.  Although there are complexes which do not admit 
perfect discrete Morse functions, these functions are the most suitable for combinatorial and computational purposes and they have been studied 
greatly.  

In this paper we study perfect discrete Morse functions on connected sums of triangulated manifolds and consider the problem of (de)composing 
such maps as perfect discrete Morse functions on the summands. We show that a perfect discrete Morse function on a connected sum of triangulated manifolds of any dimension is obtained by 
combining perfect discrete Morse functions on the summands, after possibly minor subdivisons in a neighborhood of the connecting sphere (cf. Theorem \ref{compose}). On decomposing a given perfect discrete Morse function, we consider only the $2$-dimensional case (cf. Theorem \ref{decompose}). For a triangulated connected sum 
$M_1\# M_2$ of two closed connected surfaces with a given perfect discrete Morse function  we give an explicit construction of a separating circle $C$ such that 
the given perfect discrete vector field restricts to a perfect discrete vector field on each summand. Such a construction can be implemented in the form of an algorithm which could be useful in applications of topological methods for example to data, since it enables a subdivision of the data analyzed into smaller, simpler parts. As we will show in a further paper, our proof generalizes to the $3$-dimensional case, but in higher dimensions new phenomena appear. 

The paper is organized as follows:  In section $2$, we recall necessary basic notions of discrete Morse theory for our proof.  In 
section $3$, we prove how to compose a perfect discrete Morse function on a connected sum of triangulated $n$-dimensional manifolds.  
In section $4$, we show how to decompose a perfect discrete Morse function on surfaces and give an example about decomposing a perfect 
discrete Morse function on genus $2$ surface.

\vskip 0.2cm
\noindent{\bf Acknowledgements.} 
This research was supported by the Slovenian-Turkish grants BI-TR/12-14-001 and 111T667, and by the ESF Research Network Applied and Computational Algebraic Topology (ACAT).


\section{Preliminaries}
In this section we recall necessary basic notions of discrete Morse theory.  For more details we refer the reader to \cite{forman1} and \cite{forman2}.
Throughout this section $K$ denotes a finite regular cell complex.  We write $\tau > \sigma$ if $\sigma \subset \overline{\tau}$ (closure of $\tau$). 
 
A discrete function $f\colon K\rightarrow \bR$ on $K$ associates values to the cells of $K$ such that for any $p$-cell
$\sigma \in K$ each $(p-1)-$face $\sigma<\tau$ except at most one has value $f(\sigma)<f(\tau)$, and each $(p+1)-$coface $\nu>\tau$ except at most one 
has value $f(\nu)>f(\tau)$. The cells of $K$ are subdivided into critical and regular cells, where the critical cells are cells where none of the above 
exceptions occur, and regular cells appear in disjoint pairs which form the gradient vector field 
$$
V = \{(\sigma, \tau) \mid \dim \sigma=\dim \tau-1, \sigma<\tau, f(\sigma)\geq f(\tau).
$$

We draw arrows to represent the vector field as follows:
If $\tau^{(p+1)} > \sigma^{(p)}$ and $f(\sigma) \geq f(\tau)$ then we draw an arrow from $\sigma$ to $\tau$ as in Figure \ref{fig:1} .  

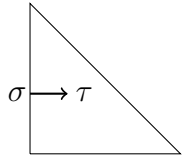
\begin{figure}[hbt]
\centering
\begin{tikzpicture}
\draw (0,0) -- (0,2) -- (2,0)-- (0,0);
\draw[thick,->] (0,0.8) -- (0.5,0.8);
\node[inner sep=0,anchor=west,text width=3.3cm] (note1) at (0.6,0.8) {$\tau$};
\node[inner sep=0,anchor=west,text width=3.3cm] (note1) at (-0.3,0.8) {$\sigma$};
\end{tikzpicture}
\caption{$f(\sigma) \geq f(\tau)$.} 
\label{fig:1}
\end{figure}
\noindent

A cell is critical if and only if it is neither the tail nor the head of an arrow.  
 
Critical cells are related with the topology of the complex as in the case of a smooth manifold with a smooth Morse function, that is, a regular cell complex $K$ 
with a discrete Morse function $f$ has the homotopy type of a CW complex with cells that correspond to the critical cells of the function, and the number of critical 
cells is bounded by the Betti numbers of $K$ in terms of Morse inequalities :
 
\begin{enumerate}
\item $m_{p}(f)-m_{p-1}(f)+ \ldots \pm m_{0} \geq b_{p}-b_{p-1}+ \ldots \pm b_{0}$,
\item $m_{p}(f) \geq b_{p}$,
\item $m_{0}(f)-m_{1}(f)+ \ldots +(-1)^{n}m_{n}(f)=b_{0}-b_{1}+ \ldots +(-1)^{n}b_{n}$.
\end{enumerate}
where $b_{p}$ is the $p$-th Betti number of $K$, and $m_{p}(f)$ is the number of critical $p$-cells of $f$  
for $p=0, 1, \ldots, n $ where $n$ is the dimension of $K$.  

Pairs of regular cells connect into gradient paths along which function values of $f$ descend.  A gradient path or a $V$-path of dimension $(p+1)$ is a sequence of cells 
$$
\sigma_{0}^{(p)}\to \tau_{0}^{(p+1)}>\sigma_{1}^{(p)}\to \tau_{1}^{(p+1)}>\cdots \to \tau_{r}^{(p+1)}>\sigma_{r+1}^{(p)}
$$  
such that $(\sigma_{i}^{(p)}<\tau_{i}^{(p+1)})\in V$ and $\sigma_{i}^{(p)} \neq \sigma_{i+1}^{(p)} < \tau_{i}^{(p+1)}$ for each $i = 0, 1, \ldots,  r$.  

Gradient paths are represented by a sequnce of arrows. Since function values along a gradient path of length more than $2$ must decrese, gradient paths cannot form cycles.  

In discrete Morse theory, gradient vector fields are often more useful than the underlying discrete Morse functions for the combinatorial purposes. 
Let us define an equivalence relation between discrete Morse functions.

\begin{definition}{\label{equal}}
Two discrete Morse function $f$ and $g$ are called equivalent if for every pair of cells $(\tau^{(p)}<\sigma^{(p+1)})$ in $K$,
$$
f(\sigma)<f(\tau) \quad  \text{if and only if} \quad  g(\sigma)<g(\tau).
$$
\end{definition}

The following theorem gives us ground to work with the gradient vector fields instead of the function values.

\begin{theorem}(\cite{afv2}){\label{equivalent}}
Two discrete Morse function $f$ and $g$ defined on a simplicial complex $K$ are equivalent if and only if $f$ and $g$ have the same critical cells 
and induce the same gradient vector field.
\end{theorem}

In this paper we study a special type of discrete Morse functions,  where the Morse inequalities are actually equalities. As in the smooth case, 
such Morse functions are called perfect (with respect to the given coefficient ring, which is fixed throughout most of this paper and suppressed in our notation).

\begin{definition}(\cite{ayala}){\label{perfect}}
A discrete Morse function $f\colon K \to \bR$ and its discrete gradient field are called perfect if for each $p$,
$$ 
m_{p}(f) = b_{p} = \rank H_{p}(K).
$$ 
\end{definition}

Perfect discrete Morse functions have the least number of critical cells by means of the Morse inequialities.  Therefore, they are the most 
suitable functions for combinatorial and computational purposes.  Existence of such functions has both theoretical and practical applications.  An arbitrary 
CW-complex may not have a perfect discrete Morse function defined on it. For example, any torsion in homology is an obstruction to a $\mathbb{Z}$-perfect discrete Morse function and any complex that is acyclic (homologically trivial) and non-collapsible e.g. the dunce hat and Bing's House cannot admit a perfect discrete Morse function. Also, every sphere of dimension $d>4$ has a triangulation which does not admit a perfect discrete Morse function \cite{Benedetti}.  On the other hand, it is easy to see that every 1-dimensional complex (i.e. graph) has a perfect discrete Morse function, and in dimension $2$ it is known that every closed, oriented surface has a $\mathbb{Z}$-perfect discrete Morse function \cite{lewiner}, every closed surface has a perfect $\mathbb{Z}_2$-Morse function \cite{ayala}, and every $2$-dimensional subcomplex of a $2$-manifold has a $\mathbb{Z}_2$-perfect discrete Morse function \cite{HKN}. Spheres are characterized by the existence of a perfect discrete Morse
  function in the sense that every triangulated manifold with exactly two critical cells is a sphere, and for every sphere there exists a triangulation with exactly two critical cells \cite{forman1}. Finding perfect discrete Morse functions is a difficult problem, as Joswig and Pfetsch \cite{JP} have shown it is NP-hard even on 2-dimensional complexes. 

\begin{figure}[h]
\centering
\begin{tikzpicture}
\draw (0,0)  -- (1,1) -- (2,0)-- (0,0);
\draw (0.5,0.5) node[circle,fill,inner sep=1pt] {};
\draw[thick,->] (2.2,0.5) -- (2.7,0.5);
\end{tikzpicture}
\qquad
\begin{tikzpicture}
\draw (0,0)  -- (1,1) -- (2,0)-- (0,0);
\draw (0.5,0.5) node[circle,fill,inner sep=1pt] {};
\draw (1.5,0.5) node[circle,fill,inner sep=1pt] {};
\draw[thick,->] (2.2,0.5) -- (2.7,0.5);
\end{tikzpicture}
\qquad
\begin{tikzpicture}
\draw (0,0)  -- (1,1) -- (2,0)-- (0,0);
\draw (0.5,0.5) node[circle,fill,inner sep=1pt] {};
\draw (1.5,0.5) node[circle,fill,inner sep=1pt] {};
\draw (0.5,0.5)  -- (1.5,0.5);
\end{tikzpicture}
\caption{A sequence of bisections.}
\label{fig:bisection}
\end{figure}
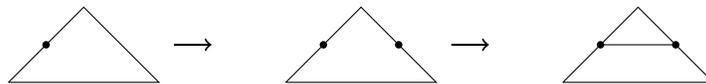

Before we finish this section, let us point out that along the process of (de)composing discrete Morse functions one 
may need to subdivide some of the cells in the complex.  As in (cf. \cite{stallings}, \cite{forman1}), a bisection 
of a cell will be a subdivision of a single cell  into two, as shown in Figure \ref{fig:bisection}. 


\section{Composing  Perfect Discrete Morse Functions on a Connected Sum}

In this section, we prove our first results.  We show how to compose two perfect discrete Morse functions on triangulated manifolds. Our method is geometrically streightforward. The construction of the perfect discrete Morse function is similar to a method of Benedetti \cite{Benedetti} to prove that perfect discrete Morse functions can be combined into a perfect discrete Morse function on the disjoint union, which he then used to show that a handlebody decomposition of a triangulated manifold gives a discrete Morse function on it with critical cells exactly corresponding to the handles in the decomposition, thus extending a similar result of Gallais \cite{Gallais} on 3-manifolds to the general case.  

Let $M=M_{1}\# M_{2} $ be the connected 
sum of two closed, oriented, triangulated $n$ dimensional manifolds, and $f_{1}$ and $f_{2}$ be perfect discrete Morse functions on $M_{1}$  and  $M_{2}$ respectively.  
We form the connected sum  $M=M_{1}\# M_{2}$ by removing an $n$-cell of $M_2$ with minimal vertex in its boundary and the critical $n$-cell of $M_1$.
\begin{theorem} \label{compose}
	Let $M=M_{1}\# M_{2}$ be given as above.  Then, there exists a polyhedral subdivision $\widetilde{M}$ of $M$ and a perfect discrete Morse function $f$ on  $\widetilde{M}$ 
	that agrees, up to a constant on each summand with $f_{1}$ and $f_{2}$, except on a neighbourhood of the two removed cells.
\end{theorem}

\begin{proof}
	In order to prove the theorem, we will show that the gradient vector fields $V_1$ and $V_2$ of the functions $f_1$ and $f_2$ respectively, induce a discrete vector field $V$ on the connected sum which restricts to $V_1$ and $V_2$ on the two summands, and which has the minimal possible number of critical cells and no loops.  We may assume that the connected sum $M=M_{1}\# M_{2}$ is formed by removing a non-critical $n$-cell $\beta$ of $V_{2}$ with the minimal vertex $v$ in its boundary and the critical $n$-cell $\alpha$ of $V_{1}$.  
	
On $M_1\setminus \alpha$, we attach a tube $L=\partial \alpha\times[0,1]$ with the natural product cell decomposition along  $\partial \alpha\times {\{0\}}$ and extend the discrete vector field $V_{1}|_{M_{1}-\alpha}$ to $(M_{1}-\alpha)\cup L$ so that each cell $\sigma$ in $\partial \alpha\times{\{1\}}$ is paired with the cell $\sigma\times(0,1)$
(see Figure \ref{fig:extending} for $n=2$).  Also note that if $\partial \alpha$ contain any critical $1$-cells that belong to $ M_{1}$, we may do bisections to push these cells in ${M_{1}-\alpha}$.
	
	\begin{figure}[h]
		\centering
		\begin{tikzpicture}
		\draw (0,0)  -- (2,3) -- (4,0)-- (0,0);
		\draw[thick,->] (2,3) -- (2.4,2.4);
		\draw[thick,->] (0,0) -- (0,-0.5);
		\draw[thick,->] (4,0) -- (4.5,0);
		\draw[thick,->] (2,0) -- (2,-0.5);
		\draw[thick,->] (1,1.5) -- (0.5,1.7);
		\node[inner sep=0,anchor=west,text width=3.3cm] (note1) at (1.9,1.2) {$\alpha$};
		\draw[white] (0,-0.5) -- (0,-0.5);
		\end{tikzpicture}
		\qquad
		\begin{tikzpicture}
		\draw (0,0)  -- (0,2) -- (3,2)-- (3,0)-- (0,0);
		\draw (0,2) -- (1.7,3) -- (3,2);
		\draw[dotted] (0,0) -- (1.7,1) -- (3,0);
		\draw[dotted] (1.7,1) -- (1.7,3);
		\draw[thick,dotted,->] (1.13,0.7) -- (0.5,0.8);
		\node[inner sep=0,anchor=west,text width=3.3cm] (note1) at (3.3,3) {$\partial \alpha\times {\{1\}}$};
		\node[inner sep=0,anchor=west,text width=3.3cm] (note1) at (3.3,0.5) {$\partial \alpha\times {\{0\}}$};
		\draw[thick,->] (0,2) -- (0,1.5);
		\draw[thick,->] (1.5,2) -- (1.5,1.5);
		\draw[thick,->] (3,2) -- (3,1.5);
		\draw[thick,dotted,->] (1,2.6) -- (1,2.1);
		\draw[thick,dotted,->] (1.7,3) -- (1.7,2.5);
		\draw[thick,dotted,->] (2.2,2.6) -- (2.2,2.1);
		\draw[thick,->] (0,0) -- (0,-0.5);
		\draw[thick,->] (3,0) -- (3.5,0);
		\draw[thick,->] (1.5,0) -- (1.5,-0.5);
		\draw[thick,dotted,->] (1.7,1) -- (2.09,0.7);
		\end{tikzpicture}
		\caption{The discrete vector field on $\partial \alpha\times\{1\}$.}
		\label{fig:extending}
	\end{figure}
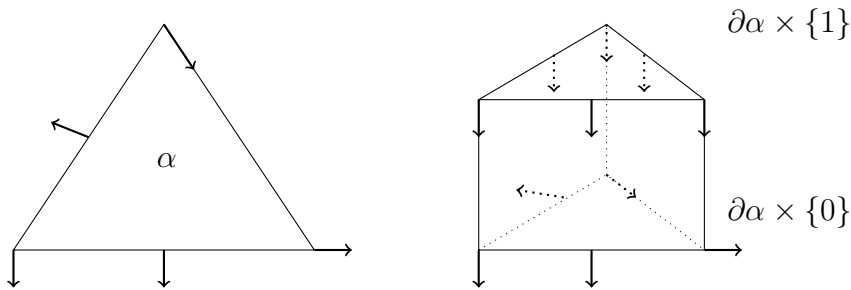
	
	On $M_2$, we subdivide the smallest simplicial complex $J$ consisting of $\beta$ and all its faces in the following way. Let $H: J \times [0,1]\to J $ be a linear (deformation) retraction of $J$ onto the critical vertex $v$. Then, $J'=H_{1/2}(J)$ represents a smaller copy of $J$ and $\beta^{'}=H_{1/2}(\beta)$. The closure $J''$ of the complement $J-J'$  is a product of the link $\Lk_J(v)$ of $v$ in $J$ (that is, the face of $\beta$ opposite to $v$ with all its faces) with the interval $[0,1/2]$, and the simplices of $J''-J'$ are in bijective correspondence to the simplices of $J-v$. 
	(see Figure \ref{fig:dividing} for $n=2$).

	\begin{figure}[h]
		\centering
		\begin{tikzpicture}
		\draw (0,0)  -- (2,3) -- (4,0)-- (0,0);
		\draw[thick,->] (0,0) -- (-0.5,0);
		\draw[thick,->] (4,0) -- (4.5,0);
		\draw[thick,->] (2,0) -- (2,0.5);
		\draw[thick,->] (1,1.5) -- (0.7,2);
		\draw[thick,->] (3,1.5) -- (3.3,2);
		\node[inner sep=0,anchor=west,text width=3.3cm] (note1) at (1.9,1.5) {$\beta$};
		\node[inner sep=0,anchor=west,text width=3.3cm] (note1) at (2,3.1) {$v$};
		\end{tikzpicture}
		\qquad
		\begin{tikzpicture}
		\draw (0,0)  -- (2,3) -- (4,0)-- (0,0);
		\draw (1,1.5)  -- (3,1.5);
		\draw[thick,->] (0,0) -- (-0.5,0);
		\draw[thick,->] (4,0) -- (4.5,0);
		\draw[thick,->] (2,0) -- (2,0.5);
		\draw[thick,->] (0.5,0.75) -- (0.2,1.2);
		\draw[thick,->] (3.5,0.75) -- (3.8,1.2);
		\draw (1,1.5) node[circle,fill,inner sep=1pt] {};
		\draw (3,1.5) node[circle,fill,inner sep=1pt] {};
		\node[inner sep=0,anchor=west,text width=3.3cm] (note1) at (1.9,2.1) {$\beta^{'}$};
		\node[inner sep=0,anchor=west,text width=3.3cm] (note1) at (1.5,0.9) {$\beta - \beta^{'}$};
		\node[inner sep=0,anchor=west,text width=3.3cm] (note1) at (2,3.1) {$v$};
		\end{tikzpicture}
		\caption{Obtaining $\beta^{'}$ with discrete vector field on it.}
		\label{fig:dividing}
	\end{figure}
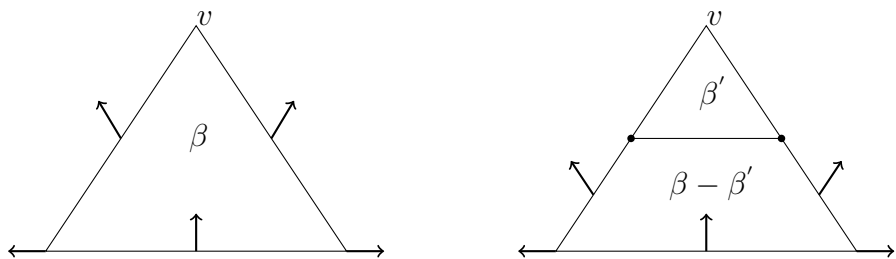

	We extend the vector field $V_2$ on $M_2-J$ to $J''-J'$ so that for any pair $(\sigma,\tau)\in V_{2}$ where $\sigma\in J$ and either $\tau\in J$ or $\tau\in (M_{2}-J)$, there is a corresponding pair $(\sigma'',\tau'')$ where $\sigma''\in J''$ and either $\tau''\in J''$ or $\tau''\in (M_{2}-J)$, and all simplices on $\partial \beta'$ are unpaired.  The connected sum$(\widetilde M)$ is formed by removing $\beta'$ and identifying $\partial \beta^{'}$  with $\partial \alpha\times \{1\}$. Since each cell $\sigma\in\partial \beta'\cong \partial \alpha\times \{1\}$ was unpaired in $V_2$, it can be without conflicts paired according to $V_1$ to the cell $\sigma\times[0,1]$ in $L$. The vector fields on $(M_1-\alpha)\cup L$ and $(M_2-\beta)\cup \beta''$, where $\beta''=\beta-\beta'$, together form a vector field $V$ on $\widetilde{M}$ given by 
	$$V(\gamma)=\left\{\begin{array}{lll}
	V_1(\gamma)&; & \gamma\in M_1-\alpha\\
	\gamma\times(0,1)&; & \gamma\in \partial \alpha\times \{1\}\cong \partial \beta' \\
	
	V_2(\sigma)&; & \gamma=\sigma''\in J''\\
	V_2(\gamma)&; & \gamma\in M_2-J
	\end{array}\right.$$ 
	
	Since $V_{1}$ and $V_{2}$ do not contain loops and all the arrows on the boundary circle point in the same direction (towards $M_{1}$), there cannot be any closed $V$-paths, and we have a discrete gradient vector field on $\widetilde{M}$.  
	Since the maximum $\alpha \in M_1$ has been removed and the minimum $v\in M_2$ has been paired, the numbers of critical cells in $V$  are
	\begin{eqnarray*}
		m_0 &=& 1=b_0(M),\\
		m_1 &=& m_1(f_1)+m_1(f_2)=b_1(M),\\
		m_2 &=& 1=b_2(M).
	\end{eqnarray*} 
	So, $V$ is a perfect discrete vector field. A perfect discrete Morse function of $V$ is given by 
	$$f(\gamma)=\left\{\begin{array}{lll}
	f_1(\gamma)&; & \gamma\in M_1-\alpha\\
	f_1(\tau)+C/2&; & \gamma=\tau\times(0,1]\in L, \tau\in \partial \alpha\\
	f_2(\gamma)+C&; & \gamma\in M_2-J\\
	f_2(\tau)+C&; & \gamma=\tau'' \in J''
	\end{array}\right.$$
	where $C$ is a big enough constant (bigger than $f_1(\alpha)+1$).
\end{proof}

\section{Decomposing Perfect Discrete Morse Functions on Surfaces}

In this section we are going to prove the converse of Theorem \ref{compose} in the case of surfaces. That is, any perfect discrete 
Morse function $f$ on a closed triangulated surface $M$ that is a connected sum $M_1\# M_2$ restricts to perfect discrete Morse 
functions $f_1$ and $f_2$ on the two summands.  Before we state our theorem let us first point out an observation for a 
discrete Morse function on a particular boundary curve of an oriented surface.  Following \cite{bruno} the cells in the boundary 
of a manifold with a discrete Morse function $f$ on it are called boundary critical cells if they are critical for $f$.

\begin{lemma} \label{boundary condition}
	Let $M$ be a closed, oriented surface of genus $g$ and $f$ be a discrete Morse function on $M$.  Let $D$ be 
	an open disk in $M$ which contains exactly one critical $0$-cell or one critical $2$-cell and let $C= \partial (M-D)$.  
	Then the number of boundary critical vertices and edges of $f|_{M-D}$ on $C$ must be equal. 
\end{lemma}

\begin{proof}
	The Euler characteristic of the closed genus $g$-surface $M$ is $\chi(M)=2-2g$.  Removing a disk decreases Euler characteristic by 
	one.  So, we have
	
	\begin{align}
	\chi(M-D)&=2-2g-1\nonumber\\
	         &=1-2g\label{1}.
	\end{align}
	Assume that $D$ contains only one critical $0$-cell of $f$ in $M$ (the argument is similar if $D$ 
	contains only one critical $2$-cell). Assume also that $f$ has $ m_{0} $ critical $0$-cells, $ m_{1} $ critical $1$-cells and $m_{2}$ critical $2$-cells.  Let $n$ be the number of the boundary critical $0$-cells  and $m$ be the number of the 
	boundary critical $1$-cells on $C$.  Note that these $0$-cells and $1$-cells are critical for $f|_{M-D}$.  Hence, $f|_{M-D}$ has $(m_{0}+n)$ 
	critical $0$-cells, $(m_{1}+m)$ critical $1$-cells and $m_{2}$ critical $2$-cells.  Then, by the Morse inequalities,

		 \begin{align}
	\chi (M)&= m_{0}-m_{1}+m_{2}=n-m+(1-2g)= 2-2g,  \text{  and}\nonumber\\
	\chi (M-D) &=(m_{0}+n-1)-(m_{1}+m)+m_{2}\nonumber\\
	           &= m_{0}-m_{1}+m_{2}+n-m-1\nonumber\\
	           &=2-2g+n-m-1\nonumber\\
	           &= 1-2g+(n-m) \label{2}.
	\end{align}
	The equations (\ref{1}) and (\ref{2}) implies that $n=m$.
\end{proof}

Indeed, $f|_C$ is a discrete Morse function in Lemma \ref{boundary condition} and 

$$
\chi (C) =0=m_{0}(f|_C)-m_{1}(f|_C)
$$ implies that $m_{0}(f|_C)=m_{1}(f|_C)$ where  $m_{0}(f|_C), m_{1}(f|_C)$ are the numbers of critical $0$-cells and critical $1$-cells of $f|_C$, respectively.  However, our main goal in Lemma \ref{boundary condition} is to show that if $C$ bounds a surface which is $M-D$, then the number of the boundary critical $0$-cells and $1$-cells is equal.

In the proof of the main theorem of this section we will also need the following two well known results which we add for the sake of 
completeness of the paper.

\begin{lemma}\label{1-paths}
	On any regular cell complex $1$-paths can merge but cannot split. All $1$-paths in the gradient field of a perfect discrete Morse function 
	on a connected cell complex thus form a tree with root at the minimal vertex.
\end{lemma}

\begin{proof} If a $1$-path would split at some vertex, this would imply that the vertex is paired with two different edges which is clearly 
	not possible. Every $1$-path ends in a critical vertex. A connected cell complex has $b_0=1$ so a perfect discrete Morse function on it has 
	one critical vertex which is the root of a tree formed by all $1$-paths. 
\end{proof}

\begin{lemma}\label{n-paths}
	The $n$-paths in the gradient field of a discrete Morse function on a closed triangulated $n$-dimensional manifold can split but not merge. 
	For every $n$-dimensional cell $\sigma$ of a triangulated $n$-manifold with a perfect discrete Morse function, with the exception of the 
	unique critical $n$-cell, there exists a unique $n$-path beginning in the boundary of the critical $n$-cell and ending in $\sigma$. 
\end{lemma}

\begin{proof} In a closed triangulated $n$-manifold, every $(n-1)$-cell is the common face of exactly two $n$-cells. If two $n$-paths would 
	merge at some $(n-1)$-cell $\tau$, then $\tau$ would be the common face of its $n$-dimensional pair as well as at least two other $n$-cells 
	which is not possible. If $\sigma$ is an $n$-cell in an $n$-path, then its pair is an $(n-1)$-cell which has precisely one other $n$-coface. 
	If this is not critical, it is the only possible previous $n$-cell in the path. Repeating this argument we eventually trace out a unique $n$-path 
	starting in a critical $n$-cell.
\end{proof}

Next we prove that a given perfect discrete Morse function on an oriented surface can be decomposed easily into two perfect discrete Morse functions if the separating 
curve is nice enough.  Then in Theorem~\ref{decompose} we prove that one can always find a nice enough separating curve.

Let $M_i$  be a closed, oriented, triangulated genus $g_i$ surface and $D_i \subset M_i$ be an embedded disk for $i=1,2$.  Also let 
$M= M_1 \#_{C} M_2$ along the boundary circle $C\approx  \partial (M-M_1) \approx  \partial (M-M_2)$.  Let $f$ be a perfect discrete Morse 
function on $M$ and $V$ be the gradient vector field induced by $f$ such that the restriction of $V$ to $M-M_{2} $ has one critical $0$-cell and  $2g_1$ critical $1$-cells and the restriction of $V$ to 
$M-M_{1} $ has one critical $2$-cell and $2g_2$ critical $1$-cells. 

\begin{theorem} \label{boundary}
	Under the above conditions, if we further assume that there are no arrows on the vertices and edges of $C$ pointing into $M-M_{1}$, 
	then we can extend $ V|_{M-M_{2}}$  to $M_{1} $ and $ V|_{M-M_{1}}$  to $M_{2}$ as perfect gradient vector fields induced by  perfect discrete Morse functions which agree with $f$ on $M-M_{1} $ and $ M-M_{2} $.
\end{theorem}

\begin{proof} 
	Each cell on the boundary curve $C$ is paired either with a cell on $C$ or a cell in $M-M_{2}$.  Assume that there are $k$ pairs of 
	cells on $C$.  Our aim is obtaining a perfect gradient vector field on $ M_{2} $ after attaching a disk $D_2$ to $M-M_{1}$ in the following way.  We triangulate the disk $D_2$ as a cone over $C$.  That is, we choose an interior point, called $v$, in $D_2$ and connect $v$ with the vertices on $C$ via straight line segments.  We pair  each boundary critical cell of $M-M_1$ (each is paired in $M$ with a cell in $M-M_2$) 
	with its coface in the cone.  For each pair $(\sigma,\tau)\in V$ of a $0$-cell $\sigma$ and a $1$-cell $\tau$ in $C$, the corresponding pair $(\sigma',\tau')$ inside the cone, where $\sigma'$ and $\tau'$ are cofaces of $\sigma$ and $\tau$ respectively, is formed.   
	The vertex $v$ of the cone will be the critical vertex (see Figure \ref{fig:disc3} for $k = 2$).  The vector field $V$ has no non-trivial closed path in $M$, and hence $ V|_{M-M_{1}}$ has no non-trivial closed path.  Additionally, we pair all boundary critical cells on $C$ with their cofaces in $D_{2}$, and hence, the extended vector field on $(M-M_1)\cup_C D_2\cong M_2$, called $V_{2}$, has no non-trivial closed paths.  Therefore, $V_{2}$ is the gradient vector field of a discrete Morse function \cite[Theorem 3.5]{forman2}, say $f_{2}$, on $M_{2}$  with the following numbers of critical cells which tells that $V_{2}$ is a perfect gradient vector field.

	  \begin{eqnarray*}
	  	m_0 &=& 1=b_0(M_{2}),\\
	  	m_1 &=& m_1(f_{2})=b_1(M_{2}),\\
	  	m_2 &=& 1=b_2(M_{2}).
	  \end{eqnarray*}

We can also obtain $f_{2}$ specifically by assigning values to the cells in $D_{2}$ so that they descend along the gradient paths, and the minimum is at the critical vertex $v$.	  
	
	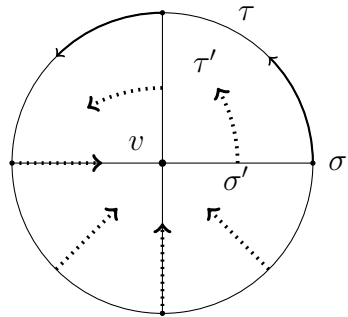
\begin{figure}[h]
		\begin{center}
			
			\begin{tikzpicture}[scale=2]
			\draw (1,0) arc (0:90:1cm);
			\draw (0,1) arc (90:180:1cm);
			\draw (-1,0) arc (180:270:1cm);
			\draw (0,-1) arc (270:360:1cm);
			\node[inner sep=0,anchor=west,text width=3.3cm] (note1) at (-0.23,0.13) {$v$};
			\node[inner sep=0,anchor=west,text width=3.3cm] (note1) at (1.1,0) {$\sigma$};
			\node[inner sep=0,anchor=west,text width=3.3cm] (note1) at (0.5,1) {$\tau$};
			\node[inner sep=0,anchor=west,text width=3.3cm] (note1) at (0.4,-0.1) {$\sigma'$};
			\node[inner sep=0,anchor=west,text width=3.3cm] (note1) at (0.2,0.68) {$\tau'$};
			\draw[thick,->] (1,0) arc (0:45:1cm);
			\draw[thick,->] (0,1) arc (90:135:1cm);
			\draw (0,-1)--(0,1);
			\draw (-1,0)--(1,0);
			\draw[thick,dotted,->] [line width=0.5mm](-1,0) -- (-0.4,0);
			\draw[thick,dotted,->] [line width=0.5mm](-0.7071,-0.7071) -- (-0.3,-0.3);
			\draw[thick,dotted,->] [line width=0.5mm](0,-1) -- (0,-0.4);
			\draw[thick,dotted,->] [line width=0.5mm](0.7071,-0.7071) -- (0.3,-0.3);
			\draw (1,0) node[circle,fill, inner sep=0.7pt] {};
			\draw (-1,0) node[circle,fill, inner sep=0.7pt] {};
			\draw (0,-1) node[circle,fill, inner sep=0.7pt] {};
			\draw (0,1) node[circle,fill, inner sep=0.7pt] {};
			\draw (0,0) node[circle,fill, inner sep=1pt] {};
			\draw[thick,dotted,->] [line width=0.5mm](0,0.5) arc (90:120:1cm);
			\draw[thick,dotted,->] [line width=0.5mm] (0.5,0) arc (0:30:1cm);
			\end{tikzpicture}
			
			\caption{The discrete vector field on the disk with the critical $0$-cell in the center.}
			\label{fig:disc3}
		\end{center}
	\end{figure}
	
	Now, we construct a perfect gradient vector field on $ M_{1} $ after attaching a disk $D_1$ to $M-M_{2}$.   In this case, there are no boundary critical cells on $C$ since each cell on $C$ is paired either to a cell on $C$ or to a cell in $M-M_2$. Again, we triangulate the 
	disk $D_1$ as a cone over $C$. 
	Let $t_{i}$ be the triangles and $e_{i}$, $e_{i+1}$ be the faces of $t_{i}$ for $i=1,2,3,\ldots,n$, which are ordered in counterclockwise direction in the interior of $D_{1}$.  We pair $e_{i}$'s with $t_{i}$'s where $i=1,2,3,\ldots,n-1$  and pair $v$ with $e_{n}$. 
	At the end, one unpaired triangle, $t_{n}$, will remain in $D_1$ since the number of edges and 
	triangles in $D_1$ equal (see Figure \ref{fig:disc4} for $k = 2$ and $n=4$).
	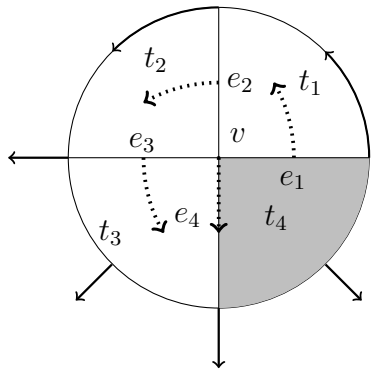
\begin{figure}[h]
		\centering
		\begin{tikzpicture}[scale=2]
		\draw (1,0) arc (0:90:1cm);
		\draw (0,1) arc (90:180:1cm);
		\draw (-1,0) arc (180:270:1cm);
		\draw (0,-1) arc (270:360:1cm);
		\fill[gray!50] (0,0) -- (0,-1) arc (270:360:1cm) -- cycle;
		\node[inner sep=0,anchor=west,text width=3.3cm] (note1) at (0.07,0.15) {$v$};
		\node[inner sep=0,anchor=west,text width=3.5cm] (note1) at (0.4,-0.15) {$e_{1}$};
		\node[inner sep=0,anchor=west,text width=3.3cm] (note1) at (0.53,0.5) {$t_{1}$};
		\node[inner sep=0,anchor=west,text width=3.3cm] (note1) at (0.05,0.5) {$e_{2}$};
		\node[inner sep=0,anchor=west,text width=3.3cm] (note1) at (-0.5,0.65) {$t_{2}$};
		\node[inner sep=0,anchor=west,text width=3.3cm] (note1) at (-0.6,0.1) {$e_{3}$};
		\node[inner sep=0,anchor=west,text width=3.3cm] (note1) at (-0.8,-0.5) {$t_{3}$};
		\node[inner sep=0,anchor=west,text width=3.3cm] (note1) at (-0.3,-0.4) {$e_{4}$};
		\node[inner sep=0,anchor=west,text width=3.3cm] (note1) at (0.3,-0.4) {$t_{4}$};
		\draw[thick,->] (1,0) arc (0:45:1cm);
		\draw[thick,->] (0,1) arc (90:135:1cm);
		\draw (0,-1)--(0,1);
		\draw (-1,0)--(1,0);
		\draw (0,0) node {.};
		\draw[thick,dotted,->] [line width=0.5mm](0,0.5) arc (90:120:1cm);
		\draw[thick,dotted,->] [line width=0.5mm](0.5,0) arc (0:30:1cm);
		\draw[thick,dotted,->] [line width=0.5mm] (-0.5,0) arc (180:210:1cm);
		\draw[thick,dotted,->] [line width=0.5mm](0,0)-- (0,-0.5);
		\draw[thick,->][line width=0.3mm](-0.71,-0.71)--(-0.95,-0.95);
		\draw[thick,->][line width=0.3mm](0.71,-0.71)--(0.95,-0.95);
		\draw[thick,->][line width=0.3mm](-1,0)--(-1.4,0);
		\draw[thick,->][line width=0.3mm](0,-1)--(0,-1.4);
		\end{tikzpicture}
		\caption{The discrete vector field on the disk with the critical $2$- cell.}
		\label{fig:disc4}
	\end{figure}
	We obtain a discrete vector field on $D_1$ with a critical $2$- cell (see Figure \ref{fig:disc4}).  We do not get any non-trivial closed paths in $M_{1}$ 
	after attaching $D_1$ since there is no non-trivial closed path in $M-M_{2}$ and there is no $i$- path, $i=1,2$, which begins in boundary of a cell in
	$D_1$ and comes back to $D_1$.  Hence, we obtain a gradient vector field $V_{1}$ induced by a discrete Morse function $f_{1}$  on  $(M-M_2)\cup_C D_1\cong M_1$ \cite[Theorem 3.5]{forman2}. In addition, $V_{1}$ is the perfect gradient vector field of $f_{1}$ with the number of critical cells below;
	
	\begin{eqnarray*}
		m_0 &=& 1=b_0(M_{1}),\\
		m_1 &=& m_1(f_{1})=b_1(M_{1}),\\
		m_2 &=& 1=b_2(M_{1}).
	\end{eqnarray*}

	Similarly, we can define $f_{1}$ specifically by assigning a value big enough to the critical triangle $t_{n}$ and descending values along the $V_{1}$-paths in $D_1$ keeping in mind that the values should all be bigger than the values on the circle $C$.
 	
\end{proof}

\begin{remark}
	 An alternative way to see that the discrete Morse function $ f|_{M-M_{2}}$ extends to $M_{1}$ as a perfect discrete Morse function is as follows:
	 We triangulate the disc $D_{1}$ as a cone with a single interior vertex $v$.  We choose one of the triangles in $D_{1}$, called $T$.  Clearly, $(D_{1}-int(T))\searrow C$.  Indeed,  $(M-M_2)\cup_C D_1\cong M_1$ and   $(M_{1}-int(T))\searrow (M-M_{2})$.  Hence, $f|_{M-M_{2}}$ can be extended to $M_{1}-T$ as a discrete Morse function without obtaining any new critical cell by following the inverse of the collapse steps of $(D_{1}-int(T))\searrow C$ \cite[Lemma 4.3]{forman1}.  Let $g$ be the extension of  $f|_{M-M_{2}}$ to $(M_{1}-T)$.  Then, we define a discrete Morse function on $M_{1}$ as in the following way,
	
	$$g'(\sigma)=\left\{\begin{array}{lll}
	g(\sigma)&; & \sigma\in M_1-T\\
	\textrm{max}~\{{g(\partial\sigma)}\}+c&; & \sigma=T\\
	\end{array}\right.$$
	
	Therefore, $g'$ is a perfect discrete Morse function with a unique critical $2$-cell $T$.
\end{remark}

Now, let $M=M_{1}\# M_{2}$ be the connected sum of two closed oriented triangulated surfaces of genus $g_{1}$ and $g_{2}$ respectively.  Let $f$ be 
a perfect discrete Morse function on $M$ such that the critical $2$-cell of $f$ is in $M-M_{1}$ and the critical $0$-cell of $f$ is in $M-M_{2}$.  
In addition, assume that the critical cells of $f$ are separated in the sense that the star of a critical cell contains no other critical cells. 
This can always be achieved after a suitable subdivision of $M$. Let $V$ be the gradient vector field induced by $f$.  In the following theorem, 
we look for a suitable boundary curve $C$, such as in Theorem~\ref{boundary}, so that one can decompose $f$ accordingly. 

\begin{theorem} \label{decompose}
	Let $M=M_{1}\# M_{2}$ be as above.  Then we can find a circle $C$ on $M$ such that $M= M_1 \#_{C} M_2$ and the cells on 
	$C$ are paired with either the cells on $C$ or the cells in $M-M_{2}$ or both.
	
\end{theorem}

\begin{proof}
	Our goal is to split $M$ into two parts, $M-M_2$ and $M-M_1$ along a curve $C$ which satisfies the above condition concerning
	how the cells on $C$ are paired. 
	There are $2(g_1+g_2)$ critical $1$-cells of $f$, bottom 
	$2g_1$ critical $1$-cells belong to $M-M_2$ and other critical $1$-cells belong to $M-M_1$.  If $f$ is not injective and there is an ambiguity in deciding 
	which cells are lower and which are higher,  we may perturb $f$ a little without creating any additional critical cells as in 
	the proof of \cite[Theorem 3.3]{forman1}.   
	
	To find the separating circle $C$ as mentioned above, we follow the flow induced by $f$ (the direction along which $f$ is non-increasing) from the 
	critical $2$-cell to the critical $1$-cells which belong to $M-M_1$.  We consider the $2$-paths that begin from a face of the critical 
	$2$-cell and end at these critical $1$-cells.  Each critical 1-cell has precisely two 2-cofaces, and according to Lemma~\ref{n-paths} each one of 
	these is the end of precisely one 2-path that begins in a face of the critical 2-cell.
	If the critical $2$-cell is removed, we may collapse along these paths, starting at a free edge in the boundary of the critical $2$-cell. These 
	collapses produce tunnels with cells in the boundary that are paired either to cells in the boundary or to cells in the remaining part of the manifold. 
	The critical $2$-cell together with these tunnels and the critical $1$-cells at their ends will represent the core of $M-M_1$. The boundary of this region 
	is a curve $C'$. If this curve is a circle, we have achieved our goal since this circle separates the region that has been removed and represents $M-M_{1}$ and the 
	arrows along the curve point either along the curve or into the remaining part $M-M_{2}$. 
	
    Note that $C'$ might not be a circle due to the fact that the triangulation on $M$ might be too coarse to allow the construction of such a circle.  
	The problem is that two different tunnels, that is, two paths from the critical $2$-cell to critical $1$-cells in $M-M_{1}$ come too close together 
	and meet in a vertex or are separated only by an edge, or more generally, by a $1$-path.  As a result of this, the curve $C'$ is a union of  wedges of 
	several circles possibly with additional paths in the interior of $M-M_{1}$ that connecting the wedges.  In order to obtain a separating curve that is a circle, we have to do some 
	subdivisions.  For instance, in Example \ref{genus-2}, the resulting boundary of  several $2$-paths meet at the vertex $3$, so the curve $C'$ is a wedge of circles meeting at this 
	vertex, and the path $a$ represents an additional edge(arc) joining these circles, since two $2$-paths meet in this edge.
	
	\textit{Case} 1: Assume first that the curve $C'$ is a wedge of $n$ circles at a single vertex $v$.  Let $c_{i}$, for $i=1,2,\ldots,n$, be the circles in the wedge 
	ordered in the counterclockwise direction and $v_{i1}, v_{i2}$ be two of the vertices on each circle $c_{i}$ such that they are in the link of $v$ in $C'$, that is $v_{i1}, v_{i2} \in Lk_{v}(C')$ and they are ordered as  $v, v_{i1}, v_{i2}$ .  
	As in this case, throughout the proof, whenever we order circles or vertices, 
	we always assume that they are ordered in the counterclockwise direction.  

	We connect the vertices in each pair  
	$\{v_{12}, v_{21}\}, \{v_{22}, v_{31}\},\ldots,\{v_{(n-1)2}, v_{n1}\}$ by a union of edges obtained 
	by subdividing as few as possible triangles in the tunnel.  Then, we follow these edges as a part of the boundary curve in order to obtain a circle instead of
	the old boundary component  (see Figure \ref{fig:wedge1} for $n=3$).  We always pair all the cells on the new boundary components with their cofaces 
	outside of the tunnel by using the method given in \cite[Theorem 12.1]{forman1} (e.g. see Figure \ref{fig:unwedge1}).  Note that the vertex which is the starting point of the dotted arrow is paired 
	with a $1$-cell outside of the tunnel.  
	
	\begin{figure}[h]
		\centering
		\resizebox{0.45\textwidth}{!}{%
			\begin{tikzpicture}[scale=2]
			\draw (0,2)--(0,4);
			\draw (2,4)--(4,4);
			\draw (4,0)--(0,0);
			\draw (2,0) -- (2,2);
			\draw (2,2.7) -- (2,2);
			\draw (2,2) -- (2.7,2);
			\draw (4,0) -- (4,2);
			\draw[fill=black!20] (0,0)--(2,0)--(2,2)--(0,0);
			\draw[fill=black!20] (2,2)--(2,4)--(4,4)--(2,2);
			\draw[fill=black!20] (0,2)--(0,4)--(2,2)--(0,2);
			\draw[fill=black!20] (2,2)--(2,0)--(4,0)--(2,2);
			\draw[fill=black!20] (2,2)--(4,0)--(4,2)--(2,2);
			\draw (0,0) -- (4,4);
			\node[inner sep=0,anchor=west,text width=3.3cm] (note1) at (0.3,1.5) {$c_{1}$};
			\node[inner sep=0,anchor=west,text width=3.3cm] (note1) at (1.2,3.3) {$c_{3}$};
			\node[inner sep=0,anchor=west,text width=3.3cm] (note1) at (3.35,2.7) {$c_{2}$};
			\node[inner sep=0,anchor=west,text width=3.3cm] (note1) at (2,1.8) {$v$};
			\draw[dotted] [line width=0.5mm](0,4) -- (2,4);
			\draw[dotted] [line width=0.5mm](2,2) -- (0,2);
			\draw[dotted] [line width=0.5mm](2,2) -- (4,4);
			\draw[dotted] [line width=0.5mm](2,2) -- (4,2);
			\draw[dotted] [line width=0.5mm](0,0) -- (2,2);
			\draw[dotted] [line width=0.5mm](2,4) -- (2,2);
			\draw[dotted] [line width=0.5mm](0,0) -- (0,2);
			\draw[dotted] [line width=0.5mm](4,4) -- (4,2);
			\draw (0,0) node[circle,fill,inner sep=2pt] {};
			\draw (2,0) node[circle,fill,inner sep=2pt] {};
			\draw (4,0) node[circle,fill,inner sep=2pt] {};
			\draw (0,2) node[circle,fill,inner sep=2pt] {};
			\draw (2,2) node[circle,fill,inner sep=2pt] {};
			\draw (4,2) node[circle,fill,inner sep=2pt] {};
			\draw (0,4) node[circle,fill,inner sep=2pt] {};
			\draw (2,4) node[circle,fill,inner sep=2pt] {};
			\draw (4,4) node[circle,fill,inner sep=2pt] {};
			\draw (2,2.7) node[circle,fill,inner sep=2pt] {};
			\draw (2.6,2.6) node[circle,fill,inner sep=2pt] {};
			\draw (2.7,2) node[circle,fill,inner sep=2pt] {};
			\draw (1.4,1.4) node[circle,fill,inner sep=2pt] {};
			\node[inner sep=0,anchor=west,text width=3.3cm] (note1) at (2.7,2.5) {$v_{22}$};
			\node[inner sep=0,anchor=west,text width=3.3cm] (note1) at (1.5,2.7) {$v_{31}$};
			\node[inner sep=0,anchor=west,text width=3.3cm] (note1) at (1.3,1.2) {$v_{12}$};
			\node[inner sep=0,anchor=west,text width=3.3cm] (note1) at (-0.5,4.3) {$v_{32}$};
			\node[inner sep=0,anchor=west,text width=3.3cm] (note1) at (-0.5,2) {$v_{11}$};
			\draw[thick,->] (2,2.7) -- (2,3);
			\draw[thick,->] (2.6,2.6) -- (2.3,2.3);
			\draw[thick,->] (2.7,2) -- (2.3,2);
			\draw[thick,->] (1.4,1.4) -- (1.7,1.7);
			\draw[thick,->] (1,4) -- (1,3.6);
			\draw[thick,->] (2,1) -- (2.4,1);
			\draw[thick,->] (3,1) -- (3.3,1.3);
			\draw[thick,->] (0,0) -- (0,0.4);
			\draw[thick,->] (0,2) -- (0.4,2);
			\draw[thick,->] (0,4) -- (0.3,3.7);
			\draw[thick,->] (1,0) -- (1,0.4);
			\draw[thick,->] (1,1) -- (0.7,1.3);
			\draw[thick,->] (0,3) -- (0.4,3);
			\draw[thick,->] (2,0) -- (2,-0.4);
			\draw[thick,->] (2,2) -- (2,2.4);
			\draw[thick,dotted,->] (2,4) -- (1.7,3.7);
			\draw[thick,->] (3,0) -- (3,-0.4);
			\draw[thick,->] (4,0) -- (4,0.4);
			\draw[thick,->] (4,2) -- (3.6,2);
			\draw[thick,->] (3,3) -- (3.3,2.7);
			\draw[thick,->] (3,4) -- (3,3.6);
			\draw[thick,->] (4,4) -- (4,3.6);
			\end{tikzpicture}
		} %
		\qquad
		\resizebox{0.45\textwidth}{!}{%
			\begin{tikzpicture}[scale=2]
			\draw (0,2)--(0,4);
			\draw (2,4)--(4,4);
			\draw (4,0)--(0,0);
			\draw (2,0) -- (2,2);
			\draw (2,2.7) -- (2,2);
			\draw (2,2) -- (2.7,2);
			\draw (4,0) -- (4,2);
			\draw[fill=black!20] (0,0)--(2,0)--(2,2)--(0,0);
			\draw[fill=black!20] (2,2)--(2,4)--(4,4)--(2,2);
			\draw[fill=black!20] (0,2)--(0,4)--(2,2)--(0,2);
			\draw[fill=black!20] (2,2)--(2,0)--(4,0)--(2,2);
			\draw[fill=black!20] (2,2)--(4,0)--(4,2)--(2,2);
			\draw (0,0) -- (4,4);
			\node[inner sep=0,anchor=west,text width=3.3cm] (note1) at (0.3,1.5) {$c_{1}$};
			\node[inner sep=0,anchor=west,text width=3.3cm] (note1) at (1.2,3.3) {$c_{3}$};
			\node[inner sep=0,anchor=west,text width=3.3cm] (note1) at (3.35,2.7) {$c_{2}$};
			\node[inner sep=0,anchor=west,text width=3.3cm] (note1) at (2,1.8) {$v$};
			\draw[dotted] [line width=0.5mm](2,2) -- (0,4);
			\draw[dotted] [line width=0.5mm](0,4) -- (2,4);
			\draw[dotted] [line width=0.5mm](2,2) -- (0,2);
			\draw[dotted] [line width=0.5mm](2,2) -- (4,4);
			\draw[dotted] [line width=0.5mm](2,2) -- (4,2);
			\draw[dotted] [line width=0.5mm](0,0) -- (2,2);
			\draw[dotted] [line width=0.5mm](2,4) -- (2,2);
			\draw[dotted] [line width=0.5mm](0,0) -- (0,2);
			\draw[dotted] [line width=0.5mm](4,4) -- (4,2);
			\draw (0,0) node[circle,fill,inner sep=2pt] {};
			\draw (2,0) node[circle,fill,inner sep=2pt] {};
			\draw (4,0) node[circle,fill,inner sep=2pt] {};
			\draw (0,2) node[circle,fill,inner sep=2pt] {};
			\draw (2,2) node[circle,fill,inner sep=2pt] {};
			\draw (4,2) node[circle,fill,inner sep=2pt] {};
			\draw (0,4) node[circle,fill,inner sep=2pt] {};
			\draw (2,4) node[circle,fill,inner sep=2pt] {};
			\draw (4,4) node[circle,fill,inner sep=2pt] {};
			\draw (2,2.7) node[circle,fill,inner sep=2pt] {};
			\draw (2.6,2.6) node[circle,fill,inner sep=2pt] {};
			\draw (2.7,2) node[circle,fill,inner sep=2pt] {};
			\draw (1.4,1.4) node[circle,fill,inner sep=2pt] {};
			\draw (2.5,1.5) node[circle,fill,inner sep=2pt] {};
			\draw (2,1.3) node[circle,fill,inner sep=2pt] {};
			\node[inner sep=0,anchor=west,text width=3.3cm] (note1) at (2.7,2.5) {$v_{22}$};
			\node[inner sep=0,anchor=west,text width=3.3cm] (note1) at (2.7,1.8) {$v_{21}$};
			\node[inner sep=0,anchor=west,text width=3.3cm] (note1) at (1.5,2.7) {$v_{31}$};
			\node[inner sep=0,anchor=west,text width=3.3cm] (note1) at (1.3,1.2) {$v_{12}$};
			\node[inner sep=0,anchor=west,text width=3.3cm] (note1) at (-0.5,4.3) {$v_{32}$};
			\node[inner sep=0,anchor=west,text width=3.3cm] (note1) at (-0.5,2) {$v_{11}$};
			\draw (2,2.7) -- (2.6,2.6);
			\draw (2.5,1.5) -- (2.7,2);
			\draw (2,1.3) -- (2.5,1.5);
			\draw (2,1.3) -- (1.4,1.4);
			\draw[thick,->] (2,2.7) -- (2,3);
			\draw[thick,->] (2.6,2.6) -- (2.3,2.3);
			\draw[thick,->] (2.7,2) -- (2.3,2);
			\draw[thick,->] (1.4,1.4) -- (1.7,1.7);
			\draw[thick,->] (1,4) -- (1,3.6);
			\draw[thick,->] (2,1) -- (2.4,1);
			\draw[thick,->] (3,1) -- (3.3,1.3);
			\draw[thick,->] (0,0) -- (0,0.4);
			\draw[thick,->] (0,2) -- (0.4,2);
			\draw[thick,->] (0,4) -- (0.3,3.7);
			\draw[thick,->] (1,0) -- (1,0.4);
			\draw[thick,->] (1,1) -- (0.7,1.3);
			\draw[thick,->] (0,3) -- (0.4,3);
			\draw[thick,->] (2,0) -- (2,-0.4);
			\draw[thick,->] (2,2) -- (2,2.4);
			\draw[thick,dotted,->] (2,4) -- (1.7,3.7);
			\draw[thick,->] (3,0) -- (3,-0.4);
			\draw[thick,->] (4,0) -- (4,0.4);
			\draw[thick,->] (4,2) -- (3.6,2);
			\draw[thick,->] (3,3) -- (3.3,2.7);
			\draw[thick,->] (3,4) -- (3,3.6);
			\draw[thick,->] (4,4) -- (4,3.6);
			\end{tikzpicture}
		} %
		
		\caption{}
		\label{fig:wedge1}
	\end{figure}
	
	In Figures \ref{fig:wedge1} and \ref{fig:unwedge1}, the gray regions represent a part of the interior of the tunnel the while white ones are a part of the outside of the tunnel, and the dotted edges represent the boundary curves.
	
	In this case, when the above process is complete, we obtain only one circle passing through the wedge point $v$, which is our boundary curve $C'$
	(see Figure \ref{fig:unwedge1}).
	\begin{figure}[h]
		\centering
		\resizebox{0.42\textwidth}{!}{%
			\begin{tikzpicture}[scale=2]
			\draw (0,2)--(0,4);
			\draw (2,4)--(4,4);
			\draw (4,0)--(0,0);
			\draw (2,0) -- (2,2);
			\draw (2,2.7) -- (2,2);
			\draw (2,2) -- (2.7,2);
			\draw (4,0) -- (4,2);
			\draw[fill=black!20] (1.4,1.4)--(2,1.3)--(2,0)--(0,0)--(1.4,1.4);
			\draw[fill=black!20] (2,2.7)--(2,4)--(4,4)--(2.6,2.6)--(2,2.7);
			\draw[fill=black!20] (0,2)--(0,4)--(2,2)--(0,2);
			\draw[fill=black!20] (2,0)--(2,1.3)--(2.5,1.5)--(4,0)--(2,0);
			\draw[fill=black!20] (2.5,1.5)--(4,0)--(4,2)--(2.7,2)--(2.5,1.5);
			\draw (0,0) -- (4,4);
			\node[inner sep=0,anchor=west,text width=3.3cm] (note1) at (0.3,1.5) {$c_{1}$};
			\node[inner sep=0,anchor=west,text width=3.3cm] (note1) at (1.2,3.3) {$c_{3}$};
			\node[inner sep=0,anchor=west,text width=3.3cm] (note1) at (3.35,2.7) {$c_{2}$};
			\node[inner sep=0,anchor=west,text width=3.3cm] (note1) at (2,1.8) {$v$};
			\draw[dotted] [line width=0.5mm](2,2) -- (0,4);
			\draw[dotted] [line width=0.5mm](0,4) -- (2,4);
			\draw[dotted] [line width=0.5mm](2,2) -- (0,2);
			\draw[dotted] [line width=0.5mm](2.6,2.6) -- (4,4);
			\draw[dotted] [line width=0.5mm](4,2) -- (4,0);
			\draw[dotted] [line width=0.5mm](4,0) -- (4,2);
			\draw[dotted] [line width=0.5mm](2.7,2) -- (4,2);
			\draw[dotted] [line width=0.5mm](0,0) -- (1.4,1.4);
			\draw[dotted] [line width=0.5mm](2,4) -- (2,2.7);
			\draw[dotted] [line width=0.5mm](0,0) -- (0,2);
			\draw[dotted] [line width=0.5mm](4,4) -- (4,2);
			\draw[dotted] [line width=0.5mm](1.4,1.4) -- (2,1.3);
			\draw[dotted] [line width=0.5mm](2,1.3) -- (2.5,1.5);
			\draw[dotted] [line width=0.5mm](2.5,1.5) -- (2.7,2);
			\draw[dotted] [line width=0.5mm](2.6,2.6) -- (2,2.7);
			\draw (2,2.7) -- (2.6,2.6);
			\draw (2.5,1.5) -- (2.7,2);
			\draw (2,1.3) -- (2.5,1.5);
			\draw (2,1.3) -- (1.4,1.4);
			\draw[dotted] [line width=0.5mm](1.4,1.4) -- (2,1.3);
			\draw[dotted] [line width=0.5mm](2,1.3) -- (2.5,1.5);
			\draw[dotted] [line width=0.5mm](2.5,1.5) -- (2.7,2);
			\draw[dotted] [line width=0.5mm](2.6,2.6) -- (2,2.7);
			\draw[thick,dotted,->] (2,1.3) -- (2,1.7);
			\draw[thick,dotted,->] (2.5,1.5) -- (2.2,1.8);
			\draw[thick,dotted,->] (1.7,1.35) -- (1.8,1.7);
			\draw[thick,dotted,->] (2.25,1.4) -- (2.15,1.7);
			\draw[thick,dotted,->] (2.6,1.7) -- (2.35,1.9);
			\draw[thick,dotted,->] (2.3,2.65) -- (2.13,2.33);
			\draw (0,0) node[circle,fill,inner sep=2pt] {};
			\draw (2,0) node[circle,fill,inner sep=2pt] {};
			\draw (4,0) node[circle,fill,inner sep=2pt] {};
			\draw (0,2) node[circle,fill,inner sep=2pt] {};
			\draw (2,2) node[circle,fill,inner sep=2pt] {};
			\draw (4,2) node[circle,fill,inner sep=2pt] {};
			\draw (0,4) node[circle,fill,inner sep=2pt] {};
			\draw (2,4) node[circle,fill,inner sep=2pt] {};
			\draw (4,4) node[circle,fill,inner sep=2pt] {};
			\draw (2,2.7) node[circle,fill,inner sep=2pt] {};
			\draw (2.6,2.6) node[circle,fill,inner sep=2pt] {};
			\draw (2.7,2) node[circle,fill,inner sep=2pt] {};
			\draw (1.4,1.4) node[circle,fill,inner sep=2pt] {};
			\draw (2.5,1.5) node[circle,fill,inner sep=2pt] {};
			\draw (2,1.3) node[circle,fill,inner sep=2pt] {};
			\node[inner sep=0,anchor=west,text width=3.3cm] (note1) at (2.7,2.5) {$v_{22}$};
			\node[inner sep=0,anchor=west,text width=3.3cm] (note1) at (2.7,1.8) {$v_{21}$};
			\node[inner sep=0,anchor=west,text width=3.3cm] (note1) at (1.5,2.7) {$v_{31}$};
			\node[inner sep=0,anchor=west,text width=3.3cm] (note1) at (1.3,1.2) {$v_{12}$};
			\node[inner sep=0,anchor=west,text width=3.3cm] (note1) at (-0.5,4.3) {$v_{32}$};
			\node[inner sep=0,anchor=west,text width=3.3cm] (note1) at (-0.5,2) {$v_{11}$};
			\draw[thick,->] (2,2.7) -- (2,3);
			\draw[thick,->] (2.6,2.6) -- (2.3,2.3);
			\draw[thick,->] (2.7,2) -- (2.3,2);
			\draw[thick,->] (1.4,1.4) -- (1.7,1.7);
			\draw[thick,->] (1,4) -- (1,3.6);
			\draw[thick,->] (2,1) -- (2.4,1);
			\draw[thick,->] (3,1) -- (3.3,1.3);
			\draw[thick,->] (0,0) -- (0,0.4);
			\draw[thick,->] (0,2) -- (0.4,2);
			\draw[thick,->] (0,4) -- (0.3,3.7);
			\draw[thick,->] (1,0) -- (1,0.4);
			\draw[thick,->] (1,1) -- (0.7,1.3);
			\draw[thick,->] (0,3) -- (0.4,3);
			\draw[thick,->] (2,0) -- (2,-0.4);
			\draw[thick,->] (2,2) -- (2,2.4);
			\draw[thick,dotted,->] (2,4) -- (1.7,3.7);
			\draw[thick,->] (3,0) -- (3,-0.4);
			\draw[thick,->] (4,0) -- (4,0.4);
			\draw[thick,->] (4,2) -- (3.6,2);
			\draw[thick,->] (3,3) -- (3.3,2.7);
			\draw[thick,->] (3,4) -- (3,3.6);
			\draw[thick,->] (4,4) -- (4,3.6);
			\end{tikzpicture}
		} %
		
		\caption{}
		\label{fig:unwedge1}
	\end{figure}
	
	\textit{Case} 2: The curve $C'$ might be a union of wedges of several circles with additional arcs connecting them.  The resulting boundary of $C'$ is disconnected, i.e., it is the disjoint union of curves such that these curves are connected via interior additional arcs in the tunnel.  Firstly, we will obtain a connected boundary curve via connecting disjoint bounday components by pushing the interior arcs of the tunnel outside after doing some subdivisions as in the following explanation.

	Let $A=\displaystyle  \vee_{i}   c_{i}$ be the wedge of the circles $c_{i}$ at the point $v$, for $i=1,2,\ldots,n$ and $B=\displaystyle  \vee_{j}   c_{j}'$ be the 
	wedge of the circles $c_{j}'$ at the point $v'$, for $j=1,2,\ldots,m$.  Let $a$ be the additional arc between $A$ and $B$ with the boundary vertices $v$ and 
	$v'$  such that $A\cap \{a\}=\{v\}$ and $B\cap \{a\}=\{v'\}$.  Assume that the circles in $A$ and $B$ are ordered as 
	$a, c_{1}, c_{2}, \ldots, c_{n}$ and $a, c_{1}', c_{2}', \ldots, c_{m}'$ respectively(see Figure \ref{fig:connecting edge}).
	
	Let $v_{i1}, v_{i2}$ and $v_{j1}', v_{j2}'$ be two of the vertices on each circle $c_{i}$ and $c_{j}'$ respectively such that they are in the link $ Lk_{v}(A)$ and $ Lk_{v'}(B)$ respectively, and they are ordered as $v, v_{i1}, v_{i2}$ and 
	$v', v_{j1}', v_{j2}'$. 
	To push the additional arc outside of the tunnel, we take two vertices 
	$v_{11},v_{n2}\in Lk_{v}(A) $ and two vertices $v_{11}',v_{m2}'\in Lk_{v'}(B)$, and then, we connect the vertices $v_{11}$ with $v_{m2}'$ and $v_{n2}$ with 
	$v_{11}'$ by arcs which are a union of edges obtained by doing a sequence of bisections in the tunnels.  We take these arcs as a part of the boundary curve 
	instead of the arc $a$ (see Figure \ref{fig:connecting edge} where $n=2, m=2, t_{i}=3, p_{j}=3 $).  
	Note that at the end of this process the arc $a$ will lie outside of the tunnel and we repeat the process above for all the additional arcs to obtain a connected boundary curve.
	
	\begin{figure}[h]
		\resizebox{0.47\textwidth}{!}{%
			\begin{tikzpicture}[scale=2]
			\draw (0,2)--(0,4);
			\draw (0,6)--(4,6);
			\draw (4,2)--(4,4);
			\draw (0,0)--(4,0);
			\draw (2,4) -- (0,6);
			\draw (2,4) -- (4,6);
			\draw (2,2) -- (0,0);
			\draw (2,2) -- (4,0);
			\draw (0,2) -- (4,2);
			\draw (0,4) -- (4,4);
			\draw (0,3) -- (2,4);
			\draw (0,3) -- (2,2);
			\draw (4,3) -- (2,2);
			\draw (4,3) -- (2,4);
			\draw[fill=black!20] (0,6)--(2,4)--(4,6)--(0,6);
			\draw[fill=black!20] (0,0)--(2,2)--(4,0)--(0,0);
			\draw[fill=black!20] (0,2)--(2,2)--(0,3)--(0,2);
			\draw[fill=black!20] (0,3)--(2,2)--(2,4)--(0,3);
			\draw[fill=black!20] (0,3)--(2,4)--(0,4)--(0,3);
			\draw[fill=black!20] (2,2)--(4,2)--(4,3)--(2,2);
			\draw[fill=black!20] (2,2)--(2,4)--(4,3)--(2,2);
			\draw[fill=black!20] (4,3)--(2,4)--(4,4)--(4,3);
			\node[inner sep=0,anchor=west,text width=3.3cm] (note1) at (2.1,3) {$a$};
			\node[inner sep=0,anchor=west,text width=3.3cm] (note1) at (1.9,1.55) {$v'$};
			\node[inner sep=0,anchor=west,text width=3.3cm] (note1) at (1.97,4.3) {$v$};
			\node[inner sep=0,anchor=west,text width=3.3cm] (note1) at (3.4,4.8) {$c_{1}$};
			\node[inner sep=0,anchor=west,text width=3.3cm] (note1) at (0.5,4.8) {$c_{2}$};
			\node[inner sep=0,anchor=west,text width=3.3cm] (note1) at (0.5,1.2) {$c_{1}'$};
			\node[inner sep=0,anchor=west,text width=3.3cm] (note1) at (3.4,1.2) {$c_{2}'$};
			\draw[dotted] [line width=0.5mm](0,0) -- (2,2);
			\draw[dotted] [line width=0.5mm](4,0) -- (2,2);
			\draw[dotted] [line width=0.5mm](0,0) -- (0,2);
			\draw[dotted] [line width=0.5mm](4,2) -- (4,0);
			\draw[dotted] [line width=0.5mm](0,2) -- (4,2);
			\draw[dotted] [line width=0.5mm](2,2) -- (2,4);
			\draw[dotted] [line width=0.5mm](0,4) -- (4,4);
			\draw[dotted] [line width=0.5mm](0,4) -- (0,6);
			\draw[dotted] [line width=0.5mm](4,4) -- (4,6);
			\draw[dotted] [line width=0.5mm](0,6) -- (2,4);
			\draw[dotted] [line width=0.5mm](4,6) -- (2,4);
			\draw (0,0) node[circle,fill,inner sep=2pt] {};
			\draw (0,6) node[circle,fill,inner sep=2pt] {};
			\draw (4,6) node[circle,fill,inner sep=2pt] {};
			\draw (4,0) node[circle,fill,inner sep=2pt] {};
			\draw (0,2) node[circle,fill,inner sep=2pt] {};
			\draw (4,2) node[circle,fill,inner sep=2pt] {};
			\draw (4,4) node[circle,fill,inner sep=2pt] {};
			\draw (2,2) node[circle,fill,inner sep=2pt] {};
			\draw (2,4) node[circle,fill,inner sep=2pt] {};
			\draw (0,3) node[circle,fill,inner sep=2pt] {};
			\draw (4,3) node[circle,fill,inner sep=2pt] {};
			\draw (1,2) node[circle,fill,inner sep=2pt] {};
			\draw (1,4) node[circle,fill,inner sep=2pt] {};
			\draw (3,2) node[circle,fill,inner sep=2pt] {};
			\draw (3,4) node[circle,fill,inner sep=2pt] {};
			\node[inner sep=0,anchor=west,text width=3.3cm] (note1) at (4.2,0) {$v_{21}'$};
			\node[inner sep=0,anchor=west,text width=3.3cm] (note1) at (-0.6,0) {$v_{12}'$};
			\node[inner sep=0,anchor=west,text width=3.3cm] (note1) at (0.9,1.7) {$v_{11}'$};
			\node[inner sep=0,anchor=west,text width=3.3cm] (note1) at (0.9,4.3) {$v_{22}$};
			\node[inner sep=0,anchor=west,text width=3.3cm] (note1) at (-0.6,6) {$v_{21}$};
			\node[inner sep=0,anchor=west,text width=3.3cm] (note1) at (2.9,1.7) {$v_{22}'$};
			\node[inner sep=0,anchor=west,text width=3.3cm] (note1) at (2.9,4.3) {$v_{11}$};
			\node[inner sep=0,anchor=west,text width=3.3cm] (note1) at (4.2,6) {$v_{12}$};
			\draw[thick,->] (3,5) -- (3.3,4.7);
			\draw[thick,->] (1,2) -- (1.3,2);
			\draw[thick,->] (3,2) -- (2.7,2);
			\draw[thick,->] (1,4) -- (1.3,4);
			\draw[thick,->] (3,4) -- (3.3,4);
			\draw[thick,->] (0,0) -- (0.3,0.3);
			\draw[thick,->] (2,0) -- (2,0.3);
			\draw[thick,->] (0,0) -- (0.3,0.3);
			\draw[thick,->] (3,1) -- (3.3,1.3);
			\draw[thick,->] (4,0) -- (4,0.3);
			\draw[thick,->] (4,2) -- (3.7,2);
			\draw[thick,->] (0,2) -- (0.3,2);
			\draw[thick,->] (0,1) -- (0.3,1);
			\draw[thick,->] (0,2.5) -- (0.3,2.5);
			\draw[thick,->] (4,2.5) -- (3.7,2.5);
			\draw[thick,->] (2,2) -- (2,2.3);
			\draw[thick,->] (0,3.5) -- (-0.3,3.5);
			\draw[thick,->] (0,4) -- (0.3,4);
			\draw[thick,->] (2,4) -- (2.3,4);
			\draw[thick,->] (4,4) -- (4,4.3);
			\draw[thick,->] (4,6) -- (3.9,5.7);
			\draw[thick,->] (2,6) -- (2,5.7);
			\draw[thick,->] (0,5) -- (0.3,5);
			\draw[thick,->] (0,6) -- (0.3,5.7);
			\draw[thick,->] (0.75,3.38) -- (0.5,3.65);
			\draw[thick,->] (0,3) -- (-0.3,3);
			\draw[thick,->] (1,2.5) -- (1.2,2.8);
			\draw[thick,->] (3,2.5) -- (2.8,2.8);
			\draw[thick,->] (3.35,3.33) -- (3.59,3.62);
			\draw[thick,->] (4,3) -- (4.3,3);
			\draw[thick,->] (4,3.5) -- (4.3,3.5);
			\end{tikzpicture}
		} %
		\qquad
		\resizebox{0.47\textwidth}{!}{%
			\begin{tikzpicture}[scale=2]
			\draw (0,2)--(0,4);
			\draw (0,6)--(4,6);
			\draw (4,2)--(4,4);
			\draw (0,0)--(4,0);
			\draw (2,4) -- (0,6);
			\draw (2,4) -- (4,6);
			\draw (2,2) -- (0,0);
			\draw (2,2) -- (4,0);
			\draw (0,2) -- (4,2);
			\draw (0,4) -- (4,4);
			\draw (0,3) -- (2,4);
			\draw (0,3) -- (2,2);
			\draw (4,3) -- (2,2);
			\draw (4,3) -- (2,4);
			\draw[fill=black!20] (0,6)--(2,4)--(4,6)--(0,6);
			\draw[fill=black!20] (0,0)--(2,2)--(4,0)--(0,0);
			\draw[fill=black!20] (0,2)--(2,2)--(0,3)--(0,2);
			\draw[fill=black!20] (0,3)--(2,2)--(2,4)--(0,3);
			\draw[fill=black!20] (0,3)--(2,4)--(0,4)--(0,3);
			\draw[fill=black!20] (2,2)--(4,2)--(4,3)--(2,2);
			\draw[fill=black!20] (2,2)--(2,4)--(4,3)--(2,2);
			\draw[fill=black!20] (4,3)--(2,4)--(4,4)--(4,3);
			\node[inner sep=0,anchor=west,text width=3.3cm] (note1) at (2.1,3) {$a$};
			\node[inner sep=0,anchor=west,text width=3.3cm] (note1) at (1.9,1.55) {$v'$};
			\node[inner sep=0,anchor=west,text width=3.3cm] (note1) at (1.97,4.3) {$v$};
			\node[inner sep=0,anchor=west,text width=3.3cm] (note1) at (3.4,4.8) {$c_{1}$};
			\node[inner sep=0,anchor=west,text width=3.3cm] (note1) at (0.5,4.8) {$c_{2}$};
			\node[inner sep=0,anchor=west,text width=3.3cm] (note1) at (0.5,1.2) {$c_{1}'$};
			\node[inner sep=0,anchor=west,text width=3.3cm] (note1) at (3.4,1.2) {$c_{2}'$};
			\draw [dotted] [line width=0.5mm](0,0) -- (2,2);
			\draw [dotted][line width=0.5mm](4,0) -- (2,2);
			\draw[dotted] [line width=0.5mm](0,0) -- (0,2);
			\draw[dotted] [line width=0.5mm](4,2) -- (4,0);
			\draw [dotted][line width=0.5mm](0,2) -- (4,2);
			\draw[dotted] [line width=0.5mm](2,2) -- (2,4);
			\draw[dotted] [line width=0.5mm](0,4) -- (4,4);
			\draw[dotted] [line width=0.5mm](0,4) -- (0,6);
			\draw[dotted] [line width=0.5mm](4,4) -- (4,6);
			\draw[dotted] [line width=0.5mm](0,6) -- (2,4);
			\draw[dotted] [line width=0.5mm](4,6) -- (2,4);
			\draw (0,0) node[circle,fill,inner sep=2pt] {};
			\draw (0,6) node[circle,fill,inner sep=2pt] {};
			\draw (4,6) node[circle,fill,inner sep=2pt] {};
			\draw (4,0) node[circle,fill,inner sep=2pt] {};
			\draw (0,2) node[circle,fill,inner sep=2pt] {};
			\draw (4,2) node[circle,fill,inner sep=2pt] {};
			\draw (4,4) node[circle,fill,inner sep=2pt] {};
			\draw (2,2) node[circle,fill,inner sep=2pt] {};
			\draw (2,4) node[circle,fill,inner sep=2pt] {};
			\draw (0,3) node[circle,fill,inner sep=2pt] {};
			\draw (4,3) node[circle,fill,inner sep=2pt] {};
			\draw (1,2) node[circle,fill,inner sep=2pt] {};
			\draw (1.5,2.25) node[circle,fill,inner sep=2pt] {};
			\draw (1.5,3.75) node[circle,fill,inner sep=2pt] {};
			\draw (1,4) node[circle,fill,inner sep=2pt] {};
			\draw (3,2) node[circle,fill,inner sep=2pt] {};
			\draw (2.5,2.25) node[circle,fill,inner sep=2pt] {};
			\draw (2.5,3.75) node[circle,fill,inner sep=2pt] {};
			\draw (3,4) node[circle,fill,inner sep=2pt] {};
						\node[inner sep=0,anchor=west,text width=3.3cm] (note1) at (4.2,0) {$v_{21}'$};
						\node[inner sep=0,anchor=west,text width=3.3cm] (note1) at (-0.6,0) {$v_{12}'$};
						\node[inner sep=0,anchor=west,text width=3.3cm] (note1) at (0.9,1.7) {$v_{11}'$};
						\node[inner sep=0,anchor=west,text width=3.3cm] (note1) at (0.9,4.3) {$v_{22}$};
						\node[inner sep=0,anchor=west,text width=3.3cm] (note1) at (-0.6,6) {$v_{21}$};
						\node[inner sep=0,anchor=west,text width=3.3cm] (note1) at (2.9,1.7) {$v_{22}'$};
						\node[inner sep=0,anchor=west,text width=3.3cm] (note1) at (2.9,4.3) {$v_{11}$};
						\node[inner sep=0,anchor=west,text width=3.3cm] (note1) at (4.2,6) {$v_{12}$};
			\draw (1,2) -- (1.5,2.25);
			\draw (1.5,2.25) -- (1.5,3.75);
			\draw (1.5,3.75) -- (1,4);
			\draw (3,2) -- (2.5,2.25);
			\draw (2.5,2.25) -- (2.5,3.75);
			\draw (2.5,3.75) -- (3,4);
			\draw[thick,->] (3,5) -- (3.3,4.7);
			\draw[thick,->] (1,2) -- (1.3,2);
			\draw[thick,->] (3,2) -- (2.7,2);
			\draw[thick,->] (1,4) -- (1.3,4);
			\draw[thick,->] (3,4) -- (3.3,4);
			\draw[thick,->] (0,0) -- (0.3,0.3);
			\draw[thick,->] (2,0) -- (2,0.3);
			\draw[thick,->] (0,0) -- (0.3,0.3);
			\draw[thick,->] (3,1) -- (3.3,1.3);
			\draw[thick,->] (4,0) -- (4,0.3);
			\draw[thick,->] (4,2) -- (3.7,2);
			\draw[thick,->] (0,2) -- (0.3,2);
			\draw[thick,->] (0,1) -- (0.3,1);
			\draw[thick,->] (0,2.5) -- (0.3,2.5);
			\draw[thick,->] (4,2.5) -- (3.7,2.5);
			\draw[thick,->] (2,2) -- (2,2.3);
			\draw[thick,->] (0,3.5) -- (-0.3,3.5);
			\draw[thick,->] (0,4) -- (0.3,4);
			\draw[thick,->] (2,4) -- (2.3,4);
			\draw[thick,->] (4,4) -- (4,4.3);
			\draw[thick,->] (4,6) -- (3.9,5.7);
			\draw[thick,->] (2,6) -- (2,5.7);
			\draw[thick,->] (0,5) -- (0.3,5);
			\draw[thick,->] (0,6) -- (0.3,5.7);
			\draw[thick,->] (0.75,3.38) -- (0.5,3.65);
			\draw[thick,->] (0,3) -- (-0.3,3);
			\draw[thick,->] (1,2.5) -- (1.2,2.8);
			\draw[thick,->] (3,2.5) -- (2.8,2.8);
			\draw[thick,->] (3.35,3.33) -- (3.59,3.62);
			\draw[thick,->] (4,3) -- (4.3,3);
			\draw[thick,->] (4,3.5) -- (4.3,3.5);
			\end{tikzpicture}
		} %
		\caption{}
		\label{fig:connecting edge}
	\end{figure}

	\begin{figure}[h]
		\resizebox{0.47\textwidth}{!}{%
			\begin{tikzpicture}[scale=2]
			\draw (0,2)--(0,4);
			\draw (0,6)--(4,6);
			\draw (4,2)--(4,4);
			\draw (0,0)--(4,0);
			\draw (2,4) -- (0,6);
			\draw (2,4) -- (4,6);
			\draw (2,2) -- (0,0);
			\draw (2,2) -- (4,0);
			\draw (0,2) -- (4,2);
			\draw (0,4) -- (4,4);
			\draw (2,2) -- (2,4);
			\draw (0,3) -- (2,4);
			\draw (0,3) -- (2,2);
			\draw (4,3) -- (2,2);
			\draw (4,3) -- (2,4);
			\draw[fill=black!20] (0,6)--(2,4)--(4,6)--(0,6);
			\draw[fill=black!20] (0,0)--(2,2)--(4,0)--(0,0);
			\draw[fill=black!20] (0,2)--(1,2)--(1.5,2.25)--(0,3)--(0,2);
			\draw[fill=black!20] (0,3)--(1.5,2.25)--(1.5,3.75)--(0,3);
			\draw[fill=black!20] (0,3)--(1.5,3.75)--(1,4)--(0,4)--(0,3);
			\draw[fill=black!20] (3,2)--(4,2)--(4,3)--(2.5,2.25)--(3,2);
			\draw[fill=black!20] (2.5,2.25)--(4,3)--(2.5,3.75)--(2.5,2.25);
			\draw[fill=black!20] (4,3)--(2.5,3.75)--(3,4)--(4,4)--(4,3);
			\node[inner sep=0,anchor=west,text width=3.3cm] (note1) at (1.9,1.55) {$v'$};
			\node[inner sep=0,anchor=west,text width=3.3cm] (note1) at (1.97,4.3) {$v$};
			\node[inner sep=0,anchor=west,text width=3.3cm] (note1) at (3.4,4.8) {$c_{1}$};
			\node[inner sep=0,anchor=west,text width=3.3cm] (note1) at (0.5,4.8) {$c_{2}$};
			\node[inner sep=0,anchor=west,text width=3.3cm] (note1) at (0.5,1.2) {$c_{1}'$};
			\node[inner sep=0,anchor=west,text width=3.3cm] (note1) at (3.4,1.2) {$c_{2}'$};
			\node[inner sep=0,anchor=west,text width=3.3cm] (note1) at (2.1,3) {$a$};
			\node[inner sep=0,anchor=west,text width=3.3cm] (note1) at (1,2.2) {$e_{1}$};
			\node[inner sep=0,anchor=west,text width=3.3cm] (note1) at (1.2,3) {$e_{2}$};
			\node[inner sep=0,anchor=west,text width=3.3cm] (note1) at (0.9,3.8) {$e_{3}$};
			\node[inner sep=0,anchor=west,text width=3.3cm] (note1) at (2.8,2.2) {$e_{4}$};
			\node[inner sep=0,anchor=west,text width=3.3cm] (note1) at (2.5,3) {$e_{5}$};
			\node[inner sep=0,anchor=west,text width=3.3cm] (note1) at (2.73,3.75) {$e_{6}$};
			\draw [dotted][line width=0.5mm](0,0) -- (2,2);
			\draw[dotted] [line width=0.5mm](4,0) -- (2,2);
			\draw[dotted] [line width=0.5mm](0,0) -- (0,2);
			\draw [dotted][line width=0.5mm](4,2) -- (4,0);
			\draw [dotted][line width=0.5mm](0,2) -- (1,2);
			\draw [dotted][line width=0.5mm](3,2) -- (4,2);
			\draw [dotted][line width=0.5mm](0,4) -- (1,4);
			\draw [dotted][line width=0.5mm](3,4) -- (4,4);
			\draw[dotted] [line width=0.5mm](0,4) -- (0,6);
			\draw[dotted] [line width=0.5mm](4,4) -- (4,6);
			\draw[dotted] [line width=0.5mm](0,6) -- (2,4);
			\draw[dotted] [line width=0.5mm](4,6) -- (2,4);
			\draw (0,0) node[circle,fill,inner sep=2pt] {};
			\draw (0,6) node[circle,fill,inner sep=2pt] {};
			\draw (4,6) node[circle,fill,inner sep=2pt] {};
			\draw (4,0) node[circle,fill,inner sep=2pt] {};
			\draw (0,2) node[circle,fill,inner sep=2pt] {};
			\draw (4,2) node[circle,fill,inner sep=2pt] {};
			\draw (4,4) node[circle,fill,inner sep=2pt] {};
			\draw (2,2) node[circle,fill,inner sep=2pt] {};
			\draw (2,4) node[circle,fill,inner sep=2pt] {};
			\draw (0,3) node[circle,fill,inner sep=2pt] {};
			\draw (4,3) node[circle,fill,inner sep=2pt] {};
			\draw [dotted] [line width=0.5mm](1,2) -- (1.5,2.25);
			\draw [dotted][line width=0.5mm](1.5,2.25) -- (1.5,3.75);
			\draw [dotted][line width=0.5mm](1.5,3.75) -- (1,4);
			\draw[dotted] [line width=0.5mm](3,2) -- (2.5,2.25);
			\draw [dotted][line width=0.5mm](2.5,2.25) -- (2.5,3.75);
			\draw [dotted][line width=0.5mm](2.5,3.75) -- (3,4);
			\draw (1,2) node[circle,fill,inner sep=2pt] {};
			\draw (1.5,2.25) node[circle,fill,inner sep=2pt] {};
			\draw (1.5,3.75) node[circle,fill,inner sep=2pt] {};
			\draw (1,4) node[circle,fill,inner sep=2pt] {};
			\draw (3,2) node[circle,fill,inner sep=2pt] {};
			\draw (2.5,2.25) node[circle,fill,inner sep=2pt] {};
			\draw (2.5,3.75) node[circle,fill,inner sep=2pt] {};
			\draw (3,4) node[circle,fill,inner sep=2pt] {};
					\node[inner sep=0,anchor=west,text width=3.3cm] (note1) at (4.2,0) {$v_{21}'$};
					\node[inner sep=0,anchor=west,text width=3.3cm] (note1) at (-0.6,0) {$v_{12}'$};
					\node[inner sep=0,anchor=west,text width=3.3cm] (note1) at (0.9,1.7) {$v_{11}'$};
					\node[inner sep=0,anchor=west,text width=3.3cm] (note1) at (0.9,4.3) {$v_{22}$};
					\node[inner sep=0,anchor=west,text width=3.3cm] (note1) at (-0.6,6) {$v_{21}$};
					\node[inner sep=0,anchor=west,text width=3.3cm] (note1) at (2.9,1.7) {$v_{22}'$};
					\node[inner sep=0,anchor=west,text width=3.3cm] (note1) at (2.9,4.3) {$v_{11}$};
					\node[inner sep=0,anchor=west,text width=3.3cm] (note1) at (4.2,6) {$v_{12}$};
			\draw[thick,->] (3,5) -- (3.3,4.7);
			\draw[thick,->] (1,2) -- (1.3,2);
			\draw[thick,->] (3,2) -- (2.7,2);
			\draw[thick,->] (1,4) -- (1.3,4);
			\draw[thick,->] (3,4) -- (3.3,4);
			\draw[thick,->] (0,0) -- (0.3,0.3);
			\draw[thick,->] (2,0) -- (2,0.3);
			\draw[thick,->] (0,0) -- (0.3,0.3);
			\draw[thick,->] (3,1) -- (3.3,1.3);
			\draw[thick,->] (4,0) -- (4,0.3);
			\draw[thick,->] (4,2) -- (3.7,2);
			\draw[thick,->] (0,2) -- (0.3,2);
			\draw[thick,->] (0,1) -- (0.3,1);
			\draw[thick,->] (0,2.5) -- (0.3,2.5);
			\draw[thick,->] (4,2.5) -- (3.7,2.5);
			\draw[thick,->] (2,2) -- (2,2.3);
			\draw[thick,->] (0,3.5) -- (-0.3,3.5);
			\draw[thick,->] (0,4) -- (0.3,4);
			\draw[thick,->] (2,4) -- (2.3,4);
			\draw[thick,->] (4,4) -- (4,4.3);
			\draw[thick,->] (4,6) -- (3.9,5.7);
			\draw[thick,->] (2,6) -- (2,5.7);
			\draw[thick,->] (0,5) -- (0.3,5);
			\draw[thick,->] (0,6) -- (0.3,5.7);
			\draw[thick,->] (0.75,3.38) -- (0.5,3.65);
			\draw[thick,->] (0,3) -- (-0.3,3);
			\draw[thick,->] (1,2.5) -- (1.2,2.8);
			\draw[thick,->] (3,2.5) -- (2.8,2.8);
			\draw[thick,->] (3.35,3.33) -- (3.59,3.62);
			\draw[thick,->] (4,3) -- (4.3,3);
			\draw[thick,->] (4,3.5) -- (4.3,3.5);
			\end{tikzpicture}
		} %
		\qquad
		\resizebox{0.47\textwidth}{!}{%
			\begin{tikzpicture}[scale=2]
			\draw (0,2)--(0,4);
			\draw (0,6)--(4,6);
			\draw (4,2)--(4,4);
			\draw (0,0)--(4,0);
			\draw (2,4) -- (0,6);
			\draw (2,4) -- (4,6);
			\draw (2,2) -- (0,0);
			\draw (2,2) -- (4,0);
			\draw (0,2) -- (4,2);
			\draw (0,4) -- (4,4);
			\draw (2,2) -- (2,4);
			\draw (0,3) -- (2,4);
			\draw (0,3) -- (2,2);
			\draw (4,3) -- (2,2);
			\draw (4,3) -- (2,4);
			\draw[fill=black!20] (0,6)--(2,4)--(4,6)--(0,6);
			\draw[fill=black!20] (0,0)--(2,2)--(4,0)--(0,0);
			\draw[fill=black!20] (0,2)--(1,2)--(1.5,2.25)--(0,3)--(0,2);
			\draw[fill=black!20] (0,3)--(1.5,2.25)--(1.5,3.75)--(0,3);
			\draw[fill=black!20] (0,3)--(1.5,3.75)--(1,4)--(0,4)--(0,3);
			\draw[fill=black!20] (3,2)--(4,2)--(4,3)--(2.5,2.25)--(3,2);
			\draw[fill=black!20] (2.5,2.25)--(4,3)--(2.5,3.75)--(2.5,2.25);
			\draw[fill=black!20] (4,3)--(2.5,3.75)--(3,4)--(4,4)--(4,3);
			\node[inner sep=0,anchor=west,text width=3.3cm] (note1) at (1.9,1.55) {$v'$};
			\node[inner sep=0,anchor=west,text width=3.3cm] (note1) at (1.97,4.3) {$v$};
			\node[inner sep=0,anchor=west,text width=3.3cm] (note1) at (3.4,4.8) {$c_{1}$};
			\node[inner sep=0,anchor=west,text width=3.3cm] (note1) at (0.5,4.8) {$c_{2}$};
			\node[inner sep=0,anchor=west,text width=3.3cm] (note1) at (0.5,1.2) {$c_{1}'$};
			\node[inner sep=0,anchor=west,text width=3.3cm] (note1) at (3.4,1.2) {$c_{2}'$};
			\node[inner sep=0,anchor=west,text width=3.3cm] (note1) at (2.02,3.3) {$a$};
			\node[inner sep=0,anchor=west,text width=3.3cm] (note1) at (1,2.2) {$e_{1}$};
			\node[inner sep=0,anchor=west,text width=3.3cm] (note1) at (1.2,3) {$e_{2}$};
			\node[inner sep=0,anchor=west,text width=3.3cm] (note1) at (0.9,3.8) {$e_{3}$};
			\node[inner sep=0,anchor=west,text width=3.3cm] (note1) at (2.8,2.2) {$e_{4}$};
			\node[inner sep=0,anchor=west,text width=3.3cm] (note1) at (2.5,3) {$e_{5}$};
			\node[inner sep=0,anchor=west,text width=3.3cm] (note1) at (2.73,3.75) {$e_{6}$};
			\draw [dotted] [line width=0.5mm](0,0) -- (2,2);
			\draw [dotted] [line width=0.5mm](4,0) -- (2,2);
			\draw [dotted] [line width=0.5mm](0,0) -- (0,2);
			\draw [dotted] [line width=0.5mm](4,2) -- (4,0);
			\draw [dotted] [line width=0.5mm](0,2) -- (1,2);
			\draw [dotted] [line width=0.5mm](3,2) -- (4,2);
			\draw [dotted][line width=0.5mm](0,4) -- (1,4);
			\draw [dotted][line width=0.5mm](3,4) -- (4,4);
			\draw [dotted] [line width=0.5mm](0,4) -- (0,6);
			\draw [dotted] [line width=0.5mm](4,4) -- (4,6);
			\draw [dotted] [line width=0.5mm](0,6) -- (2,4);
			\draw [dotted] [line width=0.5mm](4,6) -- (2,4);
			\draw (0,0) node[circle,fill,inner sep=2pt] {};
			\draw (0,6) node[circle,fill,inner sep=2pt] {};
			\draw (4,6) node[circle,fill,inner sep=2pt] {};
			\draw (4,0) node[circle,fill,inner sep=2pt] {};
			\draw (0,2) node[circle,fill,inner sep=2pt] {};
			\draw (4,2) node[circle,fill,inner sep=2pt] {};
			\draw (4,4) node[circle,fill,inner sep=2pt] {};
			\draw (2,2) node[circle,fill,inner sep=2pt] {};
			\draw (2,4) node[circle,fill,inner sep=2pt] {};
			\draw (0,3) node[circle,fill,inner sep=2pt] {};
			\draw (4,3) node[circle,fill,inner sep=2pt] {};
			\draw[dotted] [line width=0.5mm](1,2) -- (1.5,2.25);
			\draw[dotted] [line width=0.5mm](1.5,2.25) -- (1.5,3.75);
			\draw[dotted] [line width=0.5mm](1.5,3.75) -- (1,4);
			\draw [dotted][line width=0.5mm](3,2) -- (2.5,2.25);
			\draw[dotted] [line width=0.5mm](2.5,2.25) -- (2.5,3.75);
			\draw [dotted][line width=0.5mm](2.5,3.75) -- (3,4);
			\draw (1,2) node[circle,fill,inner sep=2pt] {};
			\draw (1.5,2.25) node[circle,fill,inner sep=2pt] {};
			\draw (1.5,3.75) node[circle,fill,inner sep=2pt] {};
			\draw (1,4) node[circle,fill,inner sep=2pt] {};
			\draw (3,2) node[circle,fill,inner sep=2pt] {};
			\draw (2.5,2.25) node[circle,fill,inner sep=2pt] {};
			\draw (2.5,3.75) node[circle,fill,inner sep=2pt] {};
			\draw (3,4) node[circle,fill,inner sep=2pt] {};
					\node[inner sep=0,anchor=west,text width=3.3cm] (note1) at (4.2,0) {$v_{21}'$};
					\node[inner sep=0,anchor=west,text width=3.3cm] (note1) at (-0.6,0) {$v_{12}'$};
					\node[inner sep=0,anchor=west,text width=3.3cm] (note1) at (0.9,1.7) {$v_{11}'$};
					\node[inner sep=0,anchor=west,text width=3.3cm] (note1) at (0.9,4.3) {$v_{22}$};
					\node[inner sep=0,anchor=west,text width=3.3cm] (note1) at (-0.6,6) {$v_{21}$};
					\node[inner sep=0,anchor=west,text width=3.3cm] (note1) at (2.9,1.7) {$v_{22}'$};
					\node[inner sep=0,anchor=west,text width=3.3cm] (note1) at (2.9,4.3) {$v_{11}$};
					\node[inner sep=0,anchor=west,text width=3.3cm] (note1) at (4.2,6) {$v_{12}$};
			\draw[thick,->] (1,2) -- (1.3,2);
			\draw[thick,->] (3,2) -- (2.7,2);
			\draw[thick,->] (1,4) -- (1.3,4);
			\draw[thick,->] (3,4) -- (3.3,4);
			\draw[thick,->] (0,0) -- (0.3,0.3);
			\draw[thick,->] (2,0) -- (2,0.3);
			\draw[thick,->] (0,0) -- (0.3,0.3);
			\draw[thick,->] (3,1) -- (3.3,1.3);
			\draw[thick,->] (4,0) -- (4,0.3);
			\draw[thick,->] (4,2) -- (3.7,2);
			\draw[thick,->] (0,2) -- (0.3,2);
			\draw[thick,->] (0,1) -- (0.3,1);
			\draw[thick,->] (0,2.5) -- (0.3,2.5);
			\draw[thick,->] (4,2.5) -- (3.7,2.5);
			\draw[thick,->] (2,2) -- (2,2.3);
			\draw[thick,->] (0,3.5) -- (-0.3,3.5);
			\draw[thick,->] (0,4) -- (0.3,4);
			\draw[thick,->] (2,4) -- (2.3,4);
			\draw[thick,->] (4,4) -- (4,4.3);
			\draw[thick,->] (4,6) -- (3.9,5.7);
			\draw[thick,->] (2,6) -- (2,5.7);
			\draw[thick,->] (0,5) -- (0.3,5);
			\draw[thick,->] (0,6) -- (0.3,5.7);
			\draw[thick,->] (0.75,3.38) -- (0.5,3.65);
			\draw[thick,->] (0,3) -- (-0.3,3);
			\draw[thick,->] (1,2.5) -- (1.2,2.8);
			\draw[thick,->] (3,2.5) -- (2.8,2.8);
			\draw[thick,->] (3.35,3.33) -- (3.59,3.62);
			\draw[thick,->] (4,3) -- (4.3,3);
			\draw[thick,->] (4,3.5) -- (4.3,3.5);
			\draw[thick,->] (1,2) -- (1.3,2);
			\draw[thick,->] (3,2) -- (2.7,2);
			\draw[thick,dotted,->][line width=0.5mm] (1.21,2.1) -- (1.57,2.1);
			\draw[thick,dotted,->][line width=0.5mm] (1.5,2.25) -- (1.9,2.05);
			\draw[thick,dotted,->][line width=0.5mm] (1.5,3) -- (1.9,3);
			\draw[thick,dotted,->][line width=0.5mm] (1.5,3.75) -- (1.9,3.95);
			\draw[thick,->] (1,4) -- (1.3,4);
			\draw[thick,dotted,->][line width=0.5mm] (1.24,3.86) -- (1.52,3.9);
			\draw[thick,dotted,->][line width=0.5mm] (2.8,2.1) -- (2.47,2.1);
			\draw[thick,dotted,->][line width=0.5mm] (2.5,2.25) -- (2.1,2.05);
			\draw[thick,dotted,->][line width=0.5mm] (2.5,3) -- (2.1,3);
			\draw[thick,dotted,->][line width=0.5mm] (2.5,3.75) -- (2.1,3.95);
			\draw[thick,dotted,->][line width=0.5mm] (2.7,3.86) -- (2.48,3.92);
			\draw[thick,->] (3,4) -- (3.3,4);
			\end{tikzpicture}
		} %
		\caption{}
		\label{fig:circle}
	\end{figure}
	
	In Figure \ref{fig:circle}, the union of the edges $e_{1}, e_{2}, e_{3}$ and that of the edges $e_{4}, e_{5}, e_{6}$ are a part of the boundary curve taken instead of the arc $a$.
	
	After eliminating all such additional arcs, the curve is either a circle or a single wedge of circles or wedges of circles attached to some additional circles at wedge points.
	If it is a circle, we have succeed our aim.  If it is a single wedge of circles, then we use the method in \textit{Case} 1 to obtain a single circle.  
	
	\textit{Case} 3: Assume the curve is made up of two wedges of circles attached to an additional circle at wedge points of each wedge.  Let $E=\displaystyle  \vee_{i}   d_{i}$ be the wedge of the circles $d_{i}$ at the point $v$ 
	for $i=1,2,\ldots,n$ and $F=\displaystyle  \vee_{j}   d_{j}'$ be the wedge of the  circles $d_{j}'$ at the point $v'$ for $j=1,2,\ldots,m$.  
	Let $d_{0}$ be the additional circle connecting  $E$ and $F$ such that $E\cap d_{0}=\{v\}$ and $F\cap d_{0}=\{v'\}$.  Except for one wedge point, we reduce the number of circles passing through each wedge point to one to obtain a single wedge of circles as the boundary curve by using the method given in \textit{Case} 1.  In here, we reduce the number of circles passing from $v$ on $E$ to one without losing generality.  Let $d_{0}'$ be the resulting circle such that $v\in d_{0}'$ and $F\cap d_{0}'=\{v'\}$.  Now, the boundary curve is a wedge of circles at the vertex $v'$, see Figures \ref{fig:wedge2} and \ref{fig:unwedge2} where $n=2, m=2, t_{i}=3 $.  Therefore, we can use the method explained in \textit{Case} 1 to obtain a circle as the boundary curve.  Note that if the boundary curve has more than two wedges of circles and more than one additional circle after eliminating the additional interior arcs, we repeat the process given in \textit{Case} 3 consecutively until obtai
 ning a s
 ingle circle as the boundary curve.

  the pro
 cess given in \textit{Case} 3 consecutively up to obtaining a single circle as the boundary curve.


	The boundary curve $C'$ in Figure \ref{fig:unwedge2} is a wedge of $3$-circles at the point $v'$.  
	The gray regions are a part of the tunnels, while the white region is a part of the area outside of the tunnel, and the dotted 
	curve is a part of the boundary curve.  The operations described in this proof correspond to pushing the boundary curve into $M-M_{1}$ at the wedge 
	points $v$ and $v'$.   
	
	\begin{figure}[h]
		\centering
		\resizebox{0.45\textwidth}{!}{%
			\begin{tikzpicture}[scale=2]
			\draw (0,2)--(0,4);
			\draw (2,4)--(4,4);
			\draw (4,0)--(0,0);
			\draw (2,0) -- (2,2);
			\draw (2,2.7) -- (2,2);
			\draw (2,2) -- (2.7,2);
			\draw (6,0) -- (6,2);
			\draw (6,4) -- (4,4);
			\draw[fill=black!20] (0,0)--(2,0)--(2,2)--(0,0);
			\draw[fill=black!20] (2,2)--(2,4)--(4,4)--(2,2);
			\draw[fill=black!20] (0,2)--(0,4)--(2,2)--(0,2);
			\draw[fill=black!20] (2,2)--(2,0)--(4,0)--(2,2);
			\draw[fill=black!20] (2,2)--(4,0)--(4,2)--(2,2);
			\draw[fill=black!20] (4,2)--(6,0)--(6,2)--(4,2);
			\draw[fill=black!20] (4,2)--(6,4)--(4,4)--(4,2);
			\draw (0,0) -- (4,4);
			\node[inner sep=0,anchor=west,text width=3.3cm] (note1) at (5,2.7) {$d_{2}'$};
			\node[inner sep=0,anchor=west,text width=3.3cm] (note1) at (4.3,0.6) {$d_{1}'$};
			\node[inner sep=0,anchor=west,text width=3.3cm] (note1) at (0.3,1.5) {$d_{2}$};
			\node[inner sep=0,anchor=west,text width=3.3cm] (note1) at (1.2,3.3) {$d_{1}$};
			\node[inner sep=0,anchor=west,text width=3.3cm] (note1) at (3.35,2.7) {$d_{0}$};
			\node[inner sep=0,anchor=west,text width=3.3cm] (note1) at (2,1.8) {$v$};
			\node[inner sep=0,anchor=west,text width=3.3cm] (note1) at (3.65,2.2) {$v'$};
			\draw[dotted] [line width=0.5mm](0,4) -- (2,4);
			\draw[dotted] [line width=0.5mm](2,2) -- (0,2);
			\draw[dotted] [line width=0.5mm](2,2) -- (4,4);
			\draw[dotted] [line width=0.5mm](4,2) -- (4,0);
			\draw[dotted] [line width=0.5mm](4,0) -- (4,2);
			\draw[dotted] [line width=0.5mm](4,0) -- (6,0);
			\draw[dotted] [line width=0.5mm](6,0) -- (4,2);
			\draw[dotted] [line width=0.5mm](4,2) -- (6,2);
			\draw[dotted] [line width=0.5mm](4,2) -- (6,4);
			\draw[dotted] [line width=0.5mm](6,4) -- (6,2);
			\draw[dotted] [line width=0.5mm](2,2) -- (4,2);
			\draw[dotted] [line width=0.5mm](0,0) -- (2,2);
			\draw[dotted] [line width=0.5mm](2,4) -- (2,2);
			\draw[dotted] [line width=0.5mm](0,0) -- (0,2);
			\draw[dotted] [line width=0.5mm](4,4) -- (4,2);
			\draw (0,0) node[circle,fill,inner sep=2pt] {};
			\draw (2,0) node[circle,fill,inner sep=2pt] {};
			\draw (4,0) node[circle,fill,inner sep=2pt] {};
			\draw (0,2) node[circle,fill,inner sep=2pt] {};
			\draw (2,2) node[circle,fill,inner sep=2pt] {};
			\draw (4,2) node[circle,fill,inner sep=2pt] {};
			\draw (0,4) node[circle,fill,inner sep=2pt] {};
			\draw (2,4) node[circle,fill,inner sep=2pt] {};
			\draw (4,4) node[circle,fill,inner sep=2pt] {};
			\draw (2,2.7) node[circle,fill,inner sep=2pt] {};
			\draw (2.6,2.6) node[circle,fill,inner sep=2pt] {};
			\draw (2.7,2) node[circle,fill,inner sep=2pt] {};
			\draw (1.4,1.4) node[circle,fill,inner sep=2pt] {};
			\node[inner sep=0,anchor=west,text width=3.3cm] (note1) at (2.7,2.5) {$v_{2}$};
			\node[inner sep=0,anchor=west,text width=3.3cm] (note1) at (2.7,1.8) {$v_{1}$};
			\node[inner sep=0,anchor=west,text width=3.3cm] (note1) at (1.5,2.7) {$v_{11}$};
			\node[inner sep=0,anchor=west,text width=3.3cm] (note1) at (1.3,1.2) {$v_{22}$};
			\node[inner sep=0,anchor=west,text width=3.3cm] (note1) at (-0.5,4.3) {$v_{12}$};
			\node[inner sep=0,anchor=west,text width=3.3cm] (note1) at (-0.5,2) {$v_{21}$};
			\draw[thick,->] (2,2.7) -- (2,3);
			\draw[thick,->] (2.6,2.6) -- (2.3,2.3);
			\draw[thick,->] (2.7,2) -- (2.3,2);
			\draw[thick,->] (1.4,1.4) -- (1.7,1.7);
			\draw[thick,->] (5,4) -- (5,3.6);
			\draw[thick,->] (6,4) -- (5.7,3.7);
			\draw[thick,->] (6,3) -- (5.6,3);
			\draw[thick,->] (6,2) -- (5.6,2);
			\draw[thick,->] (6,1) -- (5.6,1);
			\draw[thick,->] (5,1) -- (4.7,0.7);
			\draw[thick,->] (6,0) -- (5.6,0);
			\draw[thick,->] (1,4) -- (1,3.6);
			\draw[thick,->] (2,1) -- (2.4,1);
			\draw[thick,->] (3,1) -- (3.3,1.3);
			\draw[thick,->] (0,0) -- (0,0.4);
			\draw[thick,->] (0,2) -- (0.4,2);
			\draw[thick,->] (0,4) -- (0.3,3.7);
			\draw[thick,->] (1,0) -- (1,0.4);
			\draw[thick,->] (1,1) -- (0.7,1.3);
			\draw[thick,->] (0,3) -- (0.4,3);
			\draw[thick,->] (2,0) -- (2,-0.4);
			\draw[thick,->] (2,2) -- (2,2.4);
			\draw[thick,->] (2,4) -- (1.7,3.7);
			\draw[thick,->] (3,0) -- (3,-0.4);
			\draw[thick,->] (4,0) -- (4,0.4);
			\draw[thick,->] (4,2) -- (3.6,2);
			\draw[thick,->] (3,3) -- (3.3,2.7);
			\draw[thick,->] (3,4) -- (3,3.6);
			\draw[thick,->] (4,4) -- (4,3.6);
			\end{tikzpicture}
		} %
		\qquad
		\resizebox{0.45\textwidth}{!}{%
			\begin{tikzpicture}[scale=2]
			\draw (0,2)--(0,4);
			\draw (2,4)--(4,4);
			\draw (4,0)--(0,0);
			\draw (2,0) -- (2,2);
			\draw (2,2.7) -- (2,2);
			\draw (2,2) -- (2.7,2);
			\draw (6,0) -- (6,2);
			\draw (6,4) -- (4,4);
			\draw[fill=black!20] (0,0)--(2,0)--(2,2)--(0,0);
			\draw[fill=black!20] (2,2)--(2,4)--(4,4)--(2,2);
			\draw[fill=black!20] (0,2)--(0,4)--(2,2)--(0,2);
			\draw[fill=black!20] (2,2)--(2,0)--(4,0)--(2,2);
			\draw[fill=black!20] (2,2)--(4,0)--(4,2)--(2,2);
			\draw[fill=black!20] (4,2)--(6,0)--(6,2)--(4,2);
			\draw[fill=black!20] (4,2)--(6,4)--(4,4)--(4,2);
			\draw (0,0) -- (4,4);
			\node[inner sep=0,anchor=west,text width=3.3cm] (note1) at (5,2.7) {$d_{2}'$};
			\node[inner sep=0,anchor=west,text width=3.3cm] (note1) at (4.3,0.6) {$d_{1}'$};
			\node[inner sep=0,anchor=west,text width=3.3cm] (note1) at (0.3,1.5) {$d_{2}$};
			\node[inner sep=0,anchor=west,text width=3.3cm] (note1) at (1.2,3.3) {$d_{1}$};
			\node[inner sep=0,anchor=west,text width=3.3cm] (note1) at (3.35,2.7) {$d_{0}$};
			\node[inner sep=0,anchor=west,text width=3.3cm] (note1) at (2,1.8) {$v$};
			\node[inner sep=0,anchor=west,text width=3.3cm] (note1) at (3.65,2.2) {$v'$};
			\draw[dotted] [line width=0.5mm](2,2) -- (0,4);
			\draw[dotted] [line width=0.5mm](0,4) -- (2,4);
			\draw[dotted] [line width=0.5mm](2,2) -- (0,2);
			\draw[dotted] [line width=0.5mm](2,2) -- (4,4);
			\draw[dotted] [line width=0.5mm](4,2) -- (4,0);
			\draw[dotted] [line width=0.5mm](4,0) -- (4,2);
			\draw[dotted] [line width=0.5mm](4,0) -- (6,0);
			\draw[dotted] [line width=0.5mm](6,0) -- (4,2);
			\draw[dotted] [line width=0.5mm](4,2) -- (6,2);
			\draw[dotted] [line width=0.5mm](4,2) -- (6,4);
			\draw[dotted] [line width=0.5mm](6,4) -- (6,2);
			\draw[dotted] [line width=0.5mm](2,2) -- (4,2);
			\draw[dotted] [line width=0.5mm](0,0) -- (2,2);
			\draw[dotted] [line width=0.5mm](2,4) -- (2,2);
			\draw[dotted] [line width=0.5mm](0,0) -- (0,2);
			\draw[dotted] [line width=0.5mm](4,4) -- (4,2);
			\draw (0,0) node[circle,fill,inner sep=2pt] {};
			\draw (2,0) node[circle,fill,inner sep=2pt] {};
			\draw (4,0) node[circle,fill,inner sep=2pt] {};
			\draw (0,2) node[circle,fill,inner sep=2pt] {};
			\draw (2,2) node[circle,fill,inner sep=2pt] {};
			\draw (4,2) node[circle,fill,inner sep=2pt] {};
			\draw (0,4) node[circle,fill,inner sep=2pt] {};
			\draw (2,4) node[circle,fill,inner sep=2pt] {};
			\draw (4,4) node[circle,fill,inner sep=2pt] {};
			\draw (2,2.7) node[circle,fill,inner sep=2pt] {};
			\draw (2.6,2.6) node[circle,fill,inner sep=2pt] {};
			\draw (2.7,2) node[circle,fill,inner sep=2pt] {};
			\draw (1.4,1.4) node[circle,fill,inner sep=2pt] {};
			\draw (2.5,1.5) node[circle,fill,inner sep=2pt] {};
			\draw (2,1.3) node[circle,fill,inner sep=2pt] {};
			\node[inner sep=0,anchor=west,text width=3.3cm] (note1) at (2.7,2.5) {$v_{2}$};
					\node[inner sep=0,anchor=west,text width=3.3cm] (note1) at (2.7,1.8) {$v_{1}$};
					\node[inner sep=0,anchor=west,text width=3.3cm] (note1) at (1.5,2.7) {$v_{11}$};
					\node[inner sep=0,anchor=west,text width=3.3cm] (note1) at (1.3,1.2) {$v_{22}$};
					\node[inner sep=0,anchor=west,text width=3.3cm] (note1) at (-0.5,4.3) {$v_{12}$};
					\node[inner sep=0,anchor=west,text width=3.3cm] (note1) at (-0.5,2) {$v_{21}$};
			\draw (2,2.7) -- (2.6,2.6);
			\draw (2.5,1.5) -- (2.7,2);
			\draw (2,1.3) -- (2.5,1.5);
			\draw (2,1.3) -- (1.4,1.4);
			\draw[thick,->] (2,2.7) -- (2,3);
			\draw[thick,->] (2.6,2.6) -- (2.3,2.3);
			\draw[thick,->] (2.7,2) -- (2.3,2);
			\draw[thick,->] (1.4,1.4) -- (1.7,1.7);
			\draw[thick,->] (5,4) -- (5,3.6);
			\draw[thick,->] (6,4) -- (5.7,3.7);
			\draw[thick,->] (6,3) -- (5.6,3);
			\draw[thick,->] (6,2) -- (5.6,2);
			\draw[thick,->] (6,1) -- (5.6,1);
			\draw[thick,->] (5,1) -- (4.7,0.7);
			\draw[thick,->] (6,0) -- (5.6,0);
			\draw[thick,->] (1,4) -- (1,3.6);
			\draw[thick,->] (2,1) -- (2.4,1);
			\draw[thick,->] (3,1) -- (3.3,1.3);
			\draw[thick,->] (0,0) -- (0,0.4);
			\draw[thick,->] (0,2) -- (0.4,2);
			\draw[thick,->] (0,4) -- (0.3,3.7);
			\draw[thick,->] (1,0) -- (1,0.4);
			\draw[thick,->] (1,1) -- (0.7,1.3);
			\draw[thick,->] (0,3) -- (0.4,3);
			\draw[thick,->] (2,0) -- (2,-0.4);
			\draw[thick,->] (2,2) -- (2,2.4);
			\draw[thick,->] (2,4) -- (1.7,3.7);
			\draw[thick,->] (3,0) -- (3,-0.4);
			\draw[thick,->] (4,0) -- (4,0.4);
			\draw[thick,->] (4,2) -- (3.6,2);
			\draw[thick,->] (3,3) -- (3.3,2.7);
			\draw[thick,->] (3,4) -- (3,3.6);
			\draw[thick,->] (4,4) -- (4,3.6);
			\end{tikzpicture}
		} %
		
		\caption{}
		\label{fig:wedge2}
	\end{figure}

	\begin{figure}[h]
		\centering
		\resizebox{0.42\textwidth}{!}{%
			\begin{tikzpicture}[scale=2]
			\draw (0,2)--(0,4);
			\draw (2,4)--(4,4);
			\draw (4,0)--(0,0);
			\draw (2,0) -- (2,2);
			\draw (2,2.7) -- (2,2);
			\draw (2,2) -- (2.7,2);
			\draw (6,0) -- (6,2);
			\draw (6,4) -- (4,4);
			\draw[fill=black!20] (1.4,1.4)--(2,1.3)--(2,0)--(0,0)--(1.4,1.4);
			\draw[fill=black!20] (2,2.7)--(2,4)--(4,4)--(2.6,2.6)--(2,2.7);
			\draw[fill=black!20] (0,2)--(0,4)--(2,2)--(0,2);
			\draw[fill=black!20] (2,0)--(2,1.3)--(2.5,1.5)--(4,0)--(2,0);
			\draw[fill=black!20] (2.5,1.5)--(4,0)--(4,2)--(2.7,2)--(2.5,1.5);
			\draw[fill=black!20] (4,2)--(6,0)--(6,2)--(4,2);
			\draw[fill=black!20] (4,2)--(6,4)--(4,4)--(4,2);
			\draw (0,0) -- (4,4);
			\node[inner sep=0,anchor=west,text width=3.3cm] (note1) at (5,2.7) {$d_{2}'$};
			\node[inner sep=0,anchor=west,text width=3.3cm] (note1) at (4.3,0.6) {$d_{1}'$};
			\node[inner sep=0,anchor=west,text width=3.3cm] (note1) at (0.3,1.5) {$d_{2}$};
			\node[inner sep=0,anchor=west,text width=3.3cm] (note1) at (1.2,3.3) {$d_{1}$};
			\node[inner sep=0,anchor=west,text width=3.3cm] (note1) at (3.35,2.7) {$d_{0}$};
			\node[inner sep=0,anchor=west,text width=3.3cm] (note1) at (2,1.8) {$v$};
			\node[inner sep=0,anchor=west,text width=3.3cm] (note1) at (3.65,2.2) {$v'$};
			\draw[dotted] [line width=0.5mm](2,2) -- (0,4);
			\draw[dotted] [line width=0.5mm](0,4) -- (2,4);
			\draw[dotted] [line width=0.5mm](2,2) -- (0,2);
			\draw[dotted] [line width=0.5mm](2.6,2.6) -- (4,4);
			\draw[dotted] [line width=0.5mm](4,2) -- (4,0);
			\draw[dotted] [line width=0.5mm](4,0) -- (4,2);
			\draw[dotted] [line width=0.5mm](4,0) -- (6,0);
			\draw[dotted] [line width=0.5mm](6,0) -- (4,2);
			\draw[dotted] [line width=0.5mm](4,2) -- (6,2);
			\draw[dotted] [line width=0.5mm](4,2) -- (6,4);
			\draw[dotted] [line width=0.5mm](6,4) -- (6,2);
			\draw[dotted] [line width=0.5mm](2.7,2) -- (4,2);
			\draw[dotted] [line width=0.5mm](0,0) -- (1.4,1.4);
			\draw[dotted] [line width=0.5mm](2,4) -- (2,2.7);
			\draw[dotted] [line width=0.5mm](0,0) -- (0,2);
			\draw[dotted] [line width=0.5mm](4,4) -- (4,2);
			\draw[dotted] [line width=0.5mm](1.4,1.4) -- (2,1.3);
			\draw[dotted] [line width=0.5mm](2,1.3) -- (2.5,1.5);
			\draw[dotted] [line width=0.5mm](2.5,1.5) -- (2.7,2);
			\draw[dotted] [line width=0.5mm](2.6,2.6) -- (2,2.7);
			\draw (2,2.7) -- (2.6,2.6);
			\draw (2.5,1.5) -- (2.7,2);
			\draw (2,1.3) -- (2.5,1.5);
			\draw (2,1.3) -- (1.4,1.4);
			\draw[dotted] [line width=0.5mm](1.4,1.4) -- (2,1.3);
			\draw[dotted] [line width=0.5mm](2,1.3) -- (2.5,1.5);
			\draw[dotted] [line width=0.5mm](2.5,1.5) -- (2.7,2);
			\draw[dotted] [line width=0.5mm](2.6,2.6) -- (2,2.7);
			\draw[thick,dotted,->] (2,1.3) -- (2,1.7);
			\draw[thick,dotted,->] (2.5,1.5) -- (2.2,1.8);
			\draw[thick,dotted,->] (1.7,1.35) -- (1.8,1.7);
			\draw[thick,dotted,->] (2.25,1.4) -- (2.15,1.7);
			\draw[thick,dotted,->] (2.6,1.7) -- (2.35,1.9);
			\draw[thick,dotted,->] (2.3,2.65) -- (2.13,2.33);
			\draw (0,0) node[circle,fill,inner sep=2pt] {};
			\draw (2,0) node[circle,fill,inner sep=2pt] {};
			\draw (4,0) node[circle,fill,inner sep=2pt] {};
			\draw (0,2) node[circle,fill,inner sep=2pt] {};
			\draw (2,2) node[circle,fill,inner sep=2pt] {};
			\draw (4,2) node[circle,fill,inner sep=2pt] {};
			\draw (0,4) node[circle,fill,inner sep=2pt] {};
			\draw (2,4) node[circle,fill,inner sep=2pt] {};
			\draw (4,4) node[circle,fill,inner sep=2pt] {};
			\draw (2,2.7) node[circle,fill,inner sep=2pt] {};
			\draw (2.6,2.6) node[circle,fill,inner sep=2pt] {};
			\draw (2.7,2) node[circle,fill,inner sep=2pt] {};
			\draw (1.4,1.4) node[circle,fill,inner sep=2pt] {};
			\draw (2.5,1.5) node[circle,fill,inner sep=2pt] {};
			\draw (2,1.3) node[circle,fill,inner sep=2pt] {};
			\node[inner sep=0,anchor=west,text width=3.3cm] (note1) at (2.7,2.5) {$v_{2}$};
					\node[inner sep=0,anchor=west,text width=3.3cm] (note1) at (2.7,1.8) {$v_{1}$};
					\node[inner sep=0,anchor=west,text width=3.3cm] (note1) at (1.5,2.7) {$v_{11}$};
					\node[inner sep=0,anchor=west,text width=3.3cm] (note1) at (1.3,1.2) {$v_{22}$};
					\node[inner sep=0,anchor=west,text width=3.3cm] (note1) at (-0.5,4.3) {$v_{12}$};
					\node[inner sep=0,anchor=west,text width=3.3cm] (note1) at (-0.5,2) {$v_{21}$};
			\draw[thick,->] (2,2.7) -- (2,3);
			\draw[thick,->] (2.6,2.6) -- (2.3,2.3);
			\draw[thick,->] (2.7,2) -- (2.3,2);
			\draw[thick,->] (1.4,1.4) -- (1.7,1.7);
			\draw[thick,->] (5,4) -- (5,3.6);
			\draw[thick,->] (6,4) -- (5.7,3.7);
			\draw[thick,->] (6,3) -- (5.6,3);
			\draw[thick,->] (6,2) -- (5.6,2);
			\draw[thick,->] (6,1) -- (5.6,1);
			\draw[thick,->] (5,1) -- (4.7,0.7);
			\draw[thick,->] (6,0) -- (5.6,0);
			\draw[thick,->] (1,4) -- (1,3.6);
			\draw[thick,->] (2,1) -- (2.4,1);
			\draw[thick,->] (3,1) -- (3.3,1.3);
			\draw[thick,->] (0,0) -- (0,0.4);
			\draw[thick,->] (0,2) -- (0.4,2);
			\draw[thick,->] (0,4) -- (0.3,3.7);
			\draw[thick,->] (1,0) -- (1,0.4);
			\draw[thick,->] (1,1) -- (0.7,1.3);
			\draw[thick,->] (0,3) -- (0.4,3);
			\draw[thick,->] (2,0) -- (2,-0.4);
			\draw[thick,->] (2,2) -- (2,2.4);
			\draw[thick,->] (2,4) -- (1.7,3.7);
			\draw[thick,->] (3,0) -- (3,-0.4);
			\draw[thick,->] (4,0) -- (4,0.4);
			\draw[thick,->] (4,2) -- (3.6,2);
			\draw[thick,->] (3,3) -- (3.3,2.7);
			\draw[thick,->] (3,4) -- (3,3.6);
			\draw[thick,->] (4,4) -- (4,3.6);
			\end{tikzpicture}
		} %
		\caption{}
		\label{fig:unwedge2}
	\end{figure}
	
\end{proof}

\begin{remark}
	Note that in Theorem \ref{decompose}, we need to do subdivisions to obtain the boundary as a manifold.  Hence the boundary of the $2$-manifolds should be a circle after the subdivisions since the boundary will be closed $1$-dimensional manifold.
\end{remark}

The following is an immediate corollary of Theorems~\ref{boundary} and \ref{decompose}.
\begin{corollary}
	Let $M=M_{1}\# M_{2}$ be the connected sum of two closed oriented triangulated surfaces of genus $g_{1}$ and $g_{2}$ respectively and $f$ be a perfect discrete Morse function defined on $M$.  Then we can extend $f_{|_{M-M_{2}}}$  to $M_{1}$ and 
	$f_{|_{M-M_{1}}}$  to $M_{2}$ as perfect discrete Morse functions.
\end{corollary}

To clear up the process described in Theorem~\ref{decompose}, we work it out in the following example.

\begin{example} \label{genus-2}
	Let $M$ be the triangulated genus $2$ orientable surface with a perfect discrete Morse function which induces the gradient vector field depicted in 
	Figure \ref{fig:torus}.

	\begin{figure}[h]
		\centering
		\resizebox{0.5\textwidth}{!}{%
			\begin{tikzpicture}
			\draw (0,0)  -- (0,3.6) -- (3,6)-- (6.6,6)-- (9,3.6)-- (9,0)-- (6.6,-3)-- (3,-3)-- (0,0);
			\draw (5.4,-2.5) node[circle,fill,color=red,inner sep=1pt] {};
			\draw (1.87,5.1)  -- (4.2,6);
			\draw (1.87,5.1)  -- (4.2,5);
			\draw (4.2,6)  -- (4.2,5);
			\draw (4.2,5)  -- (5.4,6);
			\draw (4.2,5)  -- (5.4,5);
			\draw (5.4,5)  -- (6.6,6);
			\draw (5.4,5)  -- (5.4,6);
			\draw (5.4,5)  -- (7.5,5.1);
			\draw (0.87,4.3)  -- (4.2,5);
			\draw (0.87,4.3)  -- (4.2,4);
			\draw (4.2,4)  -- (4.2,5);
			\draw (4.2,4)  -- (5.4,5);
			\draw (4.2,4)  -- (5.4,4);
			\draw (5.4,4)  -- (5.4,5);
			\draw (5.4,4)  -- (7.5,5.1);
			\draw (5.4,4)  -- (8.1,4.5);
			\draw (5.4,4)  -- (9,3.6);
			\draw (5.4,4)  -- (6,3);
			\draw (5.4,4)  -- (5.2,1);
			\draw (5.4,4)  -- (4.7,2);
			\draw (5.4,4)  -- (2.5,1);
			\draw (5.4,4)  -- (3,2);
			\draw (3,2)  -- (4.2,4);
			\draw (4.2,4)  -- (0,1.2);
			\draw (4.2,4)  -- (0,2.4);
			\draw (4.2,4)  -- (0,3.6);
			\draw (0,1.2)  -- (3,2);
			\draw (0,0)  -- (3,2);
			\draw (3,2)  -- (2.5,1);
			\draw (2.5,1)  -- (4.7,2);
			\draw (4.7,2)  -- (5.2,1);
			\draw (5.2,1)  -- (6,3);
			\draw (6,3)  -- (9,3.6);
			\draw (6,3)  -- (6.6,1.5);
			\draw (6.6,1.5)  -- (9,3.6);
			\draw (5.2,1)  -- (6.6,1.5);
			\draw (4,0.5)  -- (5.2,1);
			\draw (4,0.5)  -- (4.7,2);
			\draw (4,0.5)  -- (2.5,1);
			\draw (2.5,1)  -- (0,0);
			\draw (4,0.5)  -- (0,0);
			\draw (6.6,1.5)  -- (7,1);
			\draw (7,1)  -- (9,3.6);
			\draw (7,1)  -- (9,2.4);
			\draw (5,-1.5)  -- (7,1);
			\draw (5,-1.5)  -- (9,2.4);
			\draw (5,-1.5)  -- (6.6,1.5);
			\draw (5,-1.5)  -- (5.2,1);
			\draw (5,-1.5)  -- (4,0.5);
			\draw (5,-1.5)  -- (0,0);
			\draw (5,-1.5)  -- (9,1.2);
			\draw (5,-1.5)  -- (6.5,-1.5);
			\draw (5,-1.5)  -- (4.2,-2);
			\draw (5,-1.5)  -- (1,-1);
			\draw (9,1.2)  -- (6.5,-1.5);
			\draw (9,0)  -- (6.5,-1.5);
			\draw (8.19,-1)  -- (6.5,-1.5);
			\draw (5.4,-2.5)  -- (6.5,-1.5);
			\draw (4.2,-2)  -- (6.5,-1.5);
			\draw (4.2,-2)  -- (5.4,-2.5);
			\draw (8.19,-1)  -- (5.4,-2.5);
			\draw (7.39,-2)  -- (5.4,-2.5);
			\draw (5.4,-3)  -- (5.4,-2.5);
			\draw (4.2,-3)  -- (5.4,-2.5);
			\draw (5.4,-3)  -- (7.39,-2);
			\draw (4.2,-3)  -- (4.2,-2);
			\draw (3,-3)  -- (4.2,-2);
			\draw (2,-2)  -- (4.2,-2);
			\draw (1,-1)  -- (4.2,-2);
			\draw (1,-1)  -- (5,-1.5);
			\draw [line width=0.3mm, color=red](0,1.2) -- (4.2,4);
			\draw [line width=0.3mm, color=blue](5.4,-3) -- (6.6,-3);
			\draw [line width=0.3mm, color=red](4.2,4) -- (5.4,5);
			\draw [line width=0.3mm, color=blue](4.2,-3) -- (5.4,-2.5);
			\draw (5.4,-2.5) node[circle,fill,color=black,inner sep=2pt] {};
			\node[inner sep=0,anchor=west,text width=3.3cm] (note1) at (5.3,-2.3) {$w$};
			\draw[fill=black] (5.4,4)--(8.1,4.5)--(7.5,5.1)--(5.4,4);
			\draw[thick,->] (2.5,5.6) -- (2.8,5.6);
			\draw[thick,->] (3.17,5.6) -- (3.37,5.37);
			\draw[thick,->] (4.2,5.5) -- (4.5,5.5);
			\draw[thick,->] (4.2,5) -- (4.459,5.23);
			\draw[thick,->] (4.8,5) -- (4.8,5.3);
			\draw[thick,->] (5.4,5.5) -- (5.7,5.5);
			\draw[thick,->] (5.4,5) -- (5.659,5.23);
			\draw[thick,->] (6.45,5.05) -- (6.45,5.35);
			\draw[thick,->] (3.2,5.045) -- (2.9,4.9);
			\draw[thick,->] (2.7,4.68) -- (2.8,4.38);
			\draw[thick,->] (2.69,4.138) -- (2.33,4.03);
			\draw[thick,->] (4.2,4) -- (3.7,3.95);
			\draw[thick,->] (0,3) -- (0.3,3);
			\draw[thick,->] (1.9,3.125) -- (2.2,2.99);
			\draw[thick,->] (1.9,1.7) -- (2.1,2.05);
			\draw[thick,->] (1.5,1) -- (1.3,1.2);
			\draw[thick,->] (1.9,0.75) -- (1.7,1);
			\draw[thick,->] (2.35,0.295) -- (2.35,0.599);
			\draw[thick,->] (3,-0.9) -- (3.05,-0.5);
			\draw[thick,->] (2.2,-1.15) -- (2.2,-0.8);
			\draw[thick,->] (2.6,-1.5) -- (2.4,-1.8);
			\draw[thick,->] (3.05,-2) -- (3.05,-2.3);
			\draw[thick,->] (3.6,-2.5)-- (3.8,-2.75);
			\draw[thick,->] (4.2,-2.5) -- (4.5,-2.5);
			\draw[thick,->] (4.2,-3) -- (4.5,-3);
			\draw[thick,->] (5.4,-2.75) -- (5.1,-2.75);
			\draw[thick,->] (6.93,-2.6) -- (6.6,-2.6);
			\draw[thick,->] (6.3,-2.55) -- (6,-2.5);
			\draw[thick,->] (6.5,-2.22) -- (6.7,-1.95);
			\draw[thick,->] (6.7,-1.8) -- (6.5,-1.6);
			\draw[thick,->] (7.4,-1.23) -- (7.6,-1);
			\draw[thick,->] (8,-0.6) -- (8.1,-0.3);
			\draw[thick,->] (9,1.2) -- (8.7,0.88);
			\draw[thick,->] (6.6,-0.42) -- (6.6,-0.72);
			\draw[thick,->] (7,0.45) -- (6.9,0.1);
			\draw[thick,->] (7.8,1.56) -- (7.65,1.25);
			\draw[thick,->] (7,1) -- (6.77,0.7);
			\draw[thick,->] (6.85,1.2) -- (6.55,0.9);
			\draw[thick,->] (8,2.3) -- (7.9,1.9);
			\draw[thick,->] (7.8,2.555) -- (7.7,2.2);
			\draw[thick,->] (7.5,3.3) -- (7.4,3);
			\draw[thick,->] (7.3,3.8) -- (7.2,3.5);
			\draw[thick,->] (7.4,4.38) -- (7.3,4.08);
			\draw[thick,->] (6.4,4.53) -- (6.3,4.8);
			\draw[thick,->] (5.4,4) -- (5.4,4.35);
			\draw[thick,->] (4.8,4) -- (4.8,4.35);
			\draw[thick,->] (4.2,4.5) -- (4.55,4.5);
			\draw[thick,->] (3.6,3) -- (3.95,3.1);
			\draw[thick,->] (3,2) -- (3.3,2.25);
			\draw[thick,->] (2.75,1.5) -- (3.1,1.8);
			\draw[thick,->] (2.5,1) -- (2.83,1.34);
			\draw[thick,->] (3.1,0.8) -- (3.4,1);
			\draw[thick,->] (3.6,1.5) -- (3.8,1.79);
			\draw[thick,->] (4.7,2) -- (4.82,2.32);
			\draw[thick,->] (4,0.5) -- (4.32,0.64);
			\draw[thick,->] (4.33,1.2) -- (4.6,1.3);
			\draw[thick,->] (4.9,1.6) -- (5,1.9);
			\draw[thick,->] (5.2,1) -- (5.5,1.11);
			\draw[thick,->] (5.6,2.02) -- (5.9,1.9);
			\draw[thick,->] (5.3,2.5) -- (5.55,2.7);
			\draw[thick,->] (6.6,1.5) -- (6.48,1.8);
			\draw[thick,->] (6,3) -- (5.82,3.3);
			\draw[thick,->] (5.9,0.2) -- (5.6,0.3);
			\draw[thick,->] (5.13,0.05) -- (4.85,0.2);
			\draw[thick,->] (5,-1.5) -- (4.815,-1.13);
			\draw[thick,->] (6.5,-1.5) -- (6.23,-1.75);
			\draw[thick,->] (5.4,-1.74) -- (5.4,-2.05);
			\draw[thick,->] (5.75,-1.503) -- (5.42,-1.65);
			\draw[thick,->] (4.61,-1.75) -- (4.3,-1.65);
			\draw[thick,->] (2,-2) -- (1.75,-1.75);
			\draw[thick,->] (1,-1) -- (0.75,-0.75);
			\draw[thick,->] (4.2,6) -- (4.55,6);
			\draw[thick,->] (5.4,6) -- (5.75,6);
			\draw[thick,->] (7.5,5.1) -- (7.75,4.85);
			\draw[thick,->] (8.1,4.5) -- (8.35,4.25);
			\draw[thick,->] (9,3.6) -- (9,3.25);
			\draw[thick,->] (9,2.4) -- (9,2.05);
			\draw[thick,->] (8.19,-1) -- (8.41,-0.75);
			\draw[thick,->] (7.39,-2) -- (7.59,-1.75);
			\draw[thick,->] (4.2,-2) -- (4.5,-2.12);
			\node[inner sep=0,anchor=west,text width=3.3cm] (note1) at (-0.3,0) {$3$};
			\node[inner sep=0,anchor=west,text width=3.3cm] (note1) at (-0.3,3.6) {$3$}; 
			\node[inner sep=0,anchor=west,text width=3.3cm] (note1) at (2.7,6.1) {$3$};
			\node[inner sep=0,anchor=west,text width=3.3cm] (note1) at (6.8,6.1) {$3$};
			\node[inner sep=0,anchor=west,text width=3.3cm] (note1) at (9.2,3.6) {$3$};
			\node[inner sep=0,anchor=west,text width=3.3cm] (note1) at (9.2,0) {$3$};
			\node[inner sep=0,anchor=west,text width=3.3cm] (note1) at (6.8,-3.1) {$3$};
			\node[inner sep=0,anchor=west,text width=3.3cm] (note1) at (2.7,-3.1) {$3$};
			\node[inner sep=0,anchor=west,text width=3.3cm] (note1) at (0.47,4.3) {$1$};
			\node[inner sep=0,anchor=west,text width=3.3cm] (note1) at (1.47,5.1) {$2$};
			\node[inner sep=0,anchor=west,text width=3.3cm] (note1) at (4.2,6.2) {$4$};
			\node[inner sep=0,anchor=west,text width=3.3cm] (note1) at (5.4,6.2) {$5$};
			\node[inner sep=0,anchor=west,text width=3.3cm] (note1) at (7.7,5.1) {$2$};
			\node[inner sep=0,anchor=west,text width=3.3cm] (note1) at (8.3,4.5) {$1$};
			\node[inner sep=0,anchor=west,text width=3.3cm] (note1) at (9.2,2.4) {$6$};
			\node[inner sep=0,anchor=west,text width=3.3cm] (note1) at (9.2,1.2) {$7$};
			\node[inner sep=0,anchor=west,text width=3.3cm] (note1) at (8.39,-1) {$8$};
			\node[inner sep=0,anchor=west,text width=3.3cm] (note1) at (7.59,-2) {$9$};
			\node[inner sep=0,anchor=west,text width=3.3cm] (note1) at (5.4,-3.2) {$7$};
			\node[inner sep=0,anchor=west,text width=3.3cm] (note1) at (4.2,-3.2) {$6$};
			\node[inner sep=0,anchor=west,text width=3.3cm] (note1) at (1.6,-2) {$9$};
			\node[inner sep=0,anchor=west,text width=3.3cm] (note1) at (0.6,-1) {$8$};
			\node[inner sep=0,anchor=west,text width=3.3cm] (note1) at (-0.3,1.2) {$5$};
			\node[inner sep=0,anchor=west,text width=3.3cm] (note1) at (-0.3,2.4) {$4$};
			\end{tikzpicture}
		} %
		\caption{A gradient vector field on genus $2$ orientable surface.}
		\label{fig:torus} 
	\end{figure}
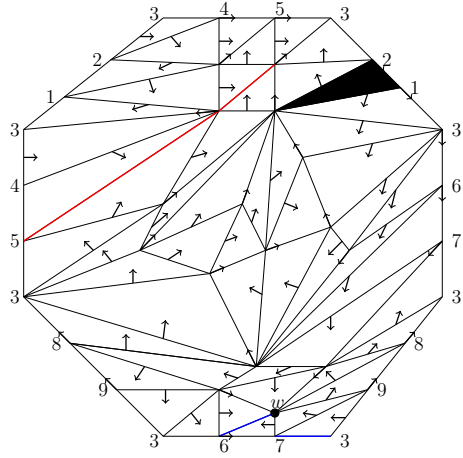
	
	The black triangle is the critical $2$-cell and the vertex $w$ is the critical $0$-cell.  The red and blue
	edges are the critical $1$-cells.  As in the proof of Theorem \ref{decompose} 
	we select two critical $1$-cells that belong to $M-M_1$.  Remember that discrete flow points the direction in which $f$ is non-increasing, so these should be the two critical $1$-cells on which $f$ attains higher values. In our example only the discrete vector field is given and the perfect discrete Morse function is not specified, we may choose any two, and this choice will determine the decomposion of the surface into two tori. 
	
	Because of our construction (see also the proof of Theorem~\ref{compose}) we do not want to see any critical cells on the resulting boundary.  If a 
	$2$-path from the boundary of the critical $2$-cell to a critical $1$-cell passes through other critical cells, then we need to perturb the vector field by doing subdivision(bisection or barycentric subdivision) in the star of the critical cell.  Since we want to 
	change the given function as little as possible, whenever we can, we choose critical $1$-cells so that no such perturbation is necessary.  
	
	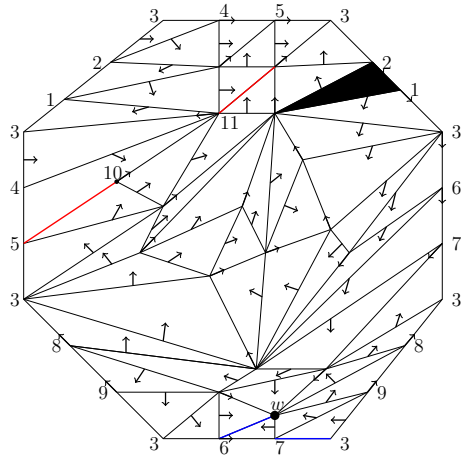
\begin{figure}[h]
		\centering
		\resizebox{0.5\textwidth}{!}{%
			\begin{tikzpicture}
			\draw (0,0)  -- (0,3.6) -- (3,6)-- (6.6,6)-- (9,3.6)-- (9,0)-- (6.6,-3)-- (3,-3)-- (0,0);
			\draw (3,2)  -- (2,2.53);
			\draw (1.87,5.1)  -- (4.2,6);
			\draw (1.87,5.1)  -- (4.2,5);
			\draw (4.2,6)  -- (4.2,5);
			\draw (4.2,5)  -- (5.4,6);
			\draw (4.2,5)  -- (5.4,5);
			\draw (5.4,5)  -- (6.6,6);
			\draw (5.4,5)  -- (5.4,6);
			\draw (5.4,5)  -- (7.5,5.1);
			\draw (0.87,4.3)  -- (4.2,5);
			\draw (0.87,4.3)  -- (4.2,4);
			\draw (4.2,4)  -- (4.2,5);
			\draw (4.2,4)  -- (5.4,5);
			\draw (4.2,4)  -- (5.4,4);
			\draw (5.4,4)  -- (5.4,5);
			\draw (5.4,4)  -- (7.5,5.1);
			\draw (5.4,4)  -- (8.1,4.5);
			\draw (5.4,4)  -- (9,3.6);
			\draw (5.4,4)  -- (6,3);
			\draw (5.4,4)  -- (5.2,1);
			\draw (5.4,4)  -- (4.7,2);
			\draw (5.4,4)  -- (2.5,1);
			\draw (5.4,4)  -- (3,2);
			\draw (3,2)  -- (4.2,4);
			\draw (4.2,4)  -- (2,2.53);
			\draw (4.2,4)  -- (0,2.4);
			\draw (4.2,4)  -- (0,3.6);
			\draw (0,1.2)  -- (3,2);
			\draw (0,0)  -- (3,2);
			\draw (3,2)  -- (2.5,1);
			\draw (2.5,1)  -- (4.7,2);
			\draw (4.7,2)  -- (5.2,1);
			\draw (5.2,1)  -- (6,3);
			\draw (6,3)  -- (9,3.6);
			\draw (6,3)  -- (6.6,1.5);
			\draw (6.6,1.5)  -- (9,3.6);
			\draw (5.2,1)  -- (6.6,1.5);
			\draw (4,0.5)  -- (5.2,1);
			\draw (4,0.5)  -- (4.7,2);
			\draw (4,0.5)  -- (2.5,1);
			\draw (2.5,1)  -- (0,0);
			\draw (4,0.5)  -- (0,0);
			\draw (6.6,1.5)  -- (7,1);
			\draw (7,1)  -- (9,3.6);
			\draw (7,1)  -- (9,2.4);
			\draw (5,-1.5)  -- (7,1);
			\draw (5,-1.5)  -- (9,2.4);
			\draw (5,-1.5)  -- (6.6,1.5);
			\draw (5,-1.5)  -- (5.2,1);
			\draw (5,-1.5)  -- (4,0.5);
			\draw (5,-1.5)  -- (0,0);
			\draw (5,-1.5)  -- (9,1.2);
			\draw (5,-1.5)  -- (6.5,-1.5);
			\draw (5,-1.5)  -- (4.2,-2);
			\draw (5,-1.5)  -- (1,-1);
			\draw (9,1.2)  -- (6.5,-1.5);
			\draw (9,0)  -- (6.5,-1.5);
			\draw (8.19,-1)  -- (6.5,-1.5);
			\draw (5.4,-2.5)  -- (6.5,-1.5);
			\draw (4.2,-2)  -- (6.5,-1.5);
			\draw (4.2,-2)  -- (5.4,-2.5);
			\draw (8.19,-1)  -- (5.4,-2.5);
			\draw (7.39,-2)  -- (5.4,-2.5);
			\draw (5.4,-3)  -- (5.4,-2.5);
			\draw (4.2,-3)  -- (5.4,-2.5);
			\draw (5.4,-3)  -- (7.39,-2);
			\draw (4.2,-3)  -- (4.2,-2);
			\draw (3,-3)  -- (4.2,-2);
			\draw (2,-2)  -- (4.2,-2);
			\draw (1,-1)  -- (4.2,-2);
			\draw (1,-1)  -- (5,-1.5);
			\draw [line width=0.3mm, color=red](0,1.2) -- (2,2.53);
			\draw [line width=0.3mm, color=blue](5.4,-3) -- (6.6,-3);
			\draw [line width=0.3mm, color=red](4.2,4) -- (5.4,5);
			\draw [line width=0.3mm, color=blue](4.2,-3) -- (5.4,-2.5);
			\draw (2,2.53) node[circle,fill,inner sep=1pt] {};
			\draw (5.4,-2.5) node[circle,fill,color=black,inner sep=2pt] {};
			\node[inner sep=0,anchor=west,text width=3.3cm] (note1) at (5.3,-2.3) {$w$};
			\draw[fill=black] (5.4,4)--(8.1,4.5)--(7.5,5.1)--(5.4,4);
			\draw[thick,->] (2.5,5.6) -- (2.8,5.6);
			\draw[thick,->] (3.17,5.6) -- (3.37,5.37);
			\draw[thick,->] (4.2,5.5) -- (4.5,5.5);
			\draw[thick,->] (4.2,5) -- (4.459,5.23);
			\draw[thick,->] (4.8,5) -- (4.8,5.3);
			\draw[thick,->] (5.4,5.5) -- (5.7,5.5);
			\draw[thick,->] (5.4,5) -- (5.659,5.23);
			\draw[thick,->] (6.45,5.05) -- (6.45,5.35);
			\draw[thick,->] (3.2,5.045) -- (2.9,4.9);
			\draw[thick,->] (2.7,4.68) -- (2.8,4.38);
			\draw[thick,->] (2.69,4.138) -- (2.33,4.03);
			\draw[thick,->] (4.2,4) -- (3.7,3.95);
			\draw[thick,->] (0,3) -- (0.3,3);
			\draw[thick,->] (1.9,3.125) -- (2.2,2.99);
			\draw[thick,->] (1.9,1.7) -- (2.1,2.05);
			\draw[thick,->] (1.5,1) -- (1.3,1.2);
			\draw[thick,->] (1.9,0.75) -- (1.7,1);
			\draw[thick,->] (2.35,0.295) -- (2.35,0.599);
			\draw[thick,->] (3,-0.9) -- (3.05,-0.5);
			\draw[thick,->] (2.2,-1.15) -- (2.2,-0.8);
			\draw[thick,->] (2.6,-1.5) -- (2.4,-1.8);
			\draw[thick,->] (3.05,-2) -- (3.05,-2.3);
			\draw[thick,->] (3.6,-2.5)-- (3.8,-2.75);
			\draw[thick,->] (4.2,-2.5) -- (4.5,-2.5);
			\draw[thick,->] (4.2,-3) -- (4.5,-3);
			\draw[thick,->] (5.4,-2.75) -- (5.1,-2.75);
			\draw[thick,->] (6.93,-2.6) -- (6.6,-2.6);
			\draw[thick,->] (6.3,-2.55) -- (6,-2.5);
			\draw[thick,->] (6.5,-2.22) -- (6.7,-1.95);
			\draw[thick,->] (6.7,-1.8) -- (6.5,-1.6);
			\draw[thick,->] (7.4,-1.23) -- (7.6,-1);
			\draw[thick,->] (8,-0.6) -- (8.1,-0.3);
			\draw[thick,->] (9,1.2) -- (8.7,0.88);
			\draw[thick,->] (6.6,-0.42) -- (6.6,-0.72);
			\draw[thick,->] (7,0.45) -- (6.9,0.1);
			\draw[thick,->] (7.8,1.56) -- (7.65,1.25);
			\draw[thick,->] (7,1) -- (6.77,0.7);
			\draw[thick,->] (6.85,1.2) -- (6.55,0.9);
			\draw[thick,->] (8,2.3) -- (7.9,1.9);
			\draw[thick,->] (7.8,2.555) -- (7.7,2.2);
			\draw[thick,->] (7.5,3.3) -- (7.4,3);
			\draw[thick,->] (7.3,3.8) -- (7.2,3.5);
			\draw[thick,->] (7.4,4.38) -- (7.3,4.08);
			\draw[thick,->] (6.4,4.53) -- (6.3,4.8);
			\draw[thick,->] (5.4,4) -- (5.4,4.35);
			\draw[thick,->] (4.8,4) -- (4.8,4.35);
			\draw[thick,->] (4.2,4.5) -- (4.55,4.5);
			\draw[thick,->] (3.6,3) -- (3.95,3.1);
			\draw[thick,->] (3,2) -- (3.3,2.25);
			\draw[thick,->] (2.75,1.5) -- (3.1,1.8);
			\draw[thick,->] (2.5,1) -- (2.83,1.34);
			\draw[thick,->] (3.1,0.8) -- (3.4,1);
			\draw[thick,->] (3.6,1.5) -- (3.8,1.79);
			\draw[thick,->] (4.7,2) -- (4.82,2.32);
			\draw[thick,->] (4,0.5) -- (4.32,0.64);
			\draw[thick,->] (4.33,1.2) -- (4.6,1.3);
			\draw[thick,->] (4.9,1.6) -- (5,1.9);
			\draw[thick,->] (5.2,1) -- (5.5,1.11);
			\draw[thick,->] (5.6,2.02) -- (5.9,1.9);
			\draw[thick,->] (5.3,2.5) -- (5.55,2.7);
			\draw[thick,->] (6.6,1.5) -- (6.48,1.8);
			\draw[thick,->] (6,3) -- (5.82,3.3);
			\draw[thick,->] (5.9,0.2) -- (5.6,0.3);
			\draw[thick,->] (5.13,0.05) -- (4.85,0.2);
			\draw[thick,->] (5,-1.5) -- (4.815,-1.13);
			\draw[thick,->] (6.5,-1.5) -- (6.23,-1.75);
			\draw[thick,->] (5.4,-1.74) -- (5.4,-2.05);
			\draw[thick,->] (5.75,-1.503) -- (5.42,-1.65);
			\draw[thick,->] (4.61,-1.75) -- (4.3,-1.65);
			\draw[thick,->] (2,-2) -- (1.75,-1.75);
			\draw[thick,->] (1,-1) -- (0.75,-0.75);
			\draw[thick,->] (4.2,6) -- (4.55,6);
			\draw[thick,->] (5.4,6) -- (5.75,6);
			\draw[thick,->] (7.5,5.1) -- (7.75,4.85);
			\draw[thick,->] (8.1,4.5) -- (8.35,4.25);
			\draw[thick,->] (9,3.6) -- (9,3.25);
			\draw[thick,->] (9,2.4) -- (9,2.05);
			\draw[thick,->] (8.19,-1) -- (8.41,-0.75);
			\draw[thick,->] (7.39,-2) -- (7.59,-1.75);
			\draw[thick,->] (4.2,-2) -- (4.5,-2.12);
			\draw[thick,->] (2,2.53) -- (2.3,2.74);
			\draw[thick,->] (2.43,2.3) -- (2.73,2.5);
			\node[inner sep=0,anchor=west,text width=2.5cm] (note1) at (1.7,2.73) {$10$};
			\node[inner sep=0,anchor=west,text width=2.5cm] (note1) at (4.22,3.8) {$11$};
			\node[inner sep=0,anchor=west,text width=3.3cm] (note1) at (-0.3,0) {$3$};
			\node[inner sep=0,anchor=west,text width=3.3cm] (note1) at (-0.3,3.6) {$3$}; 
			\node[inner sep=0,anchor=west,text width=3.3cm] (note1) at (2.7,6.1) {$3$};
			\node[inner sep=0,anchor=west,text width=3.3cm] (note1) at (6.8,6.1) {$3$};
			\node[inner sep=0,anchor=west,text width=3.3cm] (note1) at (9.2,3.6) {$3$};
			\node[inner sep=0,anchor=west,text width=3.3cm] (note1) at (9.2,0) {$3$};
			\node[inner sep=0,anchor=west,text width=3.3cm] (note1) at (6.8,-3.1) {$3$};
			\node[inner sep=0,anchor=west,text width=3.3cm] (note1) at (2.7,-3.1) {$3$};
			\node[inner sep=0,anchor=west,text width=3.3cm] (note1) at (0.47,4.3) {$1$};
			\node[inner sep=0,anchor=west,text width=3.3cm] (note1) at (1.47,5.1) {$2$};
			\node[inner sep=0,anchor=west,text width=3.3cm] (note1) at (4.2,6.2) {$4$};
			\node[inner sep=0,anchor=west,text width=3.3cm] (note1) at (5.4,6.2) {$5$};
			\node[inner sep=0,anchor=west,text width=3.3cm] (note1) at (7.7,5.1) {$2$};
			\node[inner sep=0,anchor=west,text width=3.3cm] (note1) at (8.3,4.5) {$1$};
			\node[inner sep=0,anchor=west,text width=3.3cm] (note1) at (9.2,2.4) {$6$};
			\node[inner sep=0,anchor=west,text width=3.3cm] (note1) at (9.2,1.2) {$7$};
			\node[inner sep=0,anchor=west,text width=3.3cm] (note1) at (8.39,-1) {$8$};
			\node[inner sep=0,anchor=west,text width=3.3cm] (note1) at (7.59,-2) {$9$};
			\node[inner sep=0,anchor=west,text width=3.3cm] (note1) at (5.4,-3.2) {$7$};
			\node[inner sep=0,anchor=west,text width=3.3cm] (note1) at (4.2,-3.2) {$6$};
			\node[inner sep=0,anchor=west,text width=3.3cm] (note1) at (1.6,-2) {$9$};
			\node[inner sep=0,anchor=west,text width=3.3cm] (note1) at (0.6,-1) {$8$};
			\node[inner sep=0,anchor=west,text width=3.3cm] (note1) at (-0.3,1.2) {$5$};
			\node[inner sep=0,anchor=west,text width=3.3cm] (note1) at (-0.3,2.4) {$4$};
			\end{tikzpicture}
		} %
		\caption{Separation of the critical edges.}
		\label{fig:torussep} 
	\end{figure}
	We choose red colored edges as the critical edges that belong to $M-M_{1}$.   These red edges have a common vertex.  
	We separate them using bisections (see Figure \ref{fig:torussep}).
	\begin{figure}[h]
		\centering 
		\resizebox{0.5\textwidth}{!}{%
			\begin{tikzpicture}
			\draw (0,0)  -- (0,3.6) -- (3,6)-- (6.6,6)-- (9,3.6)-- (9,0)-- (6.6,-3)-- (3,-3)-- (0,0);
			\draw (3,2)  -- (2,2.53);
			\draw (1.87,5.1)  -- (4.2,6);
			\draw (1.87,5.1)  -- (4.2,5);
			\draw (4.2,6)  -- (4.2,5);
			\draw (4.2,5)  -- (5.4,6);
			\draw (4.2,5)  -- (5.4,5);
			\draw (5.4,5)  -- (6.6,6);
			\draw (5.4,5)  -- (5.4,6);
			\draw (5.4,5)  -- (7.5,5.1);
			\draw (0.87,4.3)  -- (4.2,5);
			\draw (0.87,4.3)  -- (4.2,4);
			\draw (4.2,4)  -- (4.2,5);
			\draw (4.2,4)  -- (5.4,5);
			\draw (4.2,4)  -- (5.4,4);
			\draw (5.4,4)  -- (5.4,5);
			\draw (5.4,4)  -- (7.5,5.1);
			\draw (5.4,4)  -- (8.1,4.5);
			\draw (5.4,4)  -- (9,3.6);
			\draw (5.4,4)  -- (6,3);
			\draw (5.4,4)  -- (5.2,1);
			\draw (5.4,4)  -- (4.7,2);
			\draw (5.4,4)  -- (2.5,1);
			\draw (5.4,4)  -- (3,2);
			\draw (3,2)  -- (4.2,4);
			\draw (4.2,4)  -- (2,2.53);
			\draw (4.2,4)  -- (0,2.4);
			\draw (4.2,4)  -- (0,3.6);
			\draw (0,1.2)  -- (3,2);
			\draw (0,0)  -- (3,2);
			\draw (3,2)  -- (2.5,1);
			\draw (2.5,1)  -- (4.7,2);
			\draw (4.7,2)  -- (5.2,1);
			\draw (5.2,1)  -- (6,3);
			\draw (6,3)  -- (9,3.6);
			\draw (6,3)  -- (6.6,1.5);
			\draw (6.6,1.5)  -- (9,3.6);
			\draw (5.2,1)  -- (6.6,1.5);
			\draw (4,0.5)  -- (5.2,1);
			\draw (4,0.5)  -- (4.7,2);
			\draw (4,0.5)  -- (2.5,1);
			\draw (2.5,1)  -- (0,0);
			\draw (4,0.5)  -- (0,0);
			\draw (6.6,1.5)  -- (7,1);
			\draw (7,1)  -- (9,3.6);
			\draw (7,1)  -- (9,2.4);
			\draw (5,-1.5)  -- (7,1);
			\draw (5,-1.5)  -- (9,2.4);
			\draw (5,-1.5)  -- (6.6,1.5);
			\draw (5,-1.5)  -- (5.2,1);
			\draw (5,-1.5)  -- (4,0.5);
			\draw (5,-1.5)  -- (0,0);
			\draw (5,-1.5)  -- (9,1.2);
			\draw (5,-1.5)  -- (6.5,-1.5);
			\draw (5,-1.5)  -- (4.2,-2);
			\draw (5,-1.5)  -- (1,-1);
			\draw (9,1.2)  -- (6.5,-1.5);
			\draw (9,0)  -- (6.5,-1.5);
			\draw (8.19,-1)  -- (6.5,-1.5);
			\draw (5.4,-2.5)  -- (6.5,-1.5);
			\draw (4.2,-2)  -- (6.5,-1.5);
			\draw (4.2,-2)  -- (5.4,-2.5);
			\draw (8.19,-1)  -- (5.4,-2.5);
			\draw (7.39,-2)  -- (5.4,-2.5);
			\draw (5.4,-3)  -- (5.4,-2.5);
			\draw (4.2,-3)  -- (5.4,-2.5);
			\draw (5.4,-3)  -- (7.39,-2);
			\draw (4.2,-3)  -- (4.2,-2);
			\draw (3,-3)  -- (4.2,-2);
			\draw (2,-2)  -- (4.2,-2);
			\draw (1,-1)  -- (4.2,-2);
			\draw (1,-1)  -- (5,-1.5);
			\draw [line width=0.3mm, color=blue](4.2,-3) -- (5.4,-2.5);
			\draw (2,2.53) node[circle,fill,inner sep=1pt] {};
			\draw (5.4,-2.5) node[circle,fill,color=black,inner sep=2pt] {};
			\node[inner sep=0,anchor=west,text width=3.3cm] (note1) at (5.3,-2.3) {$w$};
			\draw[fill=black] (5.4,4)--(8.1,4.5)--(7.5,5.1)--(5.4,4);
			\draw[fill=black!20] (5.4,4)--(8.1,4.5)--(9,3.6)--(5.4,4);
			\draw[fill=black!20] (5.4,4)--(6,3)--(9,3.6)--(5.4,4);
			\draw[fill=black!20] (6,3)--(6.6,1.5)--(9,3.6)--(6,3);
			\draw[fill=black!20] (6.6,1.5)--(7,1)--(9,3.6)--(6.6,1.5);
			\draw[fill=black!20] (7,1)--(9,2.4)--(9,3.6)--(7,1);
			\draw[fill=black!20] (7,1)--(5,-1.5)--(9,2.4)--(7,1);
			\draw[fill=black!20] (5,-1.5)--(9,1.2)--(9,2.4)--(5,-1.5);
			\draw[fill=black!20] (5,-1.5)--(9,1.2)--(6.5,-1.5)--(5,-1.5);
			\draw[fill=black!20] (5,-1.5)--(6.5,-1.5)--(4.2,-2)--(5,-1.5);
			\draw[fill=black!20] (5,-1.5)--(4.2,-2)--(1,-1)--(5,-1.5);
			\draw[fill=black!20] (5,-1.5)--(1,-1)--(0,0)--(5,-1.5);
			\draw[fill=black!20] (5,-1.5)--(0,0)--(4,0.5)--(5,-1.5);
			\draw[fill=black!20] (0,0)--(4,0.5)--(2.5,1)--(0,0);
			\draw[fill=black!20] (0,0)--(3,2)--(2.5,1)--(0,0);
			\draw[fill=black!20] (0,0)--(3,2)--(0,1.2)--(0,0);
			\draw[fill=black!20] (3,2)--(2,2.53)--(0,1.2)--(3,2);
			\draw[fill=black!20] (3,2)--(4.2,4)--(2,2.53)--(3,2);
			\draw[fill=black!20] (3,2)--(5.4,4)--(4.2,4)--(3,2);
			\draw[fill=black!20] (4.2,4)--(5.4,4)--(5.4,5)--(4.2,4);
			\draw[fill=black!20] (5.4,4)--(5.4,5)--(7.5,5.1)--(5.4,4);
			\draw[fill=black!20] (5.4,5)--(7.5,5.1)--(6.6,6)--(5.4,5);
			\draw[fill=black!20] (1.87,5.1)--(4.2,6)--(3,6)--(1.87,5.1);
			\draw[fill=black!20] (1.87,5.1)--(4.2,6)--(4.2,5)--(1.87,5.1);
			\draw[fill=black!20] (1.87,5.1)--(4.2,5)--(0.87,4.3)--(1.87,5.1);
			\draw[fill=black!20] (0.87,4.3)--(4.2,4)--(4.2,5)--(0.87,4.3);
			\draw[fill=black!20] (4.2,4)--(5.4,5)--(4.2,5)--(4.2,4);
			\draw[fill=black!20] (0,3.6)--(4.2,4)--(0,2.4)--(0,3.6);
			\draw[fill=black!20] (0,2.4)--(4.2,4)--(0,1.2)--(0,2.4);
			\draw [line width=0.3mm, color=red](0,1.2) -- (2,2.53);
			\draw  [line width=0.3mm, color=blue](5.4,-3) -- (6.6,-3);
			\draw  [line width=0.3mm, color=red](4.2,4) -- (5.4,5);
			\node[inner sep=0,anchor=west,text width=3.3cm] (note1) at (5.5,4.5) {$a$};
			\draw[dotted]  [line width=0.5mm](7.5,5.1) -- (9,3.6);
			\draw[dotted]  [line width=0.5mm](9,3.6) -- (9,1.2);
			\draw[dotted]  [line width=0.5mm](9,1.2) -- (6.5,-1.5);
			\draw[dotted]  [line width=0.5mm](6.5,-1.5) -- (4.2,-2);
			\draw[dotted]  [line width=0.5mm](4.2,-2) -- (1,-1);
			\draw[dotted]  [line width=0.5mm](1,-1) -- (0,0);
			\draw[dotted]  [line width=0.5mm](0,0) -- (0,2.4);
			\draw[dotted]  [line width=0.5mm](0.87,4.3) -- (4.2,4);
			\draw[dotted]  [line width=0.5mm](4.2,5) -- (5.4,5);
			\draw[dotted]  [line width=0.5mm](4.2,5) -- (4.2,6);
			\draw[dotted]  [line width=0.5mm](4.2,6) -- (6.6,6);
			\draw[dotted]  [line width=0.5mm](5.4,5) -- (6.6,6);
			\draw[dotted]  [line width=0.5mm](3,2) -- (5.4,4);
			\draw[dotted]  [line width=0.5mm](3,2) -- (2.5,1);
			\draw[dotted]  [line width=0.5mm](2.5,1) -- (4,0.5);
			\draw[dotted]  [line width=0.5mm](4,0.5) -- (5,-1.5);
			\draw[dotted]  [line width=0.5mm](5,-1.5) -- (7,1);
			\draw[dotted]  [line width=0.5mm](7,1) -- (6.6,1.5);
			\draw[dotted]  [line width=0.5mm](6.6,1.5) -- (6,3);
			\draw[dotted]  [line width=0.5mm](6,3) -- (5.4,4);
			\draw[dotted]  [line width=0.5mm](2,2.53) -- (4.2,4);
			\draw[dotted]  [line width=0.5mm](0,3.6) -- (4.2,4);
			\draw[dotted]  [line width=0.5mm](5.4,4) -- (5.4,5);
			\draw[thick,->] (2.5,5.6) -- (2.8,5.6);
			\draw[thick,->] (3.17,5.6) -- (3.37,5.37);
			\draw[thick,->] (4.2,5.5) -- (4.5,5.5);
			\draw[thick,->] (4.2,5) -- (4.459,5.23);
			\draw[thick,->] (4.8,5) -- (4.8,5.3);
			\draw[thick,->] (5.4,5.5) -- (5.7,5.5);
			\draw[thick,->] (5.4,5) -- (5.659,5.23);
			\draw[thick,->] (6.45,5.05) -- (6.45,5.35);
			\draw[thick,->] (3.2,5.045) -- (2.9,4.9);
			\draw[thick,->] (2.7,4.68) -- (2.8,4.38);
			\draw[thick,->] (2.69,4.138) -- (2.33,4.03);
			\draw[thick,->] (4.2,4) -- (3.7,3.95);
			\draw[thick,->] (0,3) -- (0.3,3);
			\draw[thick,->] (1.9,3.125) -- (2.2,2.99);
			\draw[thick,->] (1.9,1.7) -- (2.1,2.05);
			\draw[thick,->] (1.5,1) -- (1.3,1.2);
			\draw[thick,->] (1.9,0.75) -- (1.7,1);
			\draw[thick,->] (2.35,0.295) -- (2.35,0.599);
			\draw[thick,->] (3,-0.9) -- (3.05,-0.5);
			\draw[thick,->] (2.2,-1.15) -- (2.2,-0.8);
			\draw[thick,->] (2.6,-1.5) -- (2.4,-1.8);
			\draw[thick,->] (3.05,-2) -- (3.05,-2.3);
			\draw[thick,->] (3.6,-2.5)-- (3.8,-2.75);
			\draw[thick,->] (4.2,-2.5) -- (4.5,-2.5);
			\draw[thick,->] (4.2,-3) -- (4.5,-3);
			\draw[thick,->] (5.4,-2.75) -- (5.1,-2.75);
			\draw[thick,->] (6.93,-2.6) -- (6.6,-2.6);
			\draw[thick,->] (6.3,-2.55) -- (6,-2.5);
			\draw[thick,->] (6.5,-2.22) -- (6.7,-1.95);
			\draw[thick,->] (6.7,-1.8) -- (6.5,-1.6);
			\draw[thick,->] (7.4,-1.23) -- (7.6,-1);
			\draw[thick,->] (8,-0.6) -- (8.1,-0.3);
			\draw[thick,->] (9,1.2) -- (8.7,0.88);
			\draw[thick,->] (6.6,-0.42) -- (6.6,-0.72);
			\draw[thick,->] (7,0.45) -- (6.9,0.1);
			\draw[thick,->] (7.8,1.56) -- (7.65,1.25);
			\draw[thick,->] (7,1) -- (6.77,0.7);
			\draw[thick,->] (6.85,1.2) -- (6.55,0.9);
			\draw[thick,->] (8,2.3) -- (7.9,1.9);
			\draw[thick,->] (7.8,2.555) -- (7.7,2.2);
			\draw[thick,->] (7.5,3.3) -- (7.4,3);
			\draw[thick,->] (7.3,3.8) -- (7.2,3.5);
			\draw[thick,->] (7.4,4.38) -- (7.3,4.08);
			\draw[thick,->] (6.4,4.53) -- (6.3,4.8);
			\draw[thick,->] (5.4,4) -- (5.4,4.35);
			\draw[thick,->] (4.8,4) -- (4.8,4.35);
			\draw[thick,->] (4.2,4.5) -- (4.55,4.5);
			\draw[thick,->] (3.6,3) -- (3.95,3.1);
			\draw[thick,->] (3,2) -- (3.3,2.25);
			\draw[thick,->] (2.75,1.5) -- (3.1,1.8);
			\draw[thick,->] (2.5,1) -- (2.83,1.34);
			\draw[thick,->] (3.1,0.8) -- (3.4,1);
			\draw[thick,->] (3.6,1.5) -- (3.8,1.79);
			\draw[thick,->] (4.7,2) -- (4.82,2.32);
			\draw[thick,->] (4,0.5) -- (4.32,0.64);
			\draw[thick,->] (4.33,1.2) -- (4.6,1.3);
			\draw[thick,->] (4.9,1.6) -- (5,1.9);
			\draw[thick,->] (5.2,1) -- (5.5,1.11);
			\draw[thick,->] (5.6,2.02) -- (5.9,1.9);
			\draw[thick,->] (5.3,2.5) -- (5.55,2.7);
			\draw[thick,->] (6.6,1.5) -- (6.48,1.8);
			\draw[thick,->] (6,3) -- (5.82,3.3);
			\draw[thick,->] (5.9,0.2) -- (5.6,0.3);
			\draw[thick,->] (5.13,0.05) -- (4.85,0.2);
			\draw[thick,->] (5,-1.5) -- (4.815,-1.13);
			\draw[thick,->] (6.5,-1.5) -- (6.23,-1.75);
			\draw[thick,->] (5.4,-1.74) -- (5.4,-2.05);
			\draw[thick,->] (5.75,-1.503) -- (5.42,-1.65);
			\draw[thick,->] (4.61,-1.75) -- (4.3,-1.65);
			\draw[thick,->] (2,-2) -- (1.75,-1.75);
			\draw[thick,->] (1,-1) -- (0.75,-0.75);
			\draw[thick,->] (4.2,6) -- (4.55,6);
			\draw[thick,->] (5.4,6) -- (5.75,6);
			\draw[thick,->] (7.5,5.1) -- (7.75,4.85);
			\draw[thick,->] (8.1,4.5) -- (8.35,4.25);
			\draw[thick,->] (9,3.6) -- (9,3.25);
			\draw[thick,->] (9,2.4) -- (9,2.05);
			\draw[thick,->] (8.19,-1) -- (8.41,-0.75);
			\draw[thick,->] (7.39,-2) -- (7.59,-1.75);
			\draw[thick,->] (4.2,-2) -- (4.5,-2.12);
			\draw[thick,->] (2,2.53) -- (2.3,2.74);
			\draw[thick,->] (2.43,2.3) -- (2.73,2.5);
			\node[inner sep=0,anchor=west,text width=2.5cm] (note1) at (1.7,2.73) {$10$};
			\node[inner sep=0,anchor=west,text width=2.5cm] (note1) at (4.22,3.8) {$11$};
			\node[inner sep=0,anchor=west,text width=3.3cm] (note1) at (-0.3,0) {$3$};
			\node[inner sep=0,anchor=west,text width=3.3cm] (note1) at (-0.3,3.6) {$3$}; 
			\node[inner sep=0,anchor=west,text width=3.3cm] (note1) at (2.7,6.1) {$3$};
			\node[inner sep=0,anchor=west,text width=3.3cm] (note1) at (6.8,6.1) {$3$};
			\node[inner sep=0,anchor=west,text width=3.3cm] (note1) at (9.2,3.6) {$3$};
			\node[inner sep=0,anchor=west,text width=3.3cm] (note1) at (9.2,0) {$3$};
			\node[inner sep=0,anchor=west,text width=3.3cm] (note1) at (6.8,-3.1) {$3$};
			\node[inner sep=0,anchor=west,text width=3.3cm] (note1) at (2.7,-3.1) {$3$};
			\node[inner sep=0,anchor=west,text width=3.3cm] (note1) at (0.47,4.3) {$1$};
			\node[inner sep=0,anchor=west,text width=3.3cm] (note1) at (1.47,5.1) {$2$};
			\node[inner sep=0,anchor=west,text width=3.3cm] (note1) at (4.2,6.2) {$4$};
			\node[inner sep=0,anchor=west,text width=3.3cm] (note1) at (5.4,6.2) {$5$};
			\node[inner sep=0,anchor=west,text width=3.3cm] (note1) at (7.7,5.1) {$2$};
			\node[inner sep=0,anchor=west,text width=3.3cm] (note1) at (8.3,4.5) {$1$};
			\node[inner sep=0,anchor=west,text width=3.3cm] (note1) at (9.2,2.4) {$6$};
			\node[inner sep=0,anchor=west,text width=3.3cm] (note1) at (9.2,1.2) {$7$};
			\node[inner sep=0,anchor=west,text width=3.3cm] (note1) at (8.39,-1) {$8$};
			\node[inner sep=0,anchor=west,text width=3.3cm] (note1) at (7.59,-2) {$9$};
			\node[inner sep=0,anchor=west,text width=3.3cm] (note1) at (5.4,-3.2) {$7$};
			\node[inner sep=0,anchor=west,text width=3.3cm] (note1) at (4.2,-3.2) {$6$};
			\node[inner sep=0,anchor=west,text width=3.3cm] (note1) at (1.6,-2) {$9$};
			\node[inner sep=0,anchor=west,text width=3.3cm] (note1) at (0.6,-1) {$8$};
			\node[inner sep=0,anchor=west,text width=3.3cm] (note1) at (-0.3,1.2) {$5$};
			\node[inner sep=0,anchor=west,text width=3.3cm] (note1) at (-0.3,2.4) {$4$};
			\end{tikzpicture}
		} %
		\caption{The $2$-paths end at the critical edges.}
		\label{fig:torusflow} 
	\end{figure}
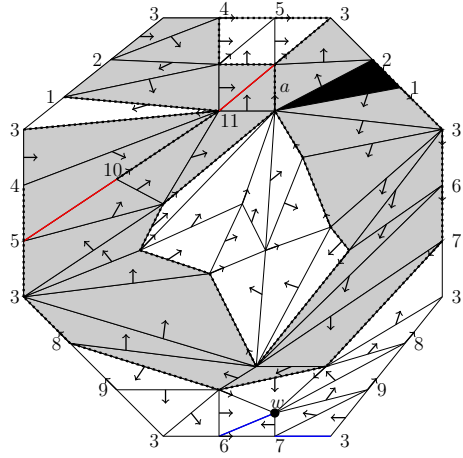
	The gray triangles show the $2$-paths that end at the critical edges in Figure \ref{fig:torusflow}.  
	
	First, we delete the interior of the critical $2$-cell and then begin to collapse free edges and triangles that paired 
	together in these $2$-paths up to the critical edges.  The gray coloured region is $M-M_{1}$, while the white region is 
	$M-M_{2}$.  The dotted curve is the boundary curve and the edge $a$ is the connecting edge (additional edge) between two circles.  Note that the connecting edge $a$ does not belong to the boundary curve, that is, $a\in Int(M-M_{1})$.  Note also that the edges $2-1, 10-11$ are not boundary components.  They are interior edges of $M-M_{1}$.  In addition, red critical edges are the interior edges of $M-M_{1}$.  Therefore, boundary curve is a disjoint union of a circle and a wedge of three circle at the point $3$ as on the figure left in Figure \ref{fig:wedge}.  These disjoint curve are connected in $M-M_{1}$ via the edge $a$ as in the figure on the right in Figure \ref{fig:wedge}.
	 
	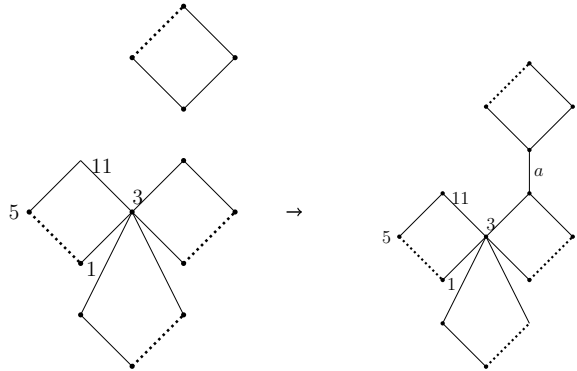
\begin{figure}[h]
		\resizebox{0.25\textwidth}{!}{%
		\begin{tikzpicture}
		
		\draw (0,3)--(1,2)--(2,3)--(1,4);
		\draw [dotted] [line width=0.5mm](0,3) -- (1,4);
		\draw (1,1)--(0,0)--(1,-1);
		\draw (1,1)--(2,0);
		\draw [dotted] [line width=0.5mm](1,-1) -- (2,0);
		\node[inner sep=0,anchor=west,text width=3.3cm] (note1) at (0,0.3) {$3$};
		\node[inner sep=0,anchor=west,text width=3.3cm] (note1) at (-2.4,0) {$5$};
		\node[inner sep=0,anchor=west,text width=3.3cm] (note1) at (-0.9,-1.1) {$1$};
		\node[inner sep=0,anchor=west,text width=3.3cm] (note1) at (-0.8,0.9) {$11$};
		\draw (1,4) node[circle,fill,inner sep=1pt] {};
		\draw (0,3) node[circle,fill,inner sep=1pt] {};
		\draw (2,3) node[circle,fill,inner sep=1pt] {};
		\draw (1,2) node[circle,fill,inner sep=1pt] {};
		\draw (1,1) node[circle,fill,inner sep=1pt] {};
		\draw (2,0) node[circle,fill,inner sep=1pt] {};
		\draw (1,-1) node[circle,fill,inner sep=1pt] {};
		\draw (0,0) node[circle,fill,inner sep=1pt] {};
		\draw (1,-2) node[circle,fill,inner sep=1pt] {};
		\draw (0,-3) node[circle,fill,inner sep=1pt] {};
		\draw (-1,-2) node[circle,fill,inner sep=1pt] {};
		\draw (-1,-1) node[circle,fill,inner sep=1pt] {};
		\draw (-2,0) node[circle,fill,inner sep=1pt] {};
		\draw (0,0) -- (-1,1);
		\draw (0,0) -- (-1,-1);
		\draw (-1,1) -- (-2,0);
		\draw [dotted] [line width=0.5mm](-2,0) -- (-1,-1);
		\draw (0,0) -- (-1,-2)--(0,-3);
		\draw (0,0) -- (1,-2);
		\draw [dotted] [line width=0.5mm](0,-3) -- (1,-2);
		\draw[thick,->] (3,0) -- (3.3,0);
		\end{tikzpicture}
	} %
		\qquad
		\resizebox{0.25\textwidth}{!}{%
			\begin{tikzpicture}
			\draw (0,3)--(1,2)--(2,3)--(1,4);
			\draw [dotted] [line width=0.5mm](0,3) -- (1,4);
			\draw (1,1)--(1,2);
			\draw (1,1)--(0,0)--(1,-1);
			\draw (1,1)--(2,0);
			\draw [dotted] [line width=0.5mm](1,-1) -- (2,0);
			\node[inner sep=0,anchor=west,text width=3.3cm] (note1) at (1.1,1.5) {$a$};
			\node[inner sep=0,anchor=west,text width=3.3cm] (note1) at (0,0.3) {$3$};
			\node[inner sep=0,anchor=west,text width=3.3cm] (note1) at (-2.4,0) {$5$};
			\node[inner sep=0,anchor=west,text width=3.3cm] (note1) at (-0.9,-1.1) {$1$};
			\node[inner sep=0,anchor=west,text width=3.3cm] (note1) at (-0.8,0.9) {$11$};
			\draw (1,4) node[circle,fill,inner sep=1pt] {};
			\draw (0,3) node[circle,fill,inner sep=1pt] {};
			\draw (2,3) node[circle,fill,inner sep=1pt] {};
			\draw (1,2) node[circle,fill,inner sep=1pt] {};
			\draw (1,1) node[circle,fill,inner sep=1pt] {};
			\draw (2,0) node[circle,fill,inner sep=1pt] {};
			\draw (1,-1) node[circle,fill,inner sep=1pt] {};
			\draw (0,0) node[circle,fill,inner sep=1pt] {};
			\draw (0,-3) node[circle,fill,inner sep=1pt] {};
			\draw (-1,-2) node[circle,fill,inner sep=1pt] {};
			\draw (-1,-1) node[circle,fill,inner sep=1pt] {};
			\draw (-2,0) node[circle,fill,inner sep=1pt] {};
			\draw (-1,1) node[circle,fill,inner sep=1pt] {};
			\draw (0,0) -- (-1,1);
			\draw (0,0) -- (-1,-1);
			\draw (-1,1) -- (-2,0);
			\draw [dotted] [line width=0.5mm](-2,0) -- (-1,-1);
			\draw (0,0) -- (-1,-2)--(0,-3);
			\draw (0,0) -- (1,-2);
			\draw [dotted] [line width=0.5mm](0,-3) -- (1,-2);
			\end{tikzpicture}
		} %
		\caption{The boundary curve.}
		\label{fig:wedge} 
	\end{figure}
	The figure on the left is the curve obtained after collapsing free edges and triangles up to critical edges.  
	If there are any free edges on this curve, we collapse them.  We obtain the figure on the right in Figure \ref{fig:wedge},  
	after deleting critical edges and collapsing free edges.  Then, we get rid of the connecting edge $a$ by using the method in 
	Figures \ref{fig:connecting edge}, \ref{fig:circle} (see Figure \ref{fig:torusflow2}).  
	\begin{figure}[h]
		\centering
		\resizebox{0.5\textwidth}{!}{%
			\begin{tikzpicture}
			\draw (0,0)  -- (0,3.6) -- (3,6)-- (6.6,6)-- (9,3.6)-- (9,0)-- (6.6,-3)-- (3,-3)-- (0,0);
			\draw (3,2)  -- (2,2.53);
			\draw (1.87,5.1)  -- (4.2,6);
			\draw (1.87,5.1)  -- (4.2,5);
			\draw (4.2,6)  -- (4.2,5);
			\draw (4.2,5)  -- (5.4,6);
			\draw (4.2,5)  -- (5.4,5);
			\draw (5.4,5)  -- (6.6,6);
			\draw (5.4,5)  -- (5.4,6);
			\draw (5.4,5)  -- (7.5,5.1);
			\draw (0.87,4.3)  -- (4.2,5);
			\draw (0.87,4.3)  -- (4.2,4);
			\draw (4.2,4)  -- (4.2,5);
			\draw (4.2,4)  -- (5.4,5);
			\draw (4.2,4)  -- (5.4,4);
			\draw (5.4,4)  -- (5.4,5);
			\draw (5.4,4)  -- (7.5,5.1);
			\draw (5.4,4)  -- (8.1,4.5);
			\draw (5.4,4)  -- (9,3.6);
			\draw (5.4,4)  -- (6,3);
			\draw (5.4,4)  -- (5.2,1);
			\draw (5.4,4)  -- (4.7,2);
			\draw (5.4,4)  -- (2.5,1);
			\draw (5.4,4)  -- (3,2);
			\draw (3,2)  -- (4.2,4);
			\draw (4.2,4)  -- (2,2.53);
			\draw (4.2,4)  -- (0,2.4);
			\draw (4.2,4)  -- (0,3.6);
			\draw (0,1.2)  -- (3,2);
			\draw (0,0)  -- (3,2);
			\draw (3,2)  -- (2.5,1);
			\draw (2.5,1)  -- (4.7,2);
			\draw (4.7,2)  -- (5.2,1);
			\draw (5.2,1)  -- (6,3);
			\draw (6,3)  -- (9,3.6);
			\draw (6,3)  -- (6.6,1.5);
			\draw (6.6,1.5)  -- (9,3.6);
			\draw (5.2,1)  -- (6.6,1.5);
			\draw (4,0.5)  -- (5.2,1);
			\draw (4,0.5)  -- (4.7,2);
			\draw (4,0.5)  -- (2.5,1);
			\draw (2.5,1)  -- (0,0);
			\draw (4,0.5)  -- (0,0);
			\draw (6.6,1.5)  -- (7,1);
			\draw (7,1)  -- (9,3.6);
			\draw (7,1)  -- (9,2.4);
			\draw (5,-1.5)  -- (7,1);
			\draw (5,-1.5)  -- (9,2.4);
			\draw (5,-1.5)  -- (6.6,1.5);
			\draw (5,-1.5)  -- (5.2,1);
			\draw (5,-1.5)  -- (4,0.5);
			\draw (5,-1.5)  -- (0,0);
			\draw (5,-1.5)  -- (9,1.2);
			\draw (5,-1.5)  -- (6.5,-1.5);
			\draw (5,-1.5)  -- (4.2,-2);
			\draw (5,-1.5)  -- (1,-1);
			\draw (9,1.2)  -- (6.5,-1.5);
			\draw (9,0)  -- (6.5,-1.5);
			\draw (8.19,-1)  -- (6.5,-1.5);
			\draw (5.4,-2.5)  -- (6.5,-1.5);
			\draw (4.2,-2)  -- (6.5,-1.5);
			\draw (4.2,-2)  -- (5.4,-2.5);
			\draw (8.19,-1)  -- (5.4,-2.5);
			\draw (7.39,-2)  -- (5.4,-2.5);
			\draw (5.4,-3)  -- (5.4,-2.5);
			\draw (4.2,-3)  -- (5.4,-2.5);
			\draw (5.4,-3)  -- (7.39,-2);
			\draw (4.2,-3)  -- (4.2,-2);
			\draw (3,-3)  -- (4.2,-2);
			\draw (2,-2)  -- (4.2,-2);
			\draw (1,-1)  -- (4.2,-2);
			\draw (1,-1)  -- (5,-1.5);
			\draw (4.8,4.5) -- (5.4,5);
			\draw (3,2)  -- (4.8,4);
			\draw (4.8,4)  -- (4.8,4.5);
			\draw (4.2,5)  -- (4.8,4.5);
			
			\draw [line width=0.3mm, color=blue](4.2,-3) -- (5.4,-2.5);
			\draw (2,2.53) node[circle,fill,inner sep=1pt] {};
			\draw (5.4,-2.5) node[circle,fill,color=black,inner sep=2pt] {};
			\node[inner sep=0,anchor=west,text width=3.3cm] (note1) at (5.3,-2.3) {$w$};
			\draw[fill=black] (6.6,4.23)--(8.1,4.5)--(7.5,5.1)--(6.2,4.43)--(6.6,4.23);
			\draw[fill=black!20] (6.6,4.23)--(8.1,4.5)--(9,3.6)--(6.6,3.87)--(6.6,4.23);
			\draw[fill=black!20] (6.6,3.87)--(6,3)--(9,3.6)--(6.6,3.87);
			\draw[fill=black!20] (6,3)--(6.6,1.5)--(9,3.6)--(6,3);
			\draw[fill=black!20] (6.6,1.5)--(7,1)--(9,3.6)--(6.6,1.5);
			\draw[fill=black!20] (7,1)--(9,2.4)--(9,3.6)--(7,1);
			\draw[fill=black!20] (7,1)--(5,-1.5)--(9,2.4)--(7,1);
			\draw[fill=black!20] (5,-1.5)--(9,1.2)--(9,2.4)--(5,-1.5);
			\draw[fill=black!20] (5,-1.5)--(9,1.2)--(6.5,-1.5)--(5,-1.5);
			\draw[fill=black!20] (5,-1.5)--(6.5,-1.5)--(4.2,-2)--(5,-1.5);
			\draw[fill=black!20] (5,-1.5)--(4.2,-2)--(1,-1)--(5,-1.5);
			\draw[fill=black!20] (5,-1.5)--(1,-1)--(0,0)--(5,-1.5);
			\draw[fill=black!20] (5,-1.5)--(0,0)--(4,0.5)--(5,-1.5);
			\draw[fill=black!20] (0,0)--(4,0.5)--(2.5,1)--(0,0);
			\draw[fill=black!20] (0,0)--(3,2)--(2.5,1)--(0,0);
			\draw[fill=black!20] (0,0)--(3,2)--(0,1.2)--(0,0);
			\draw[fill=black!20] (3,2)--(2,2.53)--(0,1.2)--(3,2);
			\draw[fill=black!20] (3,2)--(4.2,4)--(2,2.53)--(3,2);
			\draw[fill=black!20] (3,2)--(4.8,4)--(4.2,4)--(3,2);
			\draw[fill=black!20] (4.2,4)--(4.8,4)--(4.8,4.5)--(4.2,4);
			\draw[fill=black!20] (4.2,4)--(4.8,4.5)--(4.2,5)--(4.2,4);
			\draw[fill=black!20] (6.05,5.03)--(7.5,5.1)--(6.2,4.43)--(6.05,5.03);
			\draw[fill=black!20] (6.05,5.03)--(7.5,5.1)--(6.6,6)--(6.05,5.03);
			\draw[fill=black!20] (1.87,5.1)--(4.2,6)--(3,6)--(1.87,5.1);
			\draw[fill=black!20] (1.87,5.1)--(4.2,6)--(4.2,5)--(1.87,5.1);
			\draw[fill=black!20] (1.87,5.1)--(4.2,5)--(0.87,4.3)--(1.87,5.1);
			\draw[fill=black!20] (0.87,4.3)--(4.2,4)--(4.2,5)--(0.87,4.3);
			\draw[fill=black!20] (0,3.6)--(4.2,4)--(0,2.4)--(0,3.6);
			\draw[fill=black!20] (0,2.4)--(4.2,4)--(0,1.2)--(0,2.4);
			\draw [line width=0.3mm, color=red](0,1.2) -- (2,2.53);
			\draw [line width=0.3mm, color=blue](5.4,-3) -- (6.6,-3);
			\draw [line width=0.3mm, color=red](4.2,4) -- (4.8,4.5);
			\node[inner sep=0,anchor=west,text width=3.3cm] (note1) at (5.5,4.5) {$a$};
			\draw[dotted]  [line width=0.5mm](7.5,5.1) -- (9,3.6);
			\draw[dotted]  [line width=0.5mm](9,3.6) -- (9,1.2);
			\draw[dotted]  [line width=0.5mm](9,1.2) -- (6.5,-1.5);
			\draw[dotted]  [line width=0.5mm](6.5,-1.5) -- (4.2,-2);
			\draw[dotted]  [line width=0.5mm](4.2,-2) -- (1,-1);
			\draw[dotted]  [line width=0.5mm](1,-1) -- (0,0);
			\draw[dotted]  [line width=0.5mm](0,0) -- (0,2.4);
			\draw[dotted]  [line width=0.5mm](0.87,4.3) -- (4.2,4);
			\draw[dotted]  [line width=0.5mm](4.2,5) -- (4.8,4.5);
			\draw[dotted]  [line width=0.5mm](4.2,5) -- (4.2,6);
			\draw[dotted]  [line width=0.5mm](4.2,6) -- (6.6,6);
			\draw[dotted]  [line width=0.5mm](3,2) -- (4.8,4);
			\draw[dotted]  [line width=0.5mm](3,2) -- (2.5,1);
			\draw[dotted]  [line width=0.5mm](2.5,1) -- (4,0.5);
			\draw[dotted]  [line width=0.5mm](4,0.5) -- (5,-1.5);
			\draw[dotted]  [line width=0.5mm](5,-1.5) -- (7,1);
			\draw[dotted]  [line width=0.5mm](7,1) -- (6.6,1.5);
			\draw[dotted]  [line width=0.5mm](6.6,1.5) -- (6,3);
			\draw[dotted]  [line width=0.5mm](2,2.53) -- (4.2,4);
			\draw[dotted]  [line width=0.5mm](0,3.6) -- (4.2,4);
			\draw[dotted]  [line width=0.5mm](4.8,4) -- (4.8,4.5);
			\draw (4.8,4) node[circle,fill,inner sep=1pt] {};
			\draw (4.8,4.5) node[circle,fill,inner sep=1pt] {};
			\draw (2,2.53) node[circle,fill,inner sep=1pt] {};
			\draw[thick,->] (4.35,3.5) -- (4.7,3.6);
			\draw[thick,->] (4.5,4.75) -- (4.8,4.75);
			\draw[thick,->] (2.5,5.6) -- (2.8,5.6);
			\draw[thick,->] (3.17,5.6) -- (3.37,5.37);
			\draw[thick,->] (4.2,5.5) -- (4.5,5.5);
			\draw[thick,->] (4.2,5) -- (4.459,5.23);
			\draw[thick,->] (4.8,5) -- (4.8,5.3);
			\draw[thick,->] (5.4,5.5) -- (5.7,5.5);
			\draw[thick,->] (5.4,5) -- (5.659,5.23);
			\draw[thick,->] (6.45,5.05) -- (6.45,5.35);
			\draw[thick,->] (3.2,5.045) -- (2.9,4.9);
			\draw[thick,->] (2.7,4.68) -- (2.8,4.38);
			\draw[thick,->] (2.69,4.138) -- (2.33,4.03);
			\draw[thick,->] (4.2,4) -- (3.7,3.95);
			\draw[thick,->] (0,3) -- (0.3,3);
			\draw[thick,->] (1.9,3.125) -- (2.2,2.99);
			\draw[thick,->] (1.9,1.7) -- (2.1,2.05);
			\draw[thick,->] (1.5,1) -- (1.3,1.2);
			\draw[thick,->] (1.9,0.75) -- (1.7,1);
			\draw[thick,->] (2.35,0.295) -- (2.35,0.599);
			\draw[thick,->] (3,-0.9) -- (3.05,-0.5);
			\draw[thick,->] (2.2,-1.15) -- (2.2,-0.8);
			\draw[thick,->] (2.6,-1.5) -- (2.4,-1.8);
			\draw[thick,->] (3.05,-2) -- (3.05,-2.3);
			\draw[thick,->] (3.6,-2.5)-- (3.8,-2.75);
			\draw[thick,->] (4.2,-2.5) -- (4.5,-2.5);
			\draw[thick,->] (4.2,-3) -- (4.5,-3);
			\draw[thick,->] (5.4,-2.75) -- (5.1,-2.75);
			\draw[thick,->] (6.93,-2.6) -- (6.6,-2.6);
			\draw[thick,->] (6.3,-2.55) -- (6,-2.5);
			\draw[thick,->] (6.5,-2.22) -- (6.7,-1.95);
			\draw[thick,->] (6.7,-1.8) -- (6.5,-1.6);
			\draw[thick,->] (7.4,-1.23) -- (7.6,-1);
			\draw[thick,->] (8,-0.6) -- (8.1,-0.3);
			\draw[thick,->] (9,1.2) -- (8.7,0.88);
			\draw[thick,->] (6.6,-0.42) -- (6.6,-0.72);
			\draw[thick,->] (7,0.45) -- (6.9,0.1);
			\draw[thick,->] (7.8,1.56) -- (7.65,1.25);
			\draw[thick,->] (7,1) -- (6.77,0.7);
			\draw[thick,->] (6.85,1.2) -- (6.55,0.9);
			\draw[thick,->] (8,2.3) -- (7.9,1.9);
			\draw[thick,->] (7.8,2.555) -- (7.7,2.2);
			\draw[thick,->] (7.5,3.3) -- (7.4,3);
			\draw[thick,->] (7.3,3.8) -- (7.2,3.5);
			\draw[thick,->] (7.4,4.38) -- (7.3,4.08);
			\draw[thick,->] (6.4,4.53) -- (6.3,4.8);
			\draw[thick,->] (5.4,4) -- (5.4,4.35);
			\draw[thick,->] (4.2,4.5) -- (4.55,4.5);
			\draw[thick,->] (3.6,3) -- (3.95,3.1);
			\draw[thick,->] (3,2) -- (3.3,2.25);
			\draw[thick,->] (2.75,1.5) -- (3.1,1.8);
			\draw[thick,->] (2.5,1) -- (2.83,1.34);
			\draw[thick,->] (3.1,0.8) -- (3.4,1);
			\draw[thick,->] (3.6,1.5) -- (3.8,1.79);
			\draw[thick,->] (4.7,2) -- (4.82,2.32);
			\draw[thick,->] (4,0.5) -- (4.32,0.64);
			\draw[thick,->] (4.33,1.2) -- (4.6,1.3);
			\draw[thick,->] (4.9,1.6) -- (5,1.9);
			\draw[thick,->] (5.2,1) -- (5.5,1.11);
			\draw[thick,->] (5.6,2.02) -- (5.9,1.9);
			\draw[thick,->] (5.3,2.5) -- (5.55,2.7);
			\draw[thick,->] (6.6,1.5) -- (6.48,1.8);
			\draw[thick,->] (6,3) -- (5.82,3.3);
			\draw[thick,->] (5.9,0.2) -- (5.6,0.3);
			\draw[thick,->] (5.13,0.05) -- (4.85,0.2);
			\draw[thick,->] (5,-1.5) -- (4.815,-1.13);
			\draw[thick,->] (6.5,-1.5) -- (6.23,-1.75);
			\draw[thick,->] (5.4,-1.74) -- (5.4,-2.05);
			\draw[thick,->] (5.75,-1.503) -- (5.42,-1.65);
			\draw[thick,->] (4.61,-1.75) -- (4.3,-1.65);
			\draw[thick,->] (2,-2) -- (1.75,-1.75);
			\draw[thick,->] (1,-1) -- (0.75,-0.75);
			\draw[thick,->] (4.2,6) -- (4.55,6);
			\draw[thick,->] (5.4,6) -- (5.75,6);
			\draw[thick,->] (7.5,5.1) -- (7.75,4.85);
			\draw[thick,->] (8.1,4.5) -- (8.35,4.25);
			\draw[thick,->] (9,3.6) -- (9,3.25);
			\draw[thick,->] (9,2.4) -- (9,2.05);
			\draw[thick,->] (8.19,-1) -- (8.41,-0.75);
			\draw[thick,->] (7.39,-2) -- (7.59,-1.75);
			\draw[thick,->] (4.2,-2) -- (4.5,-2.12);
			\draw[thick,->] (2,2.53) -- (2.3,2.74);
			\draw[thick,->] (2.43,2.3) -- (2.73,2.5);
			\draw[thick,->] (4.8,4) -- (5.1,4);
			\draw[thick,->] (4.5,4) -- (4.5,4.2);
			\draw[thick,->] (4.8,4.25) -- (5.1,4.25);
			\draw[thick,->] (4.8,4.5) -- (5.1,4.74);
			\node[inner sep=0,anchor=west,text width=2.5cm] (note1) at (1.7,2.73) {$10$};
			\node[inner sep=0,anchor=west,text width=2.5cm] (note1) at (4.22,3.8) {$11$};
			\node[inner sep=0,anchor=west,text width=3.3cm] (note1) at (-0.3,0) {$3$};
			\node[inner sep=0,anchor=west,text width=3.3cm] (note1) at (-0.3,3.6) {$3$}; 
			\node[inner sep=0,anchor=west,text width=3.3cm] (note1) at (2.7,6.1) {$3$};
			\node[inner sep=0,anchor=west,text width=3.3cm] (note1) at (6.8,6.1) {$3$};
			\node[inner sep=0,anchor=west,text width=3.3cm] (note1) at (9.2,3.6) {$3$};
			\node[inner sep=0,anchor=west,text width=3.3cm] (note1) at (9.2,0) {$3$};
			\node[inner sep=0,anchor=west,text width=3.3cm] (note1) at (6.8,-3.1) {$3$};
			\node[inner sep=0,anchor=west,text width=3.3cm] (note1) at (2.7,-3.1) {$3$};
			\node[inner sep=0,anchor=west,text width=3.3cm] (note1) at (0.47,4.3) {$1$};
			\node[inner sep=0,anchor=west,text width=3.3cm] (note1) at (1.47,5.1) {$2$};
			\node[inner sep=0,anchor=west,text width=3.3cm] (note1) at (4.2,6.2) {$4$};
			\node[inner sep=0,anchor=west,text width=3.3cm] (note1) at (5.4,6.2) {$5$};
			\node[inner sep=0,anchor=west,text width=3.3cm] (note1) at (7.7,5.1) {$2$};
			\node[inner sep=0,anchor=west,text width=3.3cm] (note1) at (8.3,4.5) {$1$};
			\node[inner sep=0,anchor=west,text width=3.3cm] (note1) at (9.2,2.4) {$6$};
			\node[inner sep=0,anchor=west,text width=3.3cm] (note1) at (9.2,1.2) {$7$};
			\node[inner sep=0,anchor=west,text width=3.3cm] (note1) at (8.39,-1) {$8$};
			\node[inner sep=0,anchor=west,text width=3.3cm] (note1) at (7.59,-2) {$9$};
			\node[inner sep=0,anchor=west,text width=3.3cm] (note1) at (5.4,-3.2) {$7$};
			\node[inner sep=0,anchor=west,text width=3.3cm] (note1) at (4.2,-3.2) {$6$};
			\node[inner sep=0,anchor=west,text width=3.3cm] (note1) at (1.6,-2) {$9$};
			\node[inner sep=0,anchor=west,text width=3.3cm] (note1) at (0.6,-1) {$8$};
			\node[inner sep=0,anchor=west,text width=3.3cm] (note1) at (-0.3,1.2) {$5$};
			\node[inner sep=0,anchor=west,text width=3.3cm] (note1) at (-0.3,2.4) {$4$};
			\draw (6.6,3.87) node[circle,fill,inner sep=1pt] {};
			\draw (6.6,4.23) node[circle,fill,inner sep=1pt] {};
			\draw (6.2,4.43) node[circle,fill,inner sep=1pt] {};
			\draw (6.05,5.03) node[circle,fill,inner sep=1pt] {};
			\draw[dotted]  [line width=0.5mm](6,3) -- (6.6,3.87);
			\draw[dotted]  [line width=0.5mm](6.6,3.87) -- (6.6,4.23);
			\draw[dotted]  [line width=0.5mm](6.6,4.23) -- (6.2,4.43);
			\draw[dotted]  [line width=0.5mm](6.2,4.43) -- (6.05,5.03);
			\draw[dotted]  [line width=0.5mm](6.05,5.03) -- (6.6,6);
			\draw[thick,->] (6.05,5.03) -- (5.65,5.02);
			\draw[thick,->] (6.2,4.43) -- (5.9,4.27);
			\draw[thick,->] (6.6,4.23) -- (6.2,4.15);
			\draw[thick,->] (6.6,3.87) -- (6.2,3.897);
			\draw[thick,->] (6.3,5.45) -- (5.9,5.23);
			\draw[thick,->] (6.15,4.69) -- (5.8,4.63);
			\draw[thick,->] (6.35,4.33) -- (6.05,4.2);
			\draw[thick,->] (6.6,4.1) -- (6.3,4);
			\draw[thick,->] (6.35,3.5) -- (6,3.7);
			\end{tikzpicture}
		} %
		\caption{}
		\label{fig:torusflow2} 
	\end{figure}
	
	Now, the dotted curve is a wedge of three circles at the point $3$.  We get rid of this wedge point 
	by using the method in Figures \ref{fig:wedge1} and \ref{fig:unwedge1} (see Figure \ref{fig:torusflow3}).  
	We pair all the cells on the new boundary components as we mention in the proof of the Theorem \ref{decompose}.  
	At the end, the boundary curve is a circle as we want.
	\begin{figure}[h]
		\centering
		\resizebox{0.5\textwidth}{!}{%
			\begin{tikzpicture}
			\draw (0,0)  -- (0,3.6) -- (3,6)-- (6.6,6)-- (9,3.6)-- (9,0)-- (6.6,-3)-- (3,-3)-- (0,0);
			\draw (3,2)  -- (2,2.53);
			\draw (1.87,5.1)  -- (4.2,6);
			\draw (1.87,5.1)  -- (4.2,5);
			\draw (4.2,6)  -- (4.2,5);
			\draw (4.2,5)  -- (5.4,6);
			\draw (4.2,5)  -- (5.4,5);
			\draw (5.4,5)  -- (6.6,6);
			\draw (5.4,5)  -- (5.4,6);
			\draw (5.4,5)  -- (7.5,5.1);
			\draw (0.87,4.3)  -- (4.2,5);
			\draw (0.87,4.3)  -- (4.2,4);
			\draw (4.2,4)  -- (4.2,5);
			\draw (4.2,4)  -- (5.4,5);
			\draw (4.2,4)  -- (5.4,4);
			\draw (5.4,4)  -- (5.4,5);
			\draw (5.4,4)  -- (7.5,5.1);
			\draw (5.4,4)  -- (8.1,4.5);
			\draw (5.4,4)  -- (9,3.6);
			\draw (5.4,4)  -- (6,3);
			\draw (5.4,4)  -- (5.2,1);
			\draw (5.4,4)  -- (4.7,2);
			\draw (5.4,4)  -- (2.5,1);
			\draw (5.4,4)  -- (3,2);
			\draw (3,2)  -- (4.2,4);
			\draw (4.2,4)  -- (2,2.53);
			\draw (4.2,4)  -- (0,2.4);
			\draw (4.2,4)  -- (0,3.6);
			\draw (0,1.2)  -- (3,2);
			\draw (0,0)  -- (3,2);
			\draw (3,2)  -- (2.5,1);
			\draw (2.5,1)  -- (4.7,2);
			\draw (4.7,2)  -- (5.2,1);
			\draw (5.2,1)  -- (6,3);
			\draw (6,3)  -- (9,3.6);
			\draw (6,3)  -- (6.6,1.5);
			\draw (6.6,1.5)  -- (9,3.6);
			\draw (5.2,1)  -- (6.6,1.5);
			\draw (4,0.5)  -- (5.2,1);
			\draw (4,0.5)  -- (4.7,2);
			\draw (4,0.5)  -- (2.5,1);
			\draw (2.5,1)  -- (0,0);
			\draw (4,0.5)  -- (0,0);
			\draw (6.6,1.5)  -- (7,1);
			\draw (7,1)  -- (9,3.6);
			\draw (7,1)  -- (9,2.4);
			\draw (5,-1.5)  -- (7,1);
			\draw (5,-1.5)  -- (9,2.4);
			\draw (5,-1.5)  -- (6.6,1.5);
			\draw (5,-1.5)  -- (5.2,1);
			\draw (5,-1.5)  -- (4,0.5);
			\draw (5,-1.5)  -- (0,0);
			\draw (5,-1.5)  -- (9,1.2);
			\draw (5,-1.5)  -- (6.5,-1.5);
			\draw (5,-1.5)  -- (4.2,-2);
			\draw (5,-1.5)  -- (1,-1);
			\draw (9,1.2)  -- (6.5,-1.5);
			\draw (9,0)  -- (6.5,-1.5);
			\draw (8.19,-1)  -- (6.5,-1.5);
			\draw (5.4,-2.5)  -- (6.5,-1.5);
			\draw (4.2,-2)  -- (6.5,-1.5);
			\draw (4.2,-2)  -- (5.4,-2.5);
			\draw (8.19,-1)  -- (5.4,-2.5);
			\draw (7.39,-2)  -- (5.4,-2.5);
			\draw (5.4,-3)  -- (5.4,-2.5);
			\draw (4.2,-3)  -- (5.4,-2.5);
			\draw (5.4,-3)  -- (7.39,-2);
			\draw (4.2,-3)  -- (4.2,-2);
			\draw (3,-3)  -- (4.2,-2);
			\draw (2,-2)  -- (4.2,-2);
			\draw (1,-1)  -- (4.2,-2);
			\draw (1,-1)  -- (5,-1.5);
			\draw (4.8,4.5) -- (5.4,5);
			\draw (3,2)  -- (4.8,4);
			\draw (4.8,4)  -- (4.8,4.5);
			\draw (4.2,5)  -- (4.8,4.5);
			\draw [line width=0.3mm, color=blue](4.2,-3) -- (5.4,-2.5);
			\draw (2,2.53) node[circle,fill,inner sep=1pt] {};
			\draw (5.4,-2.5) node[circle,fill,color=black,inner sep=2pt] {};
			\node[inner sep=0,anchor=west,text width=3.3cm] (note1) at (5.3,-2.3) {$w$};
			\draw[fill=black] (6.6,4.23)--(8.1,4.5)--(7.5,5.1)--(6.2,4.43)--(6.6,4.23);
			\draw[fill=black!20] (6.6,4.23)--(8.1,4.5)--(7.8,3.73)--(6.6,3.87)--(6.6,4.23);
			\draw[fill=black!20] (6.6,3.87)--(6,3)--(7.8,3.36)--(7.8,3.73)--(6.6,3.87);
			\draw[fill=black!20] (6,3)--(6.6,1.5)--(8,2.73)--(7.8,3.36)--(6,3);
			\draw[fill=black!20] (6.6,1.5)--(7,1)--(8.2,2.55)--(8,2.73)--(6.6,1.5);
			\draw[fill=black!20] (7,1)--(9,2.4)--(8.2,2.55)--(7,1);
			\draw[fill=black!20] (7,1)--(5,-1.5)--(9,2.4)--(7,1);
			\draw[fill=black!20] (5,-1.5)--(9,1.2)--(9,2.4)--(5,-1.5);
			\draw[fill=black!20] (5,-1.5)--(9,1.2)--(6.5,-1.5)--(5,-1.5);
			\draw[fill=black!20] (5,-1.5)--(6.5,-1.5)--(4.2,-2)--(5,-1.5);
			\draw[fill=black!20] (5,-1.5)--(4.2,-2)--(1,-1)--(5,-1.5);
			\draw[fill=black!20] (5,-1.5)--(1,-1)--(1.5,-0.45)--(5,-1.5);
			\draw[fill=black!20] (5,-1.5)--(1.5,-0.45)--(1.5,0.185)--(4,0.5)--(5,-1.5);
			\draw[fill=black!20] (1.5,0.185)--(4,0.5)--(2.5,1)--(1.2,0.48)--(1.5,0.185);
			\draw[fill=black!20] (1,0.67)--(3,2)--(2.5,1)--(1.2,0.48)--(1,0.67);
			\draw[fill=black!20] (1,0.67)--(3,2)--(0,1.2)--(1,0.67);
			\draw[fill=black!20] (3,2)--(2,2.53)--(0,1.2)--(3,2);
			\draw[fill=black!20] (3,2)--(4.2,4)--(2,2.53)--(3,2);
			\draw[fill=black!20] (3,2)--(4.8,4)--(4.2,4)--(3,2);
			\draw[fill=black!20] (4.2,4)--(4.8,4)--(4.8,4.5)--(4.2,4);
			\draw[fill=black!20] (4.2,4)--(4.8,4.5)--(4.2,5)--(4.2,4);
			\draw[fill=black!20] (6.05,5.03)--(7.5,5.1)--(6.2,4.43)--(6.05,5.03);
			\draw[fill=black!20] (6.05,5.03)--(7.5,5.1)--(6.6,6)--(6.05,5.03);
			\draw[fill=black!20] (1.87,5.1)--(4.2,6)--(3,6)--(1.87,5.1);
			\draw[fill=black!20] (1.87,5.1)--(4.2,6)--(4.2,5)--(1.87,5.1);
			\draw[fill=black!20] (1.87,5.1)--(4.2,5)--(0.87,4.3)--(1.87,5.1);
			\draw[fill=black!20] (0.87,4.3)--(4.2,4)--(4.2,5)--(0.87,4.3);
			\draw[fill=black!20] (0,3.6)--(4.2,4)--(0,2.4)--(0,3.6);
			\draw[fill=black!20] (0,2.4)--(4.2,4)--(0,1.2)--(0,2.4);
			\draw [line width=0.3mm, color=red](0,1.2) -- (2,2.53);
			\draw [line width=0.3mm, color=blue](5.4,-3) -- (6.6,-3);
			\draw [line width=0.3mm, color=red](4.2,4) -- (4.8,4.5);
			\node[inner sep=0,anchor=west,text width=3.3cm] (note1) at (5.5,4.5) {$a$};
			\draw[dotted]  [line width=0.5mm](9,2.4) -- (9,1.2);
			\draw[dotted]  [line width=0.5mm](9,1.2) -- (6.5,-1.5);
			\draw[dotted]  [line width=0.5mm](6.5,-1.5) -- (4.2,-2);
			\draw[dotted]  [line width=0.5mm](4.2,-2) -- (1,-1);
			\draw[dotted]  [line width=0.5mm](0.87,4.3) -- (4.2,4);
			\draw[dotted]  [line width=0.5mm](4.2,5) -- (4.8,4.5);
			\draw[dotted]  [line width=0.5mm](4.2,5) -- (4.2,6);
			\draw[dotted]  [line width=0.5mm](4.2,6) -- (6.6,6);
			\draw[dotted]  [line width=0.5mm](3,2) -- (4.8,4);
			\draw[dotted]  [line width=0.5mm](3,2) -- (2.5,1);
			\draw[dotted]  [line width=0.5mm](2.5,1) -- (4,0.5);
			\draw[dotted] [line width=0.5mm](4,0.5) -- (5,-1.5);
			\draw[dotted]  [line width=0.5mm](5,-1.5) -- (7,1);
			\draw[dotted]  [line width=0.5mm](7,1) -- (6.6,1.5);
			\draw[dotted]  [line width=0.5mm](6.6,1.5) -- (6,3);
			\draw[dotted]  [line width=0.5mm](2,2.53) -- (4.2,4);
			\draw[dotted]  [line width=0.5mm](0,3.6) -- (4.2,4);
			\draw[dotted]  [line width=0.5mm](4.8,4) -- (4.8,4.5);
			\draw (4.8,4) node[circle,fill,inner sep=1pt] {};
			\draw (4.8,4.5) node[circle,fill,inner sep=1pt] {};
			\draw (2,2.53) node[circle,fill,inner sep=1pt] {};
			\draw (0,1.2) -- (1,0.67);
			\draw (1,0.67) -- (1.2,0.48);
			\draw (1.2,0.48) -- (1.5,0.185);
			\draw (1.5,0.185) -- (1.5,-0.45);
			\draw (1.5,-0.45) -- (1,-1);
			\draw (8.1,4.5) -- (7.8,3.73);
			\draw (7.8,3.73) -- (7.8,3.36);
			\draw (7.8,3.36) -- (8,2.73);
			\draw (8,2.73) -- (8.2,2.55);
			\draw (8.2,2.55) -- (9,2.4);
			\draw[thick,->] (4.35,3.5) -- (4.7,3.6);
			\draw[thick,->] (4.5,4.75) -- (4.8,4.75);
			\draw[thick,->] (2.5,5.6) -- (2.8,5.6);
			\draw[thick,->] (3.17,5.6) -- (3.37,5.37);
			\draw[thick,->] (4.2,5.5) -- (4.5,5.5);
			\draw[thick,->] (4.2,5) -- (4.459,5.23);
			\draw[thick,->] (4.8,5) -- (4.8,5.3);
			\draw[thick,->] (5.4,5.5) -- (5.7,5.5);
			\draw[thick,->] (5.4,5) -- (5.659,5.23);
			\draw[thick,->] (6.45,5.05) -- (6.45,5.35);
			\draw[thick,->] (3.2,5.045) -- (2.9,4.9);
			\draw[thick,->] (2.7,4.68) -- (2.8,4.38);
			\draw[thick,->] (2.69,4.138) -- (2.33,4.03);
			\draw[thick,->] (4.2,4) -- (3.7,3.95);
			\draw[thick,->] (0,3) -- (0.3,3);
			\draw[thick,->] (1.9,3.125) -- (2.2,2.99);
			\draw[thick,->] (1.9,1.7) -- (2.1,2.05);
			\draw[thick,->] (1.5,1) -- (1.3,1.2);
			\draw[thick,->] (1.9,0.75) -- (1.7,1);
			\draw[thick,->] (2.35,0.295) -- (2.35,0.599);
			\draw[thick,->] (3,-0.9) -- (3.05,-0.5);
			\draw[thick,->] (2.2,-1.15) -- (2.2,-0.8);
			\draw[thick,->] (2.6,-1.5) -- (2.4,-1.8);
			\draw[thick,->] (3.05,-2) -- (3.05,-2.3);
			\draw[thick,->] (3.6,-2.5)-- (3.8,-2.75);
			\draw[thick,->] (4.2,-2.5) -- (4.5,-2.5);
			\draw[thick,->] (4.2,-3) -- (4.5,-3);
			\draw[thick,->] (5.4,-2.75) -- (5.1,-2.75);
			\draw[thick,->] (6.93,-2.6) -- (6.6,-2.6);
			\draw[thick,->] (6.3,-2.55) -- (6,-2.5);
			\draw[thick,->] (6.5,-2.22) -- (6.7,-1.95);
			\draw[thick,->] (6.7,-1.8) -- (6.5,-1.6);
			\draw[thick,->] (7.4,-1.23) -- (7.6,-1);
			\draw[thick,->] (8,-0.6) -- (8.1,-0.3);
			\draw[thick,->] (9,1.2) -- (8.7,0.88);
			\draw[thick,->] (6.6,-0.42) -- (6.6,-0.72);
			\draw[thick,->] (7,0.45) -- (6.9,0.1);
			\draw[thick,->] (7.8,1.56) -- (7.65,1.25);
			\draw[thick,->] (7,1) -- (6.77,0.7);
			\draw[thick,->] (6.85,1.2) -- (6.55,0.9);
			\draw[thick,->] (8,2.3) -- (7.9,1.9);
			\draw[thick,->] (7.8,2.555) -- (7.7,2.2);
			\draw[thick,->] (7.5,3.3) -- (7.4,3);
			\draw[thick,->] (7.3,3.8) -- (7.2,3.5);
			\draw[thick,->] (7.4,4.38) -- (7.3,4.08);
			\draw[thick,->] (6.4,4.53) -- (6.3,4.8);
			\draw[thick,->] (5.4,4) -- (5.4,4.35);
			\draw[thick,->] (4.2,4.5) -- (4.55,4.5);
			\draw[thick,->] (3.6,3) -- (3.95,3.1);
			\draw[thick,->] (3,2) -- (3.3,2.25);
			\draw[thick,->] (2.75,1.5) -- (3.1,1.8);
			\draw[thick,->] (2.5,1) -- (2.83,1.34);
			\draw[thick,->] (3.1,0.8) -- (3.4,1);
			\draw[thick,->] (3.6,1.5) -- (3.8,1.79);
			\draw[thick,->] (4.7,2) -- (4.82,2.32);
			\draw[thick,->] (4,0.5) -- (4.32,0.64);
			\draw[thick,->] (4.33,1.2) -- (4.6,1.3);
			\draw[thick,->] (4.9,1.6) -- (5,1.9);
			\draw[thick,->] (5.2,1) -- (5.5,1.11);
			\draw[thick,->] (5.6,2.02) -- (5.9,1.9);
			\draw[thick,->] (5.3,2.5) -- (5.55,2.7);
			\draw[thick,->] (6.6,1.5) -- (6.48,1.8);
			\draw[thick,->] (6,3) -- (5.82,3.3);
			\draw[thick,->] (5.9,0.2) -- (5.6,0.3);
			\draw[thick,->] (5.13,0.05) -- (4.85,0.2);
			\draw[thick,->] (5,-1.5) -- (4.815,-1.13);
			\draw[thick,->] (6.5,-1.5) -- (6.23,-1.75);
			\draw[thick,->] (5.4,-1.74) -- (5.4,-2.05);
			\draw[thick,->] (5.75,-1.503) -- (5.42,-1.65);
			\draw[thick,->] (4.61,-1.75) -- (4.3,-1.65);
			\draw[thick,->] (2,-2) -- (1.75,-1.75);
			\draw[thick,->] (1,-1) -- (0.75,-0.75);
			\draw[thick,->] (4.2,6) -- (4.55,6);
			\draw[thick,->] (5.4,6) -- (5.75,6);
			\draw[thick,->] (7.5,5.1) -- (7.75,4.85);
			\draw[thick,->] (8.1,4.5) -- (8.35,4.25);
			\draw[thick,->] (9,3.6) -- (9,3.25);
			\draw[thick,->] (9,2.4) -- (9,2.05);
			\draw[thick,->] (8.19,-1) -- (8.41,-0.75);
			\draw[thick,->] (7.39,-2) -- (7.59,-1.75);
			\draw[thick,->] (4.2,-2) -- (4.5,-2.12);
			\draw[thick,->] (2,2.53) -- (2.3,2.74);
			\draw[thick,->] (2.43,2.3) -- (2.73,2.5);
			\draw[thick,->] (4.8,4) -- (5.1,4);
			\draw[thick,->] (4.5,4) -- (4.5,4.2);
			\draw[thick,->] (4.8,4.25) -- (5.1,4.25);
			\draw[thick,->] (4.8,4.5) -- (5.1,4.74);
			\draw[thick,->] (1,0.67) -- (0.7,0.48);
			\draw[thick,->] (1.08,0.6) -- (0.8,0.4);
			\draw[thick,->] (1.2,0.48) -- (0.86,0.35);
			\draw[thick,->] (1.34,0.36) -- (1,0.25);
			\draw[thick,->] (1.5,0.185) -- (1.1,0.13);
			\draw[thick,->] (1.5,-0.13) -- (1.2,-0.13);
			\draw[thick,->] (1.5,-0.45) -- (1.1,-0.335);
			\draw[thick,->] (1.28,-0.7) -- (0.98,-0.55);
			\draw[thick,->] (0.58,0.9) -- (0.3,0.7);
			\draw[thick,->] (7.95,4.1) -- (8.25,4);
			\draw[thick,->] (7.8,3.73) -- (8.2,3.685);
			\draw[thick,->] (7.8,3.55) -- (8.15,3.55);
			\draw[thick,->] (7.8,3.36) -- (8.15,3.44);
			\draw[thick,->] (7.9,3.05) -- (8.2,3.2);
			\draw[thick,->] (8,2.73) -- (8.3,2.99);
			\draw[thick,->] (8.1,2.63) -- (8.38,2.89);
			\draw[thick,->] (8.2,2.55) -- (8.43,2.85);
			\draw[thick,->] (8.497,2.49) -- (8.65,2.75);
			\draw[dotted]  [line width=0.5mm](7.5,5.1) -- (8.1,4.5);
			\draw[dotted]  [line width=0.5mm](1,-1) -- (1.5,-0.45);
			\draw[dotted]  [line width=0.5mm](0,1.2) -- (0,2.4);
			\draw[dotted]  [line width=0.5mm](0,1.2) -- (1,0.67);
			\draw[dotted]  [line width=0.5mm](1,0.67) -- (1.2,0.48);
			\draw[dotted]  [line width=0.5mm](1.2,0.48) -- (1.5,0.185);
			\draw[dotted]  [line width=0.5mm](1.5,0.185) -- (1.5,-0.45);
			\draw[dotted]  [line width=0.5mm](8.1,4.5) -- (7.8,3.73);
			\draw[dotted]  [line width=0.5mm](7.8,3.73) -- (7.8,3.36);
			\draw[dotted]  [line width=0.5mm](7.8,3.36) -- (8,2.73);
			\draw[dotted]  [line width=0.5mm](8,2.73) -- (8.2,2.55);
			\draw[dotted]  [line width=0.5mm](8.2,2.55) -- (9,2.4);
			\draw (1,0.67) node[circle,fill,inner sep=1pt] {};
			\draw (7.8,3.73) node[circle,fill,inner sep=1pt] {};
			\draw (7.8,3.36) node[circle,fill,inner sep=1pt] {};
			\draw (8,2.73) node[circle,fill,inner sep=1pt] {};
			\draw (8.2,2.55) node[circle,fill,inner sep=1pt] {};
			\draw (1.2,0.48) node[circle,fill,inner sep=1pt] {};
			\draw (1.5,0.185) node[circle,fill,inner sep=1pt] {};
			\draw (1.5,-0.45) node[circle,fill,inner sep=1pt] {};
			\node[inner sep=0,anchor=west,text width=2.5cm] (note1) at (1.7,2.73) {$10$};
			\node[inner sep=0,anchor=west,text width=2.5cm] (note1) at (4.22,3.8) {$11$};
			\node[inner sep=0,anchor=west,text width=3.3cm] (note1) at (-0.3,0) {$3$};
			\node[inner sep=0,anchor=west,text width=3.3cm] (note1) at (-0.3,3.6) {$3$}; 
			\node[inner sep=0,anchor=west,text width=3.3cm] (note1) at (2.7,6.1) {$3$};
			\node[inner sep=0,anchor=west,text width=3.3cm] (note1) at (6.8,6.1) {$3$};
			\node[inner sep=0,anchor=west,text width=3.3cm] (note1) at (9.2,3.6) {$3$};
			\node[inner sep=0,anchor=west,text width=3.3cm] (note1) at (9.2,0) {$3$};
			\node[inner sep=0,anchor=west,text width=3.3cm] (note1) at (6.8,-3.1) {$3$};
			\node[inner sep=0,anchor=west,text width=3.3cm] (note1) at (2.7,-3.1) {$3$};
			\node[inner sep=0,anchor=west,text width=3.3cm] (note1) at (0.47,4.3) {$1$};
			\node[inner sep=0,anchor=west,text width=3.3cm] (note1) at (1.47,5.1) {$2$};
			\node[inner sep=0,anchor=west,text width=3.3cm] (note1) at (4.2,6.2) {$4$};
			\node[inner sep=0,anchor=west,text width=3.3cm] (note1) at (5.4,6.2) {$5$};
			\node[inner sep=0,anchor=west,text width=3.3cm] (note1) at (7.7,5.1) {$2$};
			\node[inner sep=0,anchor=west,text width=3.3cm] (note1) at (8.3,4.5) {$1$};
			\node[inner sep=0,anchor=west,text width=3.3cm] (note1) at (9.2,2.4) {$6$};
			\node[inner sep=0,anchor=west,text width=3.3cm] (note1) at (9.2,1.2) {$7$};
			\node[inner sep=0,anchor=west,text width=3.3cm] (note1) at (8.39,-1) {$8$};
			\node[inner sep=0,anchor=west,text width=3.3cm] (note1) at (7.59,-2) {$9$};
			\node[inner sep=0,anchor=west,text width=3.3cm] (note1) at (5.4,-3.2) {$7$};
			\node[inner sep=0,anchor=west,text width=3.3cm] (note1) at (4.2,-3.2) {$6$};
			\node[inner sep=0,anchor=west,text width=3.3cm] (note1) at (1.6,-2) {$9$};
			\node[inner sep=0,anchor=west,text width=3.3cm] (note1) at (0.6,-1) {$8$};
			\node[inner sep=0,anchor=west,text width=3.3cm] (note1) at (-0.3,1.2) {$5$};
			\node[inner sep=0,anchor=west,text width=3.3cm] (note1) at (-0.3,2.4) {$4$};
			\draw (6.6,3.87) node[circle,fill,inner sep=1pt] {};
			\draw (6.6,4.23) node[circle,fill,inner sep=1pt] {};
			\draw (6.2,4.43) node[circle,fill,inner sep=1pt] {};
			\draw (6.05,5.03) node[circle,fill,inner sep=1pt] {};
			\draw[dotted]  [line width=0.5mm](6,3) -- (6.6,3.87);
			\draw[dotted]  [line width=0.5mm](6.6,3.87) -- (6.6,4.23);
			\draw[dotted]  [line width=0.5mm](6.6,4.23) -- (6.2,4.43);
			\draw[dotted]  [line width=0.5mm](6.2,4.43) -- (6.05,5.03);
			\draw[dotted]  [line width=0.5mm](6.05,5.03) -- (6.6,6);
			\draw[thick,->] (6.05,5.03) -- (5.65,5.02);
			\draw[thick,->] (6.2,4.43) -- (5.9,4.27);
			\draw[thick,->] (6.6,4.23) -- (6.2,4.15);
			\draw[thick,->] (6.6,3.87) -- (6.2,3.897);
			\draw[thick,->] (6.3,5.45) -- (5.9,5.23);
			\draw[thick,->] (6.15,4.69) -- (5.8,4.63);
			\draw[thick,->] (6.35,4.33) -- (6.05,4.2);
			\draw[thick,->] (6.6,4.1) -- (6.3,4);
			\draw[thick,->] (6.35,3.5) -- (6,3.7);
			\end{tikzpicture}
		} %
		\caption{}
		\label{fig:torusflow3} 
	\end{figure}
	
\end{example}


\bibliographystyle{amsplain}

\begin{thebibliography}{10}

\bibitem{ayala}
R.~Ayala, D.~Fern\'{a}ndez- Ternero, J.~A.~Vilches, \emph{Perfect discrete Morse functions on $2$-complexes}, 
Pattern Recognition Letters \textbf{33} (2012) 1495--1500. 


\bibitem{afv2}
R.~Ayala, L.M.~Fern\'{a}ndez, J.~A.~Vilches, \emph{Characterizing equivalent discrete Morse functions}, 
Bull. Braz. Math. Soc. New Series \textbf{49(2)} (2009) 225--235. 


\bibitem{bruno}
B.~Benedetti, \emph{Discrete Morse theory for manifolds with boundary}, 
Trans. Amer. Math. Soc. \textbf{364} (2012) 6631--6670.

\bibitem{Benedetti}
B.~Benedetti, \emph{Smoothing discrete Morse theory}, 
Annali della Scuola Normale Superiore di Pisa Classe di Scienze, Serie $V$, Vol. $XVI$, Fasc. $2$ (2016), 335--368.



\bibitem{forman1}
R.~Forman, \emph{Morse theory for cell complexes}, 
Adv. Math. \textbf{134} (1998), 90--145.

\bibitem{forman2}
R.~Forman, \emph{A user's guide to discrete Morse theory}, 
Sem. Lothar. Combin. \textbf{48} (2002) B48c, 35 pp.

\bibitem{Gallais} 
E. Gallais, \emph{Combinatorial realizations of the Thom-Smale Complexes via Discrete Morse Theory}, 
Annali della Scuola Normale Superiore di Pisa, Classe di Scienze, \textbf{9 (2)} 2010, 229--252.

\bibitem{Mischaikow}
S.~Harker,K.~Mischaikow,M.~Mrozek,V.~Nanda, \emph{Discrete Morse Theoretic Algorithms for Computing Homology of Complexes and Maps}, 
Found. Comput. Math. \textbf{14(1)} (2014), pp. 151--184.

\bibitem{hersh}
P.~Hersh, \emph{On optimizing discrete Morse functions}, 
Adv. Appl. Math. \textbf{35(3)} (2005), pp. 294--322.


\bibitem{JP} 
M.~Joswig, M.~E.~Pfetsch, \emph{Computing Optimal Morse Matchings}, SIAM J. Discrete Math.
\textbf{20} (2006), pp. 11-25


\bibitem{HKN}
H.~King, K.~Knudson,N.~Mramor, \emph{Generating discrete Morse functions from point data}, Exp. math., \textbf{14(4)} (2005), pp. 435-444.


\bibitem{lewiner}
T.~Lewiner, H.~Lopes, G.~Tavares, \emph{Optimal discrete Morse functions for $2$-manifolds}, 
Comput. Geom. Math. \textbf{26(3)} (2003), pp. 221--233.

\bibitem{stallings}
J.~R.~Stallings, \emph{Lectures on Polyhedral Topology}, 
Tata Institute of Fundamental Research, Bombay, 1968.









\end{thebibliography}
\providecommand{\bysame}{\leavevmode\hbox
to3em{\hrulefill}\thinspace}
\providecommand{\MR}{\relax\ifhmode\unskip\space\fi MR }
 \MRhref  \MR
\providecommand{\MRhref}[2]{
  \href{http://www.ams.org/mathscinet-getitem?mr=#1}{#2}
 } \providecommand{\href}[2]{#2}

\end{document}